\documentclass[11pt]{amsart}   	
\usepackage[margin = 1in]{geometry}
									
\usepackage{amsmath,amssymb, amsfonts,amsthm}
\usepackage{tensor}
\usepackage{hyperref}
\usepackage{mathtools}
\usepackage{tikz}
\usepackage{tikz-cd}
\usetikzlibrary{decorations.markings}
\usepackage{caption}
\usepackage{enumitem}
\usepackage{makecell}
\setcellgapes{4pt}

\newtheorem{theorem}{Theorem}[section]
\newtheorem{definition-theorem}[theorem]{Definition-Theorem}
\newtheorem{lemma}[theorem]{Lemma}
\newtheorem{corollary}[theorem]{Corollary}
\newtheorem{proposition}[theorem]{Proposition}

\newtheorem{thmIntro}{Theorem}

\theoremstyle{definition}
\newtheorem{definition}[theorem]{Definition}
\newtheorem{example}[theorem]{Example}

\newtheorem{remark}[theorem]{Remark}
\newtheorem{notation}[theorem]{Notation}

\newcommand{\lperp}[1]{\prescript{\perp}{}{#1}}
\newcommand{\lperpW}[1]{\prescript{\prescript{}{\W}{\perp}}{}{#1}}
\newcommand{\lperpD}[1]{\prescript{\prescript{}{\mathcal{D}}{\perp}}{}{#1}}

\newcommand{\Filt}{\mathsf{Filt}}
\DeclareMathOperator{\Gen}{\mathsf{Gen}}
\DeclareMathOperator{\Cogen}{\mathsf{Cogen}}

\newcommand{\wide}{\mathsf{wide}}
\newcommand{\brick}{\mathsf{brick}}
\newcommand{\sbrick}{\mathsf{sbrick}}
\newcommand{\tperp}{\tau\text{-}\mathsf{perp}}
\newcommand{\tors}{\mathsf{tors}}
\newcommand{\torf}{\mathsf{torf}}
\newcommand{\tex}{\tau\text{-}\mathsf{ex}}
\newcommand{\tfo}{\mathsf{tfo}}
\newcommand{\stopc}{\mathsf{stop\text{-}chain}}
\newcommand{\mods}{\mathsf{mod}}
\newcommand{\cjrep}{\mathsf{cj\text{-}rep}}
\newcommand{\cmrep}{\mathsf{cm\text{-}rep}}
\newcommand{\Hom}{\mathrm{Hom}}
\newcommand{\Ext}{\mathrm{Ext}}
\newcommand{\im}{\mathrm{im}}

\newcommand{\J}{\mathcal J}

\newcommand{\rk}{\mathrm{rk}}

\newcommand{\pop}{\mathrm{pop}_\uparrow}
\newcommand{\popd}{\mathrm{pop}_\downarrow}
\newcommand{\core}{\mathrm{core}\text{-}\mathrm{lab}_\uparrow}
\newcommand{\cored}{\mathrm{core}\text{-}\mathrm{lab}_\downarrow}
\newcommand{\Core}{\mathrm{core}_\uparrow}
\newcommand{\Cored}{\mathrm{core}_\downarrow}
\newcommand{\clo}{\mathrm{clo}_\uparrow}
\newcommand{\clod}{\mathrm{clo}_\downarrow}
\newcommand{\jlab}{\mathrm{j\text{-}label}}
\newcommand{\mlab}{\mathrm{m\text{-}label}}
\newcommand{\brlab}{\mathrm{br\text{-}label}}
\newcommand{\lab}{\mathrm{label}}
\newcommand{\tr}{\tau\text{-}\mathrm{rigid}}
\newcommand{\atom}{\mathrm{atom}}
\newcommand{\coatom}{\mathrm{coatom}}

\DeclareMathOperator{\CJR}{\mathrm{CJR}}
\DeclareMathOperator{\CMR}{\mathrm{CMR}}

\newcommand{\T}{\mathcal{T}}
\newcommand{\F}{\mathcal{F}}
\newcommand{\W}{\mathcal{W}}
\renewcommand{\F}{\mathcal{F}}
\newcommand{\U}{\mathcal{U}}
\newcommand{\V}{\mathcal{V}}
\newcommand{\E}{\mathcal{E}}
\newcommand{\X}{\mathcal{X}}

\usepackage{soul}

\newcommand{\covers}{{\,\,\,\cdot\!\!\!\! >\,\,}}
\newcommand{\covered}{{\,\,<\!\!\!\!\cdot\,\,\,}}
\newcommand{\coveredtau}{{\,\,<\!\!\!\!\cdot\,_\tau\,\,}}
\newcommand{\join}{\vee}
\renewcommand{\Join}{\bigvee}
\newcommand{\meet}{\wedge}
\newcommand{\Meet}{\bigwedge}
\DeclareMathOperator{\cji}{\mathrm{c-jir}}
\DeclareMathOperator{\cmi}{\mathrm{c-mir}}

\title[Exceptional sequences and the poset topology of wide subcategories]{Exceptional sequences in semidistributive lattices and the poset topology of wide subcategories}
\author{Emily Barnard}
\address{Department of Mathematical Sciences, DePaul University, Chicago, IL 60604, USA}
\email{e.barnard@depaul.edu}

\author{Eric J. Hanson}
\address{LACIM, Universit\'e du Qu\'ebec \`a Montr\'eal and Universit\'e de Sherbrooke, Qu\'ebec, CANADA}
\email{hanson.eric@uqam.ca}



\date{23 September, 2022}

\begin{document}
\maketitle

\begin{abstract}
Let $\Lambda$ be a finite-dimensional algebra over a field $K$. We describe how
Buan and Marsh's $\tau$-exceptional sequences can be used to give a ``brick labeling'' of a certain poset of wide subcategories of finitely-generated $\Lambda$-modules. When $\Lambda$ is representation-directed, we prove that there exists a total order on the set of bricks which makes this into an EL-labeling. Motivated by the connection between classical exceptional sequences and noncrossing partitions, we then turn our attention towards the study of (well-separated) completely semidistributive lattices. Such lattices come equipped with a bijection between their completely join-irreducible and completely meet-irreducible elements, known as rowmotion or simply the ``$\kappa$-map''. Generalizing known results for finite semidistributive lattices, we show that the $\kappa$-map  determines exactly when a set of completely join-irreducible elements forms a ``canonical join representation''. A consequence is that the corresponding ``canonical join complex'' is a flag simplicial complex, as has been shown for finite semidistributive lattices and lattices of torsion classes of finite-dimensional algebras. Finally, in the case of lattices of torsion classes of finite-dimensional algebras, we demonstrate how Jasso's $\tau$-tilting reduction can be encoded using the $\kappa$-map. We use this to define \emph{$\kappa^d$-exceptional sequences} for finite semidistributive lattices. These are distinguished sequences of completely join-irreducible elements which we prove specialize to $\tau$-exceptional sequences in the algebra setting.
\end{abstract}

\setcounter{tocdepth}{1}
\tableofcontents

\section{Introduction}

In 2002, Fomin and Zelevinsky initiated the study of cluster algebras as a framework for understanding total positivity in semisimple groups \cite{FZ1}.
In the following year, they classified cluster algebras of finite type by finite Coxeter groups via Cartan matrices \cite{FZ2}, and they showed that the number of clusters is given by the Coxeter-Catalan number \cite{FZ_y}.
Moreover, they constructed a simplicial complex, called the \emph{cluster complex}, consisting of ``compatible'' almost positive roots of the corresponding Cartan type.
The original enumeration was type-by-type, and an open problem was to provide uniform bijections to other Coxeter-Catalan objects.
One well-studied candidate was the $W$-noncrossing partition lattice, which we define below.

Let $W$ be a finite Coxeter group, and denote the set of reflections in $W$ by $T$.
We note that, in particular, $T$ contains the set of simple reflections of $W$, and thus it generates $W$.
As a consequence, we can represent each element $w\in W$ as word $t_1t_2 \ldots t_k$ in $T$.
We say that $t_1t_2 \ldots t_k$ is \emph{reduced} if there does not exist another expression $t'_1t'_2\ldots t'_j$ for $w$ in terms of $T$ with $j<k$.
We define the \emph{absolute order} of $W$ by $x \le y$ if and only if any reduced $T$-word for $x$ occurs as a subword for some reduced $T$-word for $y$.
The $W$-noncrossing partitions lattice, denoted $\mathrm{NC(W,c)}$, is the maximal interval $[e, c]$ from the identity element $e$ to a Coxeter element $c$.
See \cite{Armstrong} for more detailed background.

The classical noncrossing partition lattice (of type A) was studied as poset in 1972 by Kreweras \cite{kreweras}, who also introduced a certain anti-isomorphism of noncrossing partitions called the \emph{Kreweras Complement}.
However, again, there was no uniform (type-free) proof that the $W$-noncrossing partition lattices were in fact lattice-posets. 
In 2007, Brady and Watt provided a uniform proof of this fact, and also constructed the cluster complex in terms of noncrossing partitions \cite{BradyWatt}.
Building on this construction, Athanasiadis, Brady, and Watt gave a uniform EL-labeling of the noncrossing partition lattice in terms of the reflections of $W$ \cite{ABW}.

Independently, Reading introduced a new Coxeter-Catalan family---$c$-Cambrian lattices \cite{cambrian}.
Each $c$-Cambrian lattice of type $W$ is a lattice quotient of the (right) weak order on $W$, and can be realized as a subposet of $c$-sortable elements.
Moreover, in \cite{cambrian2}, Reading constructed two uniformly defined bijections from the set of $c$-sortable elements (of type $W$): one to the facets of the cluster complex; and the other to the noncrossing partition lattice $\mathrm{NC(W,c)}$.
The latter bijection is essentially encoded by a certain minimal join-representation called the \emph{canonical join representation} (see Definition~\ref{def:canonical_join_rep}.
The elements of a canonical join representation are called \emph{join-irreducible} (see Definition~\ref{def: irreducible}).

The main goal of this paper is to extend both the EL-labeling of Athanasiadis, Brady, and Watt and to study the canonical join representations that are central to Reading's bijections in the context of representation theory.
Throughout our paper we take $\Lambda$ to be a finite dimensional algebra over an arbitrary field $K$.
When $\Lambda$ is hereditary of Dynkin type $W$, Ingalls and Thomas showed that the noncrossing partition lattice $\mathrm{NC(W,c)}$ is isomorphic to the lattice of wide subcategories of $\mods \Lambda$ \cite{IT}.
This was later generalized to all hereditary artin algebras in \cite{IS}.
In this context, chains in the lattice of wide subcategories are in bijection with \emph{exceptional sequences}. Using the \emph{$\tau$-exceptional sequences} of \cite{BM_exceptional}, this bijection was generalized to non-hereditary algebras \cite{BM_wide,BuH}.
Our first main result uses these bijections to extend the EL-labeling given by Athanasiadis, Brady, and Watt.
\begin{thmIntro}[Theorem~\ref{thm:mainA}]\label{thm:intro:mainA}
    Let $\Lambda$ be a representation-directed algebra. Then there exists a partial order on the bricks in $\mods\Lambda$ which makes the labeling of $\wide(\mods\Lambda)$ from $\tau$-exceptional sequences into an EL-labeling.
\end{thmIntro}

We note that Athanasiadis, Brady, and Watt's EL-labeling comes from a linear ordering of the reflections $T$ called a \emph{reflection order}.
When $\Lambda$ is finite type hereditary, the set of reflections of $W$ are in bijection with the set of bricks in $\mods\Lambda$.
Thus we obtain the original EL-labeling of $\mathrm{NC(W,c)}$ as a special case of our theorem.

Reading's constructions were also translated into the language of quiver representations by Ingalls and Thomas in \cite{IT}, who showed that if $\Lambda$ is representation-finite hereditary, then the lattice of torsion classes for $\mods \Lambda$ is isomorphic to a $c$-Cambrian lattice.
More recently, canonical join-representations in type A were modeled by certain noncrossing arc diagrams \cite{reading_arcs}, from which it was observed that a set of join-irreducible elements is a canonical join representation if and only if each pair of join-irreducible elements if a canonical join representation. 
We call this the \emph{flag property}.

In \cite{emily_canonical}, it was shown that a large class of finite lattices called \emph{semidistributive lattices} have the flag property (and indeed such lattices are characterized by this property).
The class of finite semidistributive lattices includes Reading's $c$-Cambrian lattices as well as the lattices of torsion classes of $\tau$-tilting finite algebras (see Section~\ref{sec:lattice}). The flag property was also shown to hold for (possibly infinite) lattices of torsion classes of arbitrary finite-dimensional algebras in \cite{BCZ}. This was done by establishing a correspondence between canonical join representations and collections of hom-orthogonal bricks, or \emph{semibricks}, in $\mods\Lambda$.
We now extend this result to a larger class of infinite semidistributive lattices.
Recall that an element $j\in L$ is \emph{completely join-irreducible} provided that $j$ covers precisely one element in $L$. 

\begin{thmIntro}[Corollary~\ref{cor:flag}]\label{thm:intro:mainB}
    Let $L$ be a well-separated completely semidistributive lattice.
    Then a collection of completely join-irreducible elements is a canonical join representation if and only if each pair of elements is a canonical join representation.
\end{thmIntro}

Completely semidistributive lattices come equipped with a pair of inverse bijections, which we call $\kappa$ and $\kappa^d$, from the set of completely join-irreducible elements of $L$ to the set of completely meet irreducible elements of $L$:
\begin{eqnarray*}
    \kappa(j) &=& \mlab[j_*,j] =  \max\{y \in L \mid j \meet y = j_*\}\\
    \kappa^d(m) &=& \jlab[m,m^*] = \min\{y \in L \mid m \join y = m^*\}.
\end{eqnarray*}
(See Section~\ref{sec:lattice} for a more detailed explanation of these formulas.) We show that $\kappa$ gives a criterion for testing when join-irreducible elements are ``compatible'' in the sense that they are a canonical join representation. In particular, this criterion implies the flag property in Theorem~\ref{thm:intro:mainB}.

\begin{thmIntro}[Theorem~\ref{thm:covers_canonical}, simplified]\label{thm:intro:mainC}
Let $L$ be a well-separated completely semidistributive lattice.
 For any set $A$ of completely join-irreducible elements of $L$, the following are equivalent.
     \begin{enumerate}
      \item $\Join A$ is a canonical join representation.
     \item For all $i \neq j \in A$ one has $i \leq \kappa(j)$.
      \end{enumerate}
\end{thmIntro}

The maps $\kappa$ and $\kappa^d$ have recently played an important role in dynamical algebraic combinatorics and in Coxeter-Catalan combinatorics.
See e.g. \cite[Section~1.2]{DW} and the refences therein.

In \cite{BTZ}, explicit formulas for $\kappa$ are given in the case where $L$ is a lattice of torsion classes.
More recently, Enomoto used $\kappa$ to define a new partial order on the elements of a completely semidistributive lattice $L$, which he calls the $\kappa$-order \cite{enomoto}. (This order is also defined for a generalization of finite semidistributive lattices in \cite{DW}.) Enomoto further showed that if $L$ is the lattice of torsion classes of an abelian length category, then the $\kappa$-order is isomorphic to the lattice of wide subcategories.
For our last main result, we use the $\kappa$-map to give a combinatorialization of the $\tau$-exceptional sequences of $\tau$-tilting finite. We first explain how the ``extended $\kappa$-map'' $\overline{\kappa}^d$ is related to the Auslander-Reiten translation in Section~\ref{sec:combinatorial_tau_tilting}. In particular, we show that the $\tau$-tilting reduction of Jasso \cite{jasso} (see Section~\ref{sec:tau_exceptional}), and moreover the ``wide intervals'' of \cite{AP}, are related to the operator $\overline{\kappa^d}$. This motivates the definition of a \emph{$\kappa^d$-exceptional sequence}, given in Definition~\ref{def:kappa_exceptional} of Section~\ref{sec:kappa_exceptional}. Roughly speaking, while a $\tau$-exceptional sequence makes use of $\tau$-rigid modules to interatively move to smaller wide subcategories of a module category, a $\kappa^d$-exceptional sequence uses completely
join-irreducible elements to iteratively move to smaller ``nuclear intervals'' (Definition~\ref{def:nuclear}) of a finite semidistributive lattice. Our last main result shows that, up to the brick-$\tau$-rigid correspondence of \cite{DIJ}, the sets of $\kappa^d$-exceptional sequences and $\tau$-exceptional sequences coincide for $\tau$-tilting finite algebras.

\begin{thmIntro}[Theorem~\ref{thm:mainD}]\label{thm:intro:mainD}
    Let $\Lambda$ be a $\tau$-tilting finite algebra. Then there is a bijection between the set of $\kappa^d$-exceptional sequences in the lattice of torsion classes $\tors\Lambda$ and the set of $\tau$-exceptional sequences in $\mods\Lambda$. In particular, up to the correspondence between indecomposable $\tau$-rigid modules and completely join-irreducible torsion classes, the $\tau$-exceptional sequences of $\mods\Lambda$ are completely determined by the lattice of torsion classes.
\end{thmIntro}

We conclude the paper with detailed examples of $\kappa^d$-exceptional sequences and a brief section on our planned future work.


\subsection*{Acknowledgements}
EJH is grateful to Erlend D. B{\o}rve, Aslak Bakke Buan, H{\aa}vard Utne Terland, and Hugh Thomas for several insightful conversations related to this work. The authors are also thankful to Haruhisa Enomoto for helpful discussions.

\section{Semidistributive lattices and the kappa map}\label{sec:lattice}

In this section we review the fundamental definitions for posets and lattice-posets.
A key concept we use throughout is the notion of an \emph{edge-labeling} for a poset $P$.
Recall that an element $y$ \emph{covers} $x$ in $P$ if and only if $y>x$ and there does not exist $z\in P$ with $y>z>x$.
We write $y\covers x$, and we also say that $x$ is \emph{covered by} $y$, and the pair form a \emph{cover relation}. We use the \emph{Hasse quiver} of $P$ as a visual representation. This is the directed graph $\mathrm{Hasse}(P)$ whose vertices correspond to elements of $P$ such that there is an arrow $x \rightarrow y$ whenever $y \covered x$. We will use the poset whose Hasse quiver is shown in Figure~\ref{fig:run_ex} as a running example. The elements of this poset are drawn as circled nodes. The labeling of each arrow in the Hasse quiver with an elements $j_i$ is an example of the following, where the partial order on $\{j_1,j_2,j_3,j_4\}$ is inherited from $P$. 

   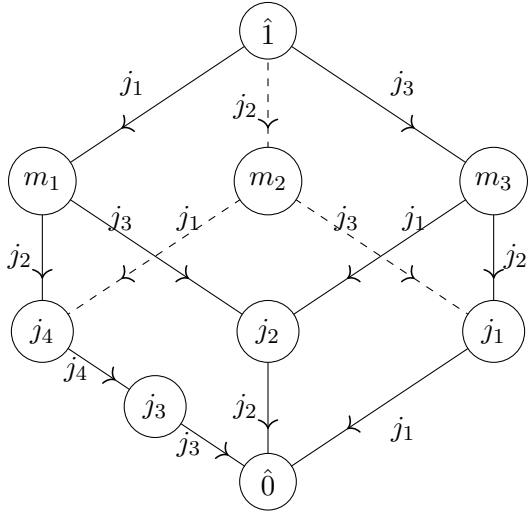
\begin{figure}
   \begin{tikzpicture}
        \begin{scope}[decoration={
	markings,
	mark=at position 0.65 with {\arrow[scale=1.5]{>}}}
	]
        \draw[postaction=decorate] (0,6) -- (-3,4) node [midway,above left] {$j_1$};
        \draw[postaction=decorate,dashed](0,6) -- (0,4) node [midway,left] {$j_2$};
        \draw[postaction=decorate](0,6) -- (3,4) node [midway,above right] {$j_3$};
        \draw[postaction=decorate](-3,4) -- (0,2) node [near start,right] {$j_3$};
        \draw[postaction=decorate,dashed](0,4) -- (-3,2) node [near start,left] {$j_1$};
        \draw[postaction=decorate](3,4) -- (0,2) node [near start, left] {$j_1$};
        \draw[postaction=decorate,dashed](0,4) -- (3,2) node [near start,right] {$j_3$};
        \draw[postaction=decorate](-3,4) -- (-3,2) node [midway, left] {$j_2$};
        \draw[postaction=decorate](3,4) -- (3,2) node [midway, right] {$j_2$};
        \draw[postaction=decorate](-3,2) -- (-1.5,1) node [midway,left] {$j_4$};
        \draw[postaction=decorate](-1.5,1) -- (0,0) node [midway,left] {$j_3$};
        \draw[postaction=decorate](0,2) -- (0,0) node [midway,left] {$j_2$};
        \draw[postaction=decorate](3,2) -- (0,0) node [midway,below right] {$j_1$};
        \end{scope}
		\node[draw,circle,fill=white] at (0,0) {$\hat{0}$};
	    \node[draw,circle,fill=white] at (-3,2) {$j_4$};
	    \node[draw,circle,fill=white] at (0,2) {$j_2$};
	    \node[draw,circle,fill=white] at (3,2) {$j_1$};
	    \node[draw,circle,fill=white] at (-1.5,1) {$j_3$};
	    \node[draw,circle,fill=white] at (3,4) {$m_3$};
	    \node[draw,circle,fill=white] at (0,4) {$m_2$};
	    \node[draw,circle,fill=white] at (-3,4) {$m_1$};
	    \node[draw,circle,fill=white] at (0,6) {$\hat{1}$};
	\end{tikzpicture}
    \caption{Our main running example of a lattice poset.}\label{fig:run_ex}
    \end{figure}
    
       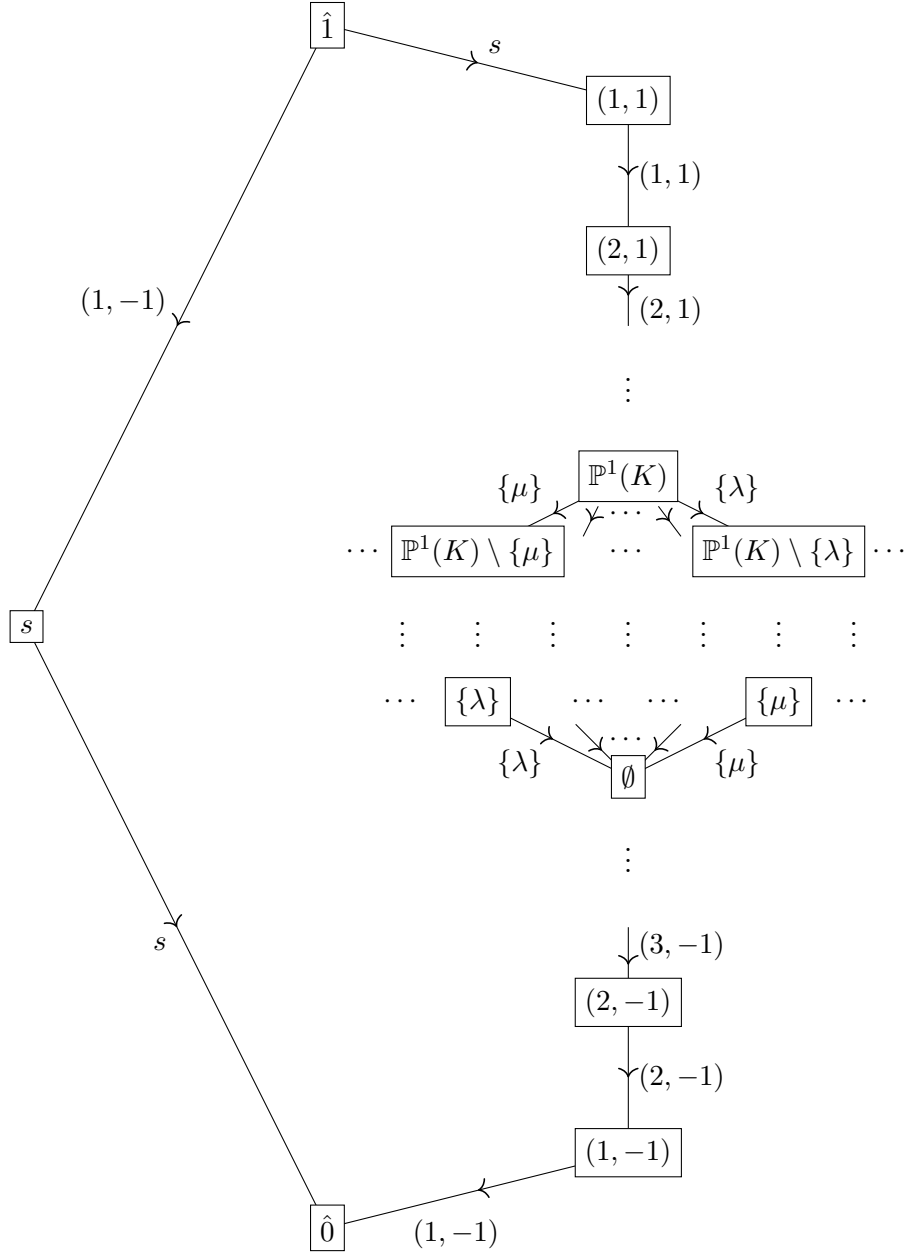
\begin{figure}
   \begin{tikzpicture}
        \begin{scope}[decoration={
	markings,
	mark=at position 0.5 with {\arrow[scale=1.5]{>}}}
	]
        \draw[postaction=decorate] (0,16) -- (-4,8) node [midway,above left] {$(1,-1)$};
        \draw[postaction=decorate] (-4,8) -- (0,0) node [midway,below left] {$s$};
        \draw[postaction=decorate] (4,1) -- (0,0) node [near end,below right] {$(1,-1)$};
        \draw[postaction=decorate] (4,3) -- (4,1) node [midway,right] {$(2,-1)$};
        \draw[postaction=decorate] (4,4) -- (4,3) node [near start,right] {$(3,-1)$};
        \draw[postaction=decorate] (2,7) -- (4,6) node [midway,below left] {$\{\lambda\}$};
        \draw[postaction=decorate] (6,7) -- (4,6) node [midway,below right] {$\{\mu\}$};
        \draw[postaction=decorate] (4.7,6.7) -- (4,6);
        \draw[postaction=decorate] (3.3,6.7) -- (4,6);
        \draw[postaction=decorate] (4,10)--(2,9) node [midway,above left] {$\{\mu\}$};
        \draw[postaction=decorate] (4,10) -- (6,9) node [midway,above right] {$\{\lambda\}$};
        \draw[postaction=decorate] (4.4,9.6)--(4.7,9.2);
        \draw[postaction=decorate] (3.6,9.6)--(3.4,9.2);
        \draw[postaction=decorate] (4,12.8) -- (4,12) node [near end, right] {$(2,1)$};
        \draw[postaction=decorate] (4,15) -- (4,13) node [midway,right] {$(1,1)$};
        \draw[postaction=decorate] (0,16) -- (4,15) node [midway,above right] {$s$};
        \end{scope}
		\node[draw,fill=white] at (0,0) {$\hat{0}$};
	    \node[draw,fill=white] at (0,16) {$\hat{1}$};
	    \node[draw,fill=white] at (4,1) {$(1,-1)$};
	    \node[draw,fill=white] at (4,3) {$(2,-1)$};
	    \node[draw,fill=white] at (4,15) {$(1,1)$};
	    \node[draw,fill=white] at (4,13) {$(2,1)$};
	    \node[draw,fill=white] at (4,6) {$\emptyset$};
	    \node[draw,fill=white] at (2,7) {$\{\lambda\}$};
	    \node[draw,fill=white] at (6,7) {$\{\mu\}$};
	    \node[draw,fill=white] at (2,9) {$\mathbb{P}^1(K) \setminus\{\mu\}$};
	    \node[draw,fill=white] at (6,9) {$\mathbb{P}^1(K)
	    \setminus\{\lambda\}$};
	    \node[draw,fill=white] at (4,10) {$\mathbb{P}^1(K)
	    $};
	    \node[draw,fill=white] at (-4,8) {$s$};
	    \node at (4,5) {$\vdots$};
	    \node at (4,11.25) {$\vdots$};
	    \node at (4,8) {$\vdots$};
	    \node at (3,8) {$\vdots$};
	    \node at (2,8) {$\vdots$};
	    \node at (1,8) {$\vdots$};
	    \node at (5,8) {$\vdots$};
	    \node at (6,8) {$\vdots$};
	    \node at (7,8) {$\vdots$};
	    \node at (3.5,7) {$\cdots$};
	    \node at (4.5,7) {$\cdots$};
	    \node at (1,7) {$\cdots$};
	    \node at (7,7) {$\cdots$};
	    \node at (4,9) {$\cdots$};
	    \node at (0.5,9) {$\cdots$};
	    \node at (7.5,9) {$\cdots$};
	    \node at (4,6.5) {$\cdots$};
	    \node at (4,9.5) {$\cdots$};
	\end{tikzpicture}
    \caption{The Hasse quiver of the lattice described in Example~\ref{ex:lattices}(2).}\label{fig:kronecker}
    \end{figure}

\begin{definition}\label{def:edge-label}
Let $P$ be a poset and let $\E(P) = \{[x,y]:\,x\covered y \text{ in $P$}\}$.
An \emph{edge-labeling} of $P$ is a function $\lab: \E(P) \to Q$, where $Q$ is a poset.
\end{definition}

Recall that a \emph{lattice-poset} $L = (L,\leq)$, or simply a lattice, is a poset satisfying the following two conditions:
First, each pair of elements $x$ and $y$ in $L$ has a unique smallest common upper bound called the \emph{join}, and denoted $x\join y$ or $\Join \{x,y\}$;
second, each pair $x$ and $y$ in $L$ has a unique greatest common lower bound called the \emph{meet}, and denoted $x\meet y$ or $\Meet \{x,y\}$.
Observe that for any finite subset $X\subseteq L$ the join $\Join X$ and the meet $\Meet X$ both exist. If an addition the join $\Join X$ and meet $\Meet X$ exist for \emph{any} subset $X \subseteq L$, then $L$ is called a \emph{complete} lattice. Unless otherwise stated, all of the lattices in this paper will be complete. In particular, each of our lattices has a unique smallest element $\hat{0}$ and a unique largest element $\hat{1}$.
We adopt the convention that $\Join \emptyset = \hat{0}$, and $\Meet \emptyset = \hat{1}$.

\begin{example}\label{ex:lattices}\
    \begin{enumerate}
        \item Our main running example of a lattice is the poset shown in Figure~\ref{fig:run_ex}. For example, in this poset we have $j_3 \join j_2 = m_1$ and $j_3 \meet j_2 = \hat{0}$. The labels on the Hasse arrows of Figure~\ref{fig:run_ex} are explained in Example~\ref{ex:label}.
        \item Let $K$ be an algebraically closed field. We define a lattice $L_{kr}(K)$ as follows. As a set, we have $L_{kr}(K) = \mathbb{N} \times \{-1,1\} \cup 2^{\mathbb{P}^1(K)} \cup \{\hat{0},\hat{1},s\}$, where $\mathbb{P}^1(K) = K \cup \{\infty\}$ is the projective line over $K$ and $\hat{0}, \hat{1}$, and $s$ are formal symbols. The relation $\leq$ is defined as follows.
        \begin{enumerate}
            \item For all $x \in L_{kr}(K)$: $\hat{0} \leq x \leq \hat{1}$.
            \item For all $n, m \in \mathbb{N}$: $(n, -1) \leq (m, -1)$ if and only if $n \leq m$ in the usual sense.
            \item For all $n, m \in \mathbb{N}$: $(n, -1) \leq (m, -1)$ if and only if $n \geq m$ in the usual sense.
            \item For all $S, S' \in 2^{\mathbb{P}^1(K)}$: $S \leq S'$ if and only if $S \subseteq S'$.
            \item For all $n \in \mathbb{N}$ and $S \in 2^{\mathbb{P}^1(K)}$: $(n,-1) \leq S \leq (n,1)$.
        \end{enumerate}
        The Hasse diagram of $L_{kr}(K)$ is shown in Figure~\ref{fig:kronecker} (see also \cite[Example~1.3]{thomas_intro}). The symbols $\lambda$ and $\mu$ denote generic elements of $\mathbb{P}^1(K)$. The labels on the Hasse arrows are explained in Example~\ref{ex:label}. Note that the interval $[\emptyset,\mathbb{P}^1(K)]$ is isomorphic to the boolean lattice on $\mathbb{P}^1(K)$, and that there are no cover relations of the form $u \covered \emptyset$ or $\mathbb{P}^1(K) \covered u$ in $L_{kr}(K)$.
        \item Two other examples of lattices which will feature in this paper are the lattices of torsion classes and of wide subcategories of a finite-dimensional algebra. In particular, $L_{kr}(K)$ is isomorphic to the lattice of torsion classes of the path algebra of the Kronecker quiver, see Example~\ref{ex:kronecker}. These lattices will be formally introduced in Section~\ref{sec:background}.
    \end{enumerate}
\end{example}

As we discuss in Section~\ref{sec:torsion}, the lattice of torsion classes also has a property called \emph{complete semidistributivity}, in which certain elements called \emph{completely join-irreducible} and \emph{completely meet-irreducible} play a role analogous to prime numbers in Number theory.

\begin{definition}\label{def: irreducible}
An element $j$ in a lattice $L$ is \emph{join-irreducible} provided that whenever $j=\Join X$ for a finite subset $X\subseteq L$ we have $j\in X$.
We say that $j$ is \emph{completely join-irreducible} provided that for any subset $X\subseteq L$ one has that $j = \Join X$ if and only if $j \in X$.
Equivalently, $j$ is completely join-irreducible if and only if there is a unique element $j_*$ such that $j\covers j_*$.
The notions of meet-irreducible and completely meet-irreducible are defined dually, by substituting ``$\Meet$'' for ``$\Join$''.
An element $m$ is completely meet-irreducible if and only if there is a unique element $m^*$ such that $m\covered m^*$.
\end{definition}

\begin{notation}
    We denote by $\cji(L)$ and $\cmi(L)$ the sets of completely join-irreducible and completely meet-irreducible elements, respectively.
\end{notation}

\begin{example}\label{ex:jirr}\
    \begin{enumerate}
        \item Note that the convention that $\Join \emptyset = \hat{0}$ means that $\hat{0}$ is not completely join-irreducible, even though it cannot be written as a join of strictly smaller elements. Similarly $\hat{1}$ is never completely meet irreducible.
        \item Let $L$ be the lattice in Figure~\ref{fig:run_ex}. Then $$\cji(L) = \{j_1,j_2,j_3,j_4\},\qquad\qquad\cmi(L) = \{m_1,m_2,m_3,j_3\}.$$
        \item Let $L_{kr}(K)$ be the lattice in Figure~\ref{fig:kronecker}. Then
        \begin{eqnarray*}
            \cji(L_{kr}(K)) &=& \{s\} \cup \mathbb{N} \times \{-1,1\} \cup \{\{\lambda\} \mid  \lambda \in \mathbb{P}^1(K)\},\\
            \cji(L_{kr}(K)) &=& \{s\} \cup \mathbb{N} \times \{-1,1\} \cup \{\mathbb{P}^1(K) \setminus \{\lambda\} \mid  \lambda \in \mathbb{P}^1(K)\}.
        \end{eqnarray*}
        Moreover, the element $\emptyset$ (resp. $\mathbb{P}^1(K)$) is join-irreducible (resp. meet-irreducible), but not completely join-irreducible (resp. completely meet-irreducible). Indeed, we have $\emptyset = \bigvee\{(n,-1)\mid n \in \mathbb{N}\}$, but there is no finite subset $X \subseteq L_{kr}(K) \setminus\{\emptyset\}$ such that $\emptyset = \bigvee X$. 
    \end{enumerate}
\end{example}

A lattice $L$ is \emph{completely semidistributive} if for any subset $Y\subseteq L$, and elements $x$ and $w$ in $L$, both of the following implications are true. (Note that by $x \join Y = w$ we mean $x \join y = w$ for all $y \in Y$, and likewise for $x \meet Y = w$.)
\begin{equation}\label{jsd}
\text{If $x\join Y = w$, then $x\join\left(\Meet Y\right) = w$}\tag{$SD_\join$}
\end{equation}
\begin{equation}\label{msd}
\text{If $x\meet Y = w$, then $x\meet\left(\Join Y\right) = w$}\tag{$SD_\meet$}
\end{equation}

Some important examples of completely semidistributive lattices are lattice of torsion classes of finite-dimensional algebras (see \cite[Theorem~4.5]{GM} and \cite[Theorem~1.3]{DIRRT}) and posets of regions of simplicial hyperplane arrangements (see e.g. \cite[Corollary~9-3.9]{reading_book}).

\begin{example}
    The lattices from Figures~\ref{fig:run_ex} and~\ref{fig:kronecker} are completely semidistributive. For example, in the lattice from Figure~\ref{fig:run_ex}, we have $m_1 \join \{m_2,m_3\} = \hat{1}$ and $m_1 \join (m_2 \meet m_3) = m_1 \join j_1 = \hat{1}$.
\end{example}

As we discuss in Section~\ref{sec:bricks}, cover relations in the lattice of torsion classes have a natural edge-labeling by certain indecomposable modules called bricks. 
This is actually a special case of a more general phenomenon which we recall now.

\begin{definition-theorem}\label{def:j-label}
    Let $L$ be a completely semidistributive lattice and let $u \covered v$ be a cover relation in $L$. By \cite[Lemma~3.11]{RST}, the set $\{y \in L \mid y \join u = v\}$ contains a minimum element $j$ which is completely join-irreducible and satisfies $j_* \leq u$. We call $j$ the \emph{join-irreducible label} of the cover relation $u \covered v$ and denote $j = \jlab[u,v]$.
\end{definition-theorem}

\begin{example}\label{ex:label}
    The cover relations in Figures~\ref{fig:run_ex} and~\ref{fig:kronecker} are decorated by their join-irreducible labels.
\end{example}

\begin{remark}\label{rem:m-label}\
    \begin{enumerate}
        \item Note that $\cji(L)$ inherits a partial order from $L$. Thus the association $[u,v] \mapsto \jlab[u,v]$ can be seen as a special case of Definition~\ref{def:edge-label}.
        \item One can also define a labeling $\mlab$ of cover relations by completely meet-irreducible elements using the dual of Definition-Theorem~\ref{def:j-label}. We note that $\mlab[u,v] = \jlab^d[v,u]$; that is the meet-irreducible label of the cover relation $[u,v]$ is the same as the join-irreducible label of the cover relation $[v,u]$ in the dual lattice $L^d$.
    \end{enumerate}
\end{remark}

\begin{notation}
    We extend Definition-Theorem~\ref{def:j-label} as follows. Given a relation $x \leq y \in L$, we denote
    $$\jlab[x,y] = \left\{\jlab[u,v] \mid x \leq u \covered v \leq y\right\}.$$
    That is, $\jlab[x,y]$ is the set of completely join-irreducible elements which label some cover relation in the interval $[x,y]$. Note that we have made a slight abuse of notation in that if the cover relation $u \covered v$ is labeled by $j$, then we use $\jlab[u,v]$ to denote both $j$ and $\{j\}$.
\end{notation}

Restricting to cover relations involving completely join- and meet-irreducible elements, Definition-Theorem~\ref{def:j-label} and its dual allow one to construct inverse bijections $\kappa: \cji(L) \rightarrow \cmi(L)$ and $\kappa^d: \cmi(L) \rightarrow \cji(L)$ such that
\begin{eqnarray}
    \kappa(j) &=& \mlab[j_*,j] =  \max\{y \in L \mid j \meet y = j_*\}\label{eqn:kappa}\\
    \kappa^d(m) &=& \jlab[m,m^*] = \min\{y \in L \mid m \join y = m^*\}.\label{eqn:kappad}
\end{eqnarray}
The bijections $\kappa$ and $\kappa^d$ are sometimes referred to as either \emph{rowmotion} or the \emph{$\kappa$-maps} of $L$. See e.g. \cite[Theorem~9.3]{thomas_intro} for a proof that these maps are indeed inverse bijections.

\begin{example}\label{ex:kappa}
    \begin{enumerate}
        \item Let $L$ be the lattice in Figure~\ref{fig:run_ex}. Then $\kappa(j_4) = j_3$ and $\kappa(j_i) = m_i$ for $i \in \{1,2,3\}$.
        \item Let $L_{kr}(K)$ be as in Figure~\ref{fig:kronecker}. Then $\kappa(s) = (1,1)$, $\kappa(1, -1) = s$, $\kappa(n, -1) = (n-1, -1)$ for $n > 1$, $\kappa(n, 1) = (n+1, 1)$ for $n \in \mathbb{N}$, and $\kappa(\{\lambda\}) = \mathbb{P}^1(K) \setminus \{\lambda\}$ for $\lambda \in \mathbb{P}^1(K)$.
    \end{enumerate}
\end{example}

\begin{remark}
When $L$ is finite, the existence of the bijections $\kappa$ and $\kappa^d$ is equivalent to (complete) semidistributivity, see \cite[Theorem~2.28]{RST}. On the other hand, there exist infinite lattices which are not completely semidistributive for which these bijections are still well-defined. See \cite[Example~3.18]{RST} for an example.
\end{remark}

The following, most of which is contained in \cite{enomoto}, will be useful in computing the join-irreducible labeling.

\begin{proposition}\label{prop:cover}
    Let $L$ be a completely semidistributive lattice and let $x \leq y \in L$. Then the following hold.
    \begin{enumerate}
        \item $\jlab[x,y] = \{j \in \cji(L) \mid j \leq y \text{ and }\kappa(j) \geq x\}$.
        \item Suppose $x \leq \kappa(j)$. Then there is a cover relation $(x \join j) \meet \kappa(j) \covered x \join j$ with $\jlab[(x \join j) \meet \kappa(j), x \join j] = j$.
        \item Suppose $x \covered y$. Then the following are equivalent.
        \begin{enumerate}
            \item $\jlab[x,y] = j$.
            \item $\mlab[x,y] = \kappa(j)$.
            \item $x \join j = y$ and $x \meet j = j_*$.
            \item $y \meet \kappa(j) = x$ and $y \join \kappa(j) = \kappa(j)^*$.
            \item $x \join j = y$ and $y \meet \kappa(j) = x$.
        \end{enumerate}
    \end{enumerate}
\end{proposition}

\begin{proof}
    Items (1) and (2) are Theorem~3.14 and Lemma~3.15 in \cite{enomoto}, respectively. The equivalence $(3a \iff 3c)$ is also \cite[Lemma~2.7]{enomoto}. This also implies the equivalence $(3b \iff 3d)$ by duality.
    
    $(3a \implies 3e)$: Suppose $\jlab[x,y] = j$. By (1), this means $j \leq y$ and $\kappa(j) \geq x$. Now by the definition of $\kappa$, we have $j \not\leq \kappa(j)$, and so $j \not \leq x$ and $\kappa(j) \not\geq y$. Since $x \covered y$ is a cover relation, it follows that $x \join j = y$ and $y \meet \kappa(j) = x$.
    
    $(3e \implies 3c)$: Suppose $x \join j = y$ and $x \meet \kappa(j) = x$. In particular, this means $j \leq y$ and $\kappa(j) \geq x$. Since $x \covered y$ is a cover relation, it follows from (1) that $\jlab[x,y] = j$.
    
    The proofs of $(3b \implies 3e)$ and $(3e \implies 3d)$ are dual to those above.
\end{proof}

To conclude this section, we recall the definitions of canonical join representations and the extended kappa-map of \cite{BTZ}.

\begin{definition}\label{def:canonical_join_rep}
    Let $L$ be a complete lattice and $x \in L$.
    \begin{enumerate}
        \item A \emph{join representation} of $x$ is an equation $x = \Join A$ where $A \subseteq L$. We say this join representation is \emph{irredundant} if for any $j \in A$ one has $x \neq \Join(A \setminus \{j\})$.
        \item Given two join representations $x = \Join A = \Join B$, we say that $\Join A$ \emph{refines} $\Join B$ if for every $a \in A$ there exists $b \in B$ such that $a \leq b$.
        \item A join representation $x = \Join A$ is called a \emph{canonical join representation} if $A$ is an antichain and $\Join A$ refines every join representation of $x$. In this case, we say that $x$ is \emph{canonically join-representable} and write $A = \CJR(x)$. We call the elements of $A$ the \emph{canonical joinands} of $x$ and say that $A$ \emph{joins canonically}.
        \item \emph{(Canonical) meet representations}, $\CMR(x)$, etc. are all defined dually.
    \end{enumerate}
\end{definition}

\begin{example}\label{ex:can_join_rep}
    Let $L$ be the lattice in Figure~\ref{fig:run_ex}. Then every element of $L$ has both a canonical join representation and a canonical meet representation. As some explicit examples:
    \begin{enumerate}
        \item Then there are exactly three irredundant join representations of $m_1$, namely $m_1 = \Join\{m_1\}$, $m_1 = \Join\{j_2,j_4\}$, and $m_1 = \Join\{j_2,j_3\}$. The canonical join representation of $m_1$ is $m_1 = j_2 \join j_3$. An example of a non-irredundant join representation of $m_1$ is $m_1 = \Join\{j_2,j_3,j_4\}$.
        \item The canonical meet representation of $\hat{0}$ is $\hat{0} = \Meet\{m_1,m_2,m_3\}$.
        \item Each $j_i \in \cji(L)$ has canonical join representation $j_i = \Join\{j_i\}$. Likewise each completely meet irreducible element is its own canonical meet representation.
        \item For all $x \in L$ it follows from \cite[Theorem~5.10]{RST} that  $$\CJR(x) = \{j \in \cjrep(L) \mid \exists u \covered x \text{ s.t. }j = \jlab[u,v] \}.$$
        This relies on the fact that $|L| < \infty$. We will discuss how this generalizes to the infinite case in Theorem~\ref{thm:covers_canonical}.
    \end{enumerate}
\end{example}

\begin{example}\label{ex:no_can_join}
   Let $L_{kr}(K)$ be as in Figure~\ref{fig:kronecker}. Then
   \begin{enumerate}
       \item $x := \emptyset$ does not have a canonical join representation. Indeed, let $x = \Join A$ be a join representation of $x$. If $x \in A$, let $B = \{(n,-1) \mid n \in \mathbb{B}\}$. Then $x = \Join B$ and $\Join A$ does not refine $\Join B$. Indeed, there does not exist $b \in B$ such that $x \leq b$. If, on the other hand, $x \notin A$, then there must be an infinite subset $S \subseteq \mathbb{N}$ such that $\{(n,-1) \mid n \in S\} \subseteq A$. In particular, this means $A$ is not an antichain. Note also that $\mathbb{P}^1(K)$ does not have a canonical meet representation for a similar reason.
       \item $\CJR(S) = \Join\{\{\lambda\} \mid \lambda \in S\}$ for all $\emptyset \neq S \in \mathbb{P}^1(K)$.
   \end{enumerate}
\end{example}

We emphasize that the join representation $\hat{0} = \Join\emptyset$ vacuously satisfies Definition~\ref{def:canonical_join_rep}(3). Likewise, we have a canonical meet representation $\hat{1} = \Meet\emptyset$.

\begin{notation}
    Let $L$ be a complete lattice. We denote by $\cjrep(L)$ (respectively $\cmrep(L)$) the subposet of canonically join-representable (resp. canonically meet-representable) elements.
\end{notation}

Before commenting on these definitions, we recall the following.

\begin{lemma}\cite[Lemma~4.5]{enomoto}\label{lem:canonical_join_rep}
    Suppose that $L$ is a complete lattice and that $x \in \cjrep(L)$. Then $\CJR(x) \subseteq \cji(L)$; that is, every canonical joinand of $x$ is completely join-irreducible.
\end{lemma}

\begin{remark}\label{rem:canonical_join_rep}
    The definition of a canonical join representation we have used in this paper agrees with those in e.g. \cite{emily_canonical,enomoto,gorbunov,reading_book,RST}. This differs slightly from the definition used in \cite{BCZ,BTZ}, where one assumes only that $\Join A$ refines every irredundant join representation of $x$. On the other hand, we will show in Theorem~\ref{thm:covers_canonical} that the definitions coincide under the additional assumption that every element of $A$ is completely join-irreducible. For example, in the setting of Example~\ref{ex:no_can_join}, the only irredundant join representation of $x$ is $\bigvee\{x\}$. This means $x$ has a canonical join representation as defined in \cite{BCZ,BTZ}. On the other hand, we showed in Example~\ref{ex:no_can_join} that $x$ does not have a canonical join representation as defined in the present paper. This is consistent with the fact that $x$ is not compeltely join-irreducible, and therefore $\Join\{x\}$ does not satisfy the hypotheses of our Theorem~\ref{thm:covers_canonical}. See \cite[Remark~4.3]{enomoto}, Theorem~\ref{thm:brick_labeling}, and Theorem~\ref{thm:covers_canonical} for additional discussion.
\end{remark}

In Section~\ref{sec:flag} we will explore the relationship between the existence of canonical join representations and that of cover relations. For now, we recall the following result.

\begin{proposition}\cite[Theorem~1]{gorbunov}\label{prop:covers_canonical_existence}
    Let $L$ be a completely semidistributive lattice and let $x \in L$. Then $x \in \cjrep(L)$ if and only if for all $y < x$ there exists a cover relation $y < z \covered x$. In particular, if $L$ is finite then $\cjrep(L) = L = \cmrep(L)$.
\end{proposition}

We are now prepared to state the final definition of this section, see \cite[Definition~1.0.3]{BTZ}.

\begin{definition}\label{def:kappa_bar}
    Let $L$ be a completely semidistributive lattice and let $x \in \cjrep(L)$. We define
    $$\overline{\kappa}(x) := \Meet \{\kappa(j) \mid j \in \CJR(x)\}.$$
    For $x \in \cmrep(L)$ we define $\overline{\kappa}^d(x)$ dually.
\end{definition}

The maps $\overline{\kappa}: \cjrep(L) \rightarrow L$ and $\overline{\kappa}^d: \cmrep(L) \rightarrow L$ are sometimes referred to as the \emph{extended kappa-maps}.

\begin{example}\label{ex:kappa_bar}
    Let $L$ be the lattice in Figure~\ref{fig:run_ex}. Written as a permutation in cycle notation, we then have
    $\overline{\kappa} = (\hat{0},\hat{1})(j_1,m_1)(j_2,m_2)(j_3,m_3,j_4).$
\end{example}

\begin{remark}
    In \cite[Question~4.13]{enomoto}, Enomoto asks whether the image of $\overline{\kappa}$ lies in $\cmrep(L)$ (as is the case for lattices of torsion classes, see \cite[Theorem~4.19]{enomoto} and \cite[Corollary~4.4.3]{BTZ}). When the answer to Enomoto's question is ``yes'', one can adapt the arguments of \cite[Section~4.1]{enomoto} to conclude that $\overline{\kappa}: \cjrep(L) \rightarrow \cmrep(L)$ is a bijection and moreover that for every $x \in \cjrep(L)$ there is an induced bijection $\kappa|_{\CJR(x)}: \CJR(x) \rightarrow \CMR(\overline{\kappa}(x))$. In Section~\ref{sec:kappa_order}, we show that this is indeed the case under the additional assumption that $L$ is \emph{well-separated}.
\end{remark}


\section{Wide subcategories, torsion classes, and $\tau$-rigid modules}\label{sec:background}

Let $\Lambda$ be a finite-dimensional basic algebra over a field $K$. We denote by $\mods\Lambda$ the category of finitely generated (left) $\Lambda$-modules. When we speak of a subcategory of $\mods\Lambda$, we will always mean a full subcategory which is closed under isomorphisms.

We recall that a subcategory $\W \subseteq \mods\Lambda$ is called \emph{wide} if it is closed under kernels, cokernels, and extensions. Equivalently, $\W$ is an exact-embedded abelian subcategory. Note that it is an immediate consequence of the definitions that a subcategory $\U \subseteq \W$ is a wide subcategory of $\W$ if and only if it is a wide subcategory of $\mods \Lambda$.

Given a fixed wide subcategory $\W \subseteq \mods\Lambda$, we denote by $\wide(\W)$ the poset of wide subcategories of $\W$ under the inclusion order. It is well known that this poset is a (complete) lattice, with the meet operation being the intersection. By the previous paragraph, we can then identify $\wide(\W)$ with the interval $[0,\W] \subseteq \wide(\mods\Lambda)$.

In the remainder of this section, we recall background information about torsion classes and $\tau$-tilting theory. As we will see in the sequel, this paper will be concerned with the $\tau$-tilting theory \emph{within} certain wide subcategories. Thus our convention will be to work within a wide subcategory $\W$ of $\mods\Lambda$ throughout this section. To recover the original constructions, one can take $\W = \mods\Lambda$.

\begin{remark}
    Many of the wide subcategories in this paper satisfy an additional property known as \emph{functorial finiteness}. While the technical definition of functorial finiteness will not be needed in this paper, it is well-known that such subcategories are precisely those which are themselves equivalent to categories of finitely-generated modules over finite-dimensional algebras, see e.g. \cite[Proposition~4.12]{enomoto_ff}. This allows any statement which is known for such module categories to be extended to any functorially finite wide subcategory. More generally, every wide subcategory is an \emph{abelian length category}, and often this is enough to recover results which have been shown for module categories, see e.g. \cite[Section~2.2]{enomoto}. One notable exception is $\tau$-tilting theory (discussed in Section~\ref{sec:tau_tilting}), which is intimately related with the notion of projective objects, and therefore will require the assumption of functorial finiteness.
\end{remark}

\subsection{Torsion pairs}\label{sec:torsion}

We begin by recalling the definition and basic properties of torsion pairs. The material in this section is standard, and can be found for example in \cite[Section~VI.1]{ASS}.

Given two subcategories $\mathcal{C}, \mathcal{D} \subseteq \mods\Lambda$, we denote
\begin{eqnarray*}
    \mathcal{C}^{\perp_\mathcal{D}} &:=& \{X \in \mathcal{D} \mid \Hom_\Lambda(-,X)|_\mathcal{C} = 0\},\\
    \lperpD{\mathcal{C}} &:=& \{X \in \mathcal{D} \mid \Hom_\Lambda(X,-)|_\mathcal{C} = 0\}.
\end{eqnarray*}
That is, $\mathcal{C}^{\perp_\mathcal{D}}$ (resp. $\lperpD{\mathcal{C}}$) denotes the subcategory consisting of those objects in $\mathcal{D}$ which admit no nonzero morphisms from (resp. to) the objects in $\mathcal{C}$.

Now fix a wide subcategory $\W \subseteq \mods\Lambda$. A \emph{torsion pair in $\W$} is a pair $(\T, \F)$ of subcategories of $\W$ such that $\T^{\perp_\W} = \F$ and $\lperpW{\F} = \T$. The subcategory $\T$ is called a \emph{torsion class of $\W$} and the subcategory $\F$ is called a \emph{torsion-free class of $\W$}. 
Given an arbitrary subcategory $\mathcal{C}$, we write $\Gen\mathcal{C}$ (resp. $\Cogen \mathcal{C}$) for the subcategory of quotients (resp. submodules) of direct sums of modules in $\mathcal{C}$. More generally, if $\mathcal{C}$ is contained in some wide subcategory $\mathcal{W} \subseteq \mods\Lambda$, we denote $\Gen_\W\mathcal{C} := \W \cap \Gen\mathcal{C}$ and $\Cogen_\W\mathcal{C} := \W \cap \Cogen\mathcal{C}$.
It is well known that $\mathcal{C} \subseteq \W$, is a torsion class (resp. a torsion-free class) in $\W$ if and only if $\mathcal{C}=\Gen_\W\mathcal{C}$ (resp. $\mathcal{C} = \Cogen_\W(\mathcal{C})$)  and $\mathcal{C}$ is closed under extensions.

Given a torsion pair $(\mathcal{T}, \mathcal{F})$ in $\W$ and a module $M \in \W$, there exist unique modules $t_\T(M) \in \T$ and $f_\F(M) \in \F$ which fit into an exact sequence of the form 
\begin{equation}\label{eqn:canonical} 0 \rightarrow t_\T(M) \xrightarrow{\iota} M \xrightarrow{q} f_\F(M) \rightarrow 0. \end{equation}
This is called the \emph{canonical exact sequence} of $M$ associated to the torsion pair $(\T,\F)$.

We denote by $\tors(\mathcal{W})$ the poset of torsion classes of $\mathcal{W}$ with the inclusion order. As mentioned in Section~\ref{sec:lattice}, it is well known that $\tors(\mathcal{W})$ is a completely semidistributive lattice. We will discuss the semidistributivity property further in Section~\ref{sec:bricks}. 

\begin{example}\label{ex:kronecker}
   Consider the Kronecker quiver $Q = \left(\begin{tikzcd}1 \arrow[r,yshift = 0.1cm]\arrow[r,yshift = -0.1cm] & 2\end{tikzcd}\right)$ and the corresponding path algebra $\Lambda = KQ$. (For simplicity, we will assume that $K$ is algebraically closed when discussing this example.) Then $\tors(\mods\Lambda)$ is isomorphic to the lattice $L_{kr}(K)$ shown in Figure~\ref{fig:kronecker}. See \cite[Example~1.3]{thomas_intro} for a detailed description of this isomorphism.
\end{example}

The following will be useful throughout this paper. Note that given an arbitrary subcategory $\mathcal{C} \subseteq \mods\Lambda$, the notation $\Filt(\mathcal{C})$ refers to the subcategory consisting of those modules $M \in \mods\Lambda$ which admit a finite filtration
$$0 = M_0 \subsetneq M_1 \subsetneq \cdots \subsetneq M_k = M$$
such that $M_j/M_{j-1} \in \mathcal{C}$ for all $j$.

\begin{proposition}\cite[Proposition~2.1]{thomas_intro}
    Let $\mathcal{C} \subseteq \W$ be any subcategory. Then $\Filt(\Gen_\W\mathcal{C})$ is the smallest torsion class in $\W$ which contains $\mathcal{C}$. Dually, $\Filt(\Cogen_\W\mathcal{C})$ is the smallest torsion-free class in $\W$ which contains $\mathcal{C}$.
\end{proposition}

\subsection{$\tau$-rigid modules}\label{sec:tau_tilting}

In this section, we recall the basic notions of $\tau$-tilting theory as constructed in \cite{AIR}. In this section, $\W$ will denote a functorially finite wide subcategory of $\mods\Lambda$.

To make the notion of $\tau$-tilting theory in a wide subcategory precise, we fix the following notation.

\begin{notation}
    Fix a finite dimensional algebra $\Lambda_\W$ and inverse equivalences of categories $F: \W \rightarrow \mods\Lambda_\W$ and $G:\mods\Lambda_W \rightarrow \W$. Let $\tau_{\Lambda_\W}$ denote the Auslander-Reiten translate in $\mods\Lambda_\W$. Then the Auslander-Reiten translate in $\W$ is the assignment $\tau_\W: \W \rightarrow \W$ given by $\tau_\W(M) = G(\tau_{\Lambda_\W}(FM))$. We note that $\tau_\W$ is defined only on objects and is well-defined only up to isomorphism.
\end{notation}

We now recall the definition of a $\tau$-rigid module.

\begin{definition}
    Let $M \in \W$ be a basic module. Then $M$ is \emph{$\tau$-rigid} (in $\W$), in symbols $M \in \tr(\W)$, if $\Hom_\Lambda(M, \tau_\W M) = 0$. 
    Let $\rk(\W)$ denote the number of non-isomorphic modules which are simple in $\W$.
    If $M\in \tr(\W)$ and $\rk(M) = \rk(\W)$, then $M$ is \emph{$\tau$-tilting} (in $\W$).
\end{definition}

\begin{remark}
    Note that we have adopted the convention of \cite{BM_exceptional} in using the term ``$\tau$-rigid (in $\W$)'' rather than ``$\tau_\W$-rigid''. Nevertheless we caution that there may be modules $M \in \W$ which are $\tau$-rigid in $\W$, but not $\tau$-rigid in $\mods\Lambda$.
\end{remark}

We also fix the following definition for use in Section~\ref{sec:kappa_tau}.

\begin{notation}\label{not:tau_bar}
    Let $M \in \W$. Let $\tau_\W$ and $\nu_\W$ denote the Auslander-Reiten translation and Nakayama functor in $\W$. (Recall that $\nu_\W$ sends the indecomposable projective $P(i)$ to the corresponding indecomposable injective $I(i)$.) If $M$ is indecomposable, we denote
    $$\overline{\tau_\W}M = \begin{cases} \tau_\W M & M \text{ is not projective in }\W\\
    \nu_\W M & M \text{ is projective in }\W.\end{cases}$$
    We then extend this definition additively to define $\overline{\tau_\W} M$ when $M$ is not indecomposable.
\end{notation}

Before continuing, we recall the following characterization of Auslander and Smal\o, which will be useful in several of our proofs.
\begin{proposition}\cite[Proposition~5.8]{AS} \label{prop:AStau}
	Let $M, N \in \W$. Then $\Hom(N,\tau_\W M) = 0$ if and only if $\Ext^1_\Lambda(M,N') = 0$ for every $N' \in \Gen_\W(N)$.
\end{proposition}

For torsion classes, we take the following as our definition of functorial finiteness. See \cite[Sections~2.2-2.3]{AIR} for justification for doing so.

\begin{definition-theorem}\cite[Sections~2.2-2.3]{AIR}\label{thm:functorially_finite}
    Let $\T \subseteq \W$ be a torsion class of $\W$. Then $\T$ is \emph{functorially finite in $\W$} if and only if there exists $M \in \tr(\W)$ such that $\T = \Gen_\W M$.
\end{definition-theorem}

As we have stated Theorem~\ref{thm:functorially_finite}, the module $M$ satisfying $\T = \Gen_\W M$ may not be unique. We use the following to formulate the uniqueness result we will use in this paper.

\begin{definition}\
    \begin{enumerate}
        \item Let $M\in\tr(\W)$. Write $M = \bigoplus_{i = 1}^k M_i$ as a direct sum of indecomposable modules. We say that $M$ is \emph{gen-minimal} if for all $j$ we have that $M_j \notin \Gen_\W\left(\bigoplus_{i \neq j} M_i\right)$. Equivalently, for all $j$ we have that $\Gen_\W M \neq \Gen_\W\left(\bigoplus_{i \neq j} M_i\right)$.
        \item Let $\mathcal{C} \subseteq \W$ be a subcategory, and let $M \in \mathcal{C}$ be indecomposable. We say that $M$ is (indecomposable) \emph{split projective} in $\mathcal{C}$ if every epimorphism in $\mathcal{C}$ with target $M$ is split.
    \end{enumerate}
\end{definition}

By combining \cite[Lemma~2.8]{IT} with \cite[Theorem~2.7]{AIR}, we obtain the following.

\begin{proposition}\label{prop:functorially_finite}
    The association $M \mapsto \Gen_\W M$ is a bijection between the set of gen-minimal $\tau$-rigid modules in $\W$ and the set of functorially finite torsion classes in $\W$. The inverse sends a (functorially finite) torsion class $\T \subseteq \W$ to the direct sum of the indecomposable modules which are split projective in $\T$.
\end{proposition}

We conclude this section with the following.

\begin{definition-theorem}\label{def:tau_tilting_finite}
    We say the algebra $\Lambda$ is \emph{$\tau$-tilting finite} if the following equivalent conditions hold.
    \begin{enumerate}
        \item The lattice $\tors(\mods\Lambda)$ is finite.
        \item There are finitely many functorially finite torsion classes in $\mods\Lambda$.
        \item Every torsion class of $\mods\Lambda$ is functorially finite.
        \item The lattice $\wide(\mods\Lambda)$ is finite.
        \item There are finitely many functorially finite wide subcategories in $\mods\Lambda$.
        \item There are finitely many $\tau$-rigid modules in $\mods\Lambda$.
        \item There are finitely many $\tau$-tilting modules in $\mods\Lambda$.
    \end{enumerate}
\end{definition-theorem}

\begin{proof}
    The equivalence of (1), (2), (3), (6), and (7) can be found in \cite{DIJ} and the implication $(4 \implies 5)$ is trivial. The implications $(1 \implies 4)$ and $(5 \implies 2)$ follow from the fact that there is an injective map $\wide(\mods\Lambda) \rightarrow \tors(\mods\Lambda)$, that every functorially finite torsion class is in the image of this map, and that the preimage of any functorially finite torsion class must be functorially finite, see \cite{IT,MS}. 
\end{proof}

\subsection{Bricks and semibricks}\label{sec:bricks}

Another important class of modules in $\tau$-tilting theory are the \emph{bricks}, defined as follows.

\begin{definition}
    We say a module $B \in \mods\Lambda$ is a \emph{brick} if $\mathrm{End}_\Lambda(B)$ is a division algebra. We say a set $\mathcal{S}$ of bricks is a \emph{semibrick} if for all $B \neq C \in \mathcal{S}$ we have $\Hom_\Lambda(B,C) = 0 = \Hom_\Lambda(C,B)$.
\end{definition}

\begin{notation}
    Let $\W \subseteq \mods\Lambda$ be a wide subcategory. We denote by $\brick(\W)$ and $\sbrick(\W)$ the set of bricks and semibricks in $\W$, respectively.
\end{notation}

\begin{remark}\label{rem:brick}
    We note that the property of being a brick does not change when one replaces $\mods\Lambda$ with a subcategory. In other words, given a wide subcategory $\W \subseteq \mods\Lambda$, one has $\brick(\W) = \W \cap \brick(\mods\Lambda)$ and $\sbrick(\W) = \W \cap \sbrick(\mods\Lambda)$.
\end{remark}

In light of the previous remark, it is sometimes more intuitive to work with (semi)bricks than with $\tau$-rigid modules. The ``brick-$\tau$-rigid correspondence'' of Demonet-Iyama-Reiten \cite{DIJ}, and its generalization to semibricks due to Asai \cite{asai}, offers further motivation for the study of bricks and semibricks:

\begin{theorem}\label{thm:DIJ}
    Let $\W \subseteq \mods\Lambda$ be a  functorially finite wide subcategory.
    \begin{enumerate}
        \item \cite[Lemma~4.3]{DIJ} Let $M\in \tr(\W)$ be indecomposable, and denote $$r(M) = \{g \in \Hom_\Lambda(M,M) \mid 0 \neq g(M) \neq M\}.$$ Then $$\beta(M):=M/\left(\sum_{g \in r(M)} \im g\right)$$
        is a brick which satisfies $\Filt\Gen_\W(\beta(M)) = \Gen_\W(M)$.
        \item \cite[Theorem~4.1]{DIJ} The association $\beta$ is an injection from the set of indecomposable modules which are $\tau$-rigid in $\W$ to the set of bricks in $\W$. The image of this injection is the set of bricks $B \in \W$ for which the torsion class $\Filt\Gen_\W(B)$ is functorially finite in $\W$.
        \item \cite[Theorem~1.3]{asai} Let $M\in \tr(\W)$ be gen-minimal and decompose $M = \bigoplus_{i = 1}^k M_i$ as a direct sum of indecomposable modules. For $j \in \{1,\ldots,k\}$, denote $\mathcal{F}_j:= \left(\bigoplus_{i \neq j} M_i\right)^{\perp_\W}$. Then
        $\mathcal{X}(M) := \left\{\beta\left(f_{\mathcal{F}_j}M_j\right)\right\}$
        is a semibrick which satisfies $\Filt\Gen_\W(\mathcal{X}(M)) = \Gen_\W(M)$.
        \item \cite[Theorem~1.3]{asai} The association $\mathcal{X}$ is a injection from the set of gen-minimal $\tau$-rigid modules in $\W$ to the set of semibricks in $\W$. The image of this injection is the set of semibricks $\mathcal{X} \subseteq \W$ for which the torsion class $\Filt(\Gen_\W \mathcal{X})$ is functorially finite in $\W$.
    \end{enumerate}
\end{theorem}

\begin{remark}\label{rem:sbrick_unique}
    Note in particular that Theorem~\ref{thm:DIJ}(4) implies that if $\mathcal{S}, \mathcal{S}' \subseteq \W$ are semibricks which satisfy $\Filt(\Gen_\W\mathcal{S}) = \Filt(\Gen_\W\mathcal{S}')$ (and this torsion class is functorially finite), then $\mathcal{S} = \mathcal{S}'$. It is also shown in \cite{asai} that this remains true without the assumption that $\Filt(\Gen_\W\mathcal{S})$ be functorially finite, but this generalization no longer follows from Theorem~\ref{thm:DIJ} in its stated form.
\end{remark}

We refer to $\beta(M)$ (resp. $\mathcal{X}(M)$) as the \emph{brick corresponding to $M$} (resp. \emph{semibrick corresponding to $M$}). We note that if $M$ is (indecomposable) $\tau$-rigid in both $\W$ and $\V$, then $\beta(M)$ and $\mathcal{X}(M)$ remain the same whether computed in $\W$ or $\V$. That is, the formulas for $\beta(M)$ and $\mathcal{X}(M)$ do not depend on the wide subcategory $\W$. Both the domains and images of $\beta$ and $\mathcal{X}$, on the other hand, do. We explain this further in Remark~\ref{rem:beta_domain} and Proposition~\ref{prop:sf-brick}.

\begin{notation}
    Let $\W$ be a functorially finite wide subcategory, let $B \in \W$ be a brick, and suppose $\Filt\Gen_\W(B)$ is functorially finite in $\W$. We denote by $\beta^{-1}_\W(B)$ the unique indecomposable module $M \in \tr(\W)$ and satisfies $\beta(M) = B$. We define $\mathcal{X}^{-1}_\W(\mathcal{S})$ analogously.
\end{notation}

\begin{remark}\label{rem:beta_domain}
    It is important to specify $\W$ in the notation $\beta^{-1}_\W$. Indeed, there are examples where $\beta_\W^{-1}(B) \neq \beta_{\U}^{-1}(B)$. For example, let $\Lambda$ be the preprojective algebra of type $A_3$, take $\W = \mods\Lambda$, and take $\U = \Filt(P_2/S_2)$. Then $\beta_\W^{-1}(P_2/S_2) = P_2$ and $\beta_\U^{-1}(P_2/S_2) = P_2/S_2$.
\end{remark}

Another important use of bricks is in interpreting the join-irreducible labeling of the lattice of torsion classes. In particular, we consider the following definition.

\begin{definition}\label{def:min_extending}\cite[Definition~1.1 and~2.10]{BCZ}
    Let $\W$ be a wide subcategory and let $\T \in \tors\W$ be a torsion class.
    \begin{enumerate}
        \item A module $M \in \W$ is called a \emph{minimal extending module} for $\T$ if the following hold:
        \begin{enumerate}
            \item If $X$ is a proper factor of $M$, then either $X \in \T$ or $X \notin \W$.
            \item For every $X \in \T$ and every nonsplit exact sequence $0 \rightarrow M \rightarrow E \rightarrow X \rightarrow 0$, one has $E \in \T$.
            \item $M \in \T^{\perp_\W}$.
        \end{enumerate}
        \item Denote $\F = \T^{\perp_\W}$. A module $M \in \W$ is called a \emph{minimal co-extending module} for $\F$ if the following hold:
        \begin{enumerate}
            \item If $X$ is a proper submodule of $M$ then either $X \in \F$ or $X \notin \W$.
            \item For every $X \in \F$ and every nonsplit exact sequence $0 \rightarrow X\rightarrow E \rightarrow M \rightarrow 0$ one has $E \in \F$.
            \item $M \in \T = \lperpW{\F}$.
        \end{enumerate}
    \end{enumerate}
\end{definition}

We also need the following notation.

\begin{notation}\label{not:kappaW}
    Let $\W \subseteq \mods\Lambda$ be a wide subcategory. We denote by $\kappa_\W$ and $\overline{\kappa_\W}$ the kappa-map of the lattice $\tors(\W)$.
\end{notation}

We conclude by recalling some of the main results of the papers \cite{BCZ,BTZ,DIRRT}. See also \cite[Theorem~2.15]{enomoto} for the explicit extension of Theorem~\ref{thm:min_extending} to arbitrary abelian length categories (and thus arbitrary wide subcategories).

\begin{theorem}\label{thm:min_extending}
    Let $\W \subseteq \mods\Lambda$ be a wide subcategory.
    \begin{enumerate}
        \item \cite[Proposition~3.1]{BCZ}\cite[Theorem~3.3(c)]{DIRRT} There is a bijection $\brick(\W) \rightarrow \cji(\tors\W)$ given by $B \mapsto \Filt(\Gen_\W B)$.
        \item \cite[Theorem~A]{BTZ} Let $B \in \brick(\W)$. Then $\kappa_\W(\Filt(\Gen_\W B)) = \lperpW{B}$. In particular, there is a bijection $\brick\W \rightarrow \cmi(\tors(\W))$ given by $B \mapsto \lperpW{B}$.
    \end{enumerate}
\end{theorem}

\begin{notation}\label{not:brick_lab}
    Let $\U \subseteq \T$ be torsion classes in $\tors(\W)$. We denote $$\brlab[\U,\T] := \{B \mid \Filt(\Gen_\W B) \in \jlab[\U,\T]\}.$$
\end{notation}

\begin{remark}\label{rem:brick_lab}\
    \begin{enumerate}
        \item The association of a cover relation $\U \covered \T$ to the brick $\brlab[\U,\T]$ is sometimes called the \emph{brick labeling} of $\tors(\W)$. Asai introduced this labeling for the subposet of functorially finite torsion classes in \cite{asai} (see also \cite{BST} for a geometric interpretation). For non-functorially finite torsion class, the brick labeling was introduced independently in \cite{BCZ,DIRRT}.
        \item While it is often unimportant in practice, to satisfy Definition~\ref{def:edge-label}, one should place a partial order on the set of bricks, see Definition~\ref{def:edge-label}. One possibility is to set $C \preceq B$ whenever $\Filt(\Gen_\W B) \subseteq \Filt(\Gen_\W C)$. Indeed, we will use a refinement of this partial order to construct an EL-labeling in Section~\ref{sec:el}, see Remark~\ref{rem: rho}.
    \end{enumerate}
\end{remark}

Our next result is essentially contained in \cite[Section~4.3]{BTZ}, albeit with a different definition of canonical join representation. We thus provide an outline of a proof for the convenience of the reader. See also \cite[Section~4.2]{enomoto}.

\begin{theorem}\label{thm:brick_labeling}
    Let $\W \subseteq \mods\Lambda$ be a wide subcategory and let $\T \in \tors(\W)$. Then
    \begin{enumerate}
        \item $\T \in \cjrep(\tors(\W))$ if and only if there exists a semibrick $\mathcal{S} \in \sbrick(\W)$ such that $\T = \Filt(\Gen_\W \mathcal{S})$. Moreover, if such a semibrick exists then $\mathcal{S}$ is the set of minimal co-extending modules for $\T^{\perp_\W}$ and $\CJR(\T) = \{\Filt(\Gen_\W B) \mid B \in \mathcal{S}\}$.
        \item The map $\overline{\kappa_\W}$ induces a bijection $\cjrep(\tors(\W)) \rightarrow \cmrep(\tors\W)$. Moreover, for $\T \in \cjrep(\W)$, one has $\CMR(\overline{\kappa}_\W(\T)) = \{\lperpW{B} \mid \Filt(\Gen_\W B) \in \CJR(\T)\}$.
    \end{enumerate}
\end{theorem}

\begin{proof}
    Suppose that $\T\in \cjrep(\tors(\W))$.
    Then by Proposition~\ref{prop:covers_canonical_existence}, for any any torsion class $\U\lneq \T$ there exists $\T'$ such that $\U\le \T' \covered \T$.
    By \cite[Corollary~3.9]{BCZ}, for $\mathcal{S}$ the set of minimal co-extending modules for $\T^{\perp_\W}$, the join $\Join \{\Filt(\Gen_\W B) \mid B \in \mathcal{S}\}$ is an irredundant join-representation of $\T$ and it refines all other irredundant join-representations of $\T$.
    Moreover, by Theorem~\ref{thm:min_extending}, $\Filt(\Gen_\W B)$ is completely join-irreducible, for each $B\in \mathcal{S}$, and $\mathcal{S}$ is a semibrick by \cite[Proposition~3.5]{BCZ}.
    We will show in Theorem~\ref{thm:covers_canonical} that this implies $\CJR(\T) =  \{\Filt(\Gen_\W B) \mid B \in \mathcal{S}\}$.
    
    For the reverse implication, assume there exists such a semibrick $\mathcal{S} \in \sbrick(\W)$ such that $\T = \Filt(\Gen_\W \mathcal{S})$.
    By \cite[Proposition~3.7]{BCZ}, $\T= \Join \{\Filt(\Gen_\W B) \mid B \in \mathcal{S}\}$ is irredundant, and refines all other irredundant join-representations of $\T$.
    As before, Theorem~\ref{thm:min_extending} and Theorem~\ref{thm:covers_canonical} implies that $\CJR(\T) =  \{\Filt(\Gen_\W B) \mid B \in \mathcal{S}\}$.
    Thus $\T\in \cjrep(\tors(\W))$.
    
    (2) The fact that $\overline{\kappa_\W} = \Join \{\lperpW B \mid \Filt(\Gen_\W B) \in \CJR(\T)\}$ is \cite[Corollary~4.4.3]{BTZ}. Moreover, by \cite[Proposition~4.4.5]{BTZ}, we have that $\{B \mid \Filt(\Gen_\W B)\}$ is the set of minimal extending modules for the torsion class $\overline{\kappa_\W}(\T)$. The fact that $\overline{\kappa_\W}$ is a bijection and that $\overline{\kappa_\W}(\T)$ has the desired canonical meet representation thus follows from the dual of (1).
\end{proof}

To conclude this section, we briefly recall the connection between torsion classes and wide subcategories in terms of the so-called ``Ingalls-Thomas correspondences'' of \cite{IT,MS}. This result is explicit as \cite[Theorem~4.17]{enomoto}, but also appears implicitly in \cite[Section~3.2 and Corollary~5.1.8]{BTZ}.

\begin{theorem}\label{thm:IT}
    Let $\W \subseteq \mods\Lambda$ be a wide subcategory. Then the association $\V \mapsto \Filt(\Gen_\W \V)$ is a bijection $\wide(\W) \rightarrow \cjrep(\tors(\W))$. The inverse sends $\T \in \cjrep(\tors(\W))$ to
        $$\alpha(\T) := \{X \in \T \mid \ker(f) \in \T \text{ for all } f: X \rightarrow Y \text{ with } Y \in \T\}.$$
\end{theorem}

\section{$\tau$-perpendicular subcategories and $\tau$-exceptional sequences}\label{sec:tex}

In this section, we briefly review the \emph{$\tau$-perpendicular subcategories} of Jasso \cite{jasso} and the \emph{$\tau$-exceptional sequences} of Buan and Marsh \cite{BM_exceptional}. We then describe how the brick-$\tau$-rigid correspondence can be used to interpret $\tau$-exceptional sequences as a ``brick labeling'' of the lattice of wide subcategories.

\subsection{$\tau$-perpendicular subcategories}\label{sec:tau_perp}

\begin{definition}\label{def:tau_perp}
Let $\W \subseteq \mods\Lambda$ be a functorially finite wide subcategory.
\begin{enumerate}
    \item Let $M \in \W$ be $\tau$-rigid in $\W$. The \emph{$\tau$-perpendicular subcategory} of $M$ (in $\W$) is
    $$\J_\W(M) := (M^{\perp_\W}) \cap (\lperpW{(\tau_\W M)}).$$
    \item Let $\U \subseteq \W$ be a (functorially finite) wide subcategory. We say that $\U$ is a \emph{$\tau$-perpendicular subcategory} of $\W$ if there exists a $\tau$-rigid module $M \in \W$ such that $\U = \J_\W(M)$.
\end{enumerate}
\end{definition}

\begin{example}\label{ex:tau_perp}\
    \begin{enumerate}
        \item By taking $M = \Lambda$ and $M = 0$, respectively, we can see both the 0-subcategory and $\mods\Lambda$ as $\tau$-perpendicular subcategories of $\mods\Lambda$. Moreover, all ``right finite'' and ``left finite'' wide subcategories are $\tau$-perpendicular. See \cite[Lemma~4.3]{BuH}.
        \item In general, there are functorially finite wide subcategories which are not $\tau$-perpendicular. See \cite[Example~3.13]{asai}.
        \item If $\Lambda$ is $\tau$-tilting finite, then every wide subcategory is a $\tau$-perpendicular subcategory. This is shown explicitly in \cite[Theorem~4.18]{DIRRT}, and also follows from the fact that if $\Lambda$ is $\tau$-tilting finite, then every wide subcategory of $\mods\Lambda$ is both left and right finite.
    \end{enumerate}
\end{example}

\begin{remark}
    We note that the definition of a $\tau$-perpendicular category given in \cite{BuH} is given in terms of $\tau$-rigid \emph{pairs} rather than $\tau$-tilting modules. It turns out, however, that the two definitions are equivalent. This fact is implicit in many places, including the seminal work \cite{jasso}. For convenience, we give a brief explanation here. Our argument is based off of \cite[Theorem~6.4]{BuH}, which is an extension of \cite[Theorem~4.3]{BM_wide}.
    
    Let $\U \subseteq \W$. We suppose that $\U$ satisfies the definition of a $\tau$-perpendicular subcategory given in \cite{BuH}; that is, that there exists a module $M$ which is $\tau$-rigid in $\W$ and a module $P$ which is projective in $\W$ such that $\Hom_\Lambda(P,M) = 0$ and $\U = \J_\W(M) \cap P^{\perp_\W}$. Since projective modules are $\tau$-rigid, it follows that $\V := P^{\perp_\W} = \J_\W(P)$ is a $\tau$-perpendicular category. Moreover, by \cite[Theorem~6.4]{BuH} (see also \cite[Proposition~2.3]{AIR} and \cite[Lemma~3.8]{BM_wide}) we have that $M$ is $\tau$-rigid in $\V$ and satisfies $\J_\V(M) = \J_\W(M)$. Applying Theorem~\cite[Theorem~6.4]{BuH} once again, we conclude that there exists a module $M'$ such that $P\oplus M'$ is $\tau$-rigid in $\W$ and $\J_\W(P\oplus M') = \U$. This shows that any subcategory satisfying the definition given in \cite{BuH} also satisfies Definition~\ref{def:tau_perp}. Conversely, it is clear that any subcategory of the form $\J_\W(M)$ can also be written in the form $\J_\W(M) \cap P^{\perp_\W}$ by taking $P = 0$.
\end{remark}

Implicit in Definition~\ref{def:tau_perp}(2) is the fact that every $\tau$-perpendicular subcategory is a functorially finite wide subcategory. This follows from the work of Jasso \cite{jasso}, where it is also shown that if $M \in \W$ is $\tau$-rigid, then \begin{equation}\label{eqn:jasso}\rk(\J_\W(M)) + \rk(M) = \rk(\W).\end{equation}

The following is critical, and ultimately allows us to refer to the poset of $\tau$-perpendicular subcategories of $\mods\Lambda$.

\begin{theorem}\cite[Corollary~6.7]{BuH}\label{thm:iterated_tau_perp} Let $\U \subseteq \V \subseteq \W$ be a chain of subcategories. If $\U$ is $\tau$-perpendicular in $\V$ and $\V$ is $\tau$-perpendicular in $\W$, then $\U$ is $\tau$-perpendicular in $\W$.
\end{theorem}

We are now ready to formally introduce the poset of $\tau$-perpendicular subcategories.

\begin{proposition}\label{prop:poset}
    Let $\W \subseteq \mods\Lambda$ be a functorially finite wide subcategory, and let $\tperp(\W)$ be the set of $\tau$-perpendicular subcategories of $\W$. Define a relation $\leq_\tau$ on $\tperp(\W)$ so that $\U \leq_\tau \V$ if and only if $\U$ is a $\tau$-perpendicular subcategory of $\V$. Then $\leq_\tau$ is a partial order on $\tperp(\W)$.
\end{proposition}

\begin{proof}
    Reflexivity follows from the fact that $\V = \J_\V(0)$ for all $\V \in \tperp(\W)$. Transitivity is shown in Theorem~\ref{thm:iterated_tau_perp}. It thus remains only to show that $\leq_\tau$ is antisymmetric. To see this, we note that if $\U \lneq_\tau \V$, then $\rk(\U) < \rk(\V)$ by Equation~\ref{eqn:jasso}.
\end{proof}

\begin{remark}
    If $\Lambda$ is $\tau$-tilting finite, then every wide subcategory is a $\tau$-perpendicular subcategory and $\leq_\tau$ coincides with the inclusion order. (This follows from Example~\ref{ex:tau_perp}(3).) We conjecture that $\leq_\tau$ concides with the inclusion order more generally, even though $\tperp(\mods\Lambda) \neq \wide(\mods\Lambda)$ outside of the functorially finite case. See \cite[Conjecture~6.8]{BuH}.
\end{remark}

From now on, we use $\tperp(\W)$ to refer to the poset $(\tperp(\W),\leq_\tau)$. An example for the Kronecker path algebra will be discussed in Example~\ref{ex:kronecker_kappa} after we introduce the ``$\kappa$-order'' in Section~\ref{sec:kappa_order}. The following is direct consequence of Theorem~\ref{thm:iterated_tau_perp}.

\begin{corollary}\label{cor:iterated_tau_perp}
    Let $\V \subseteq \W$ be a chain of $\tau$-perpendicular subcategories. Then
    $$\tperp(\V) = \{\U \in \tperp(\W) \mid \U \leq_\tau \V\}$$
    as partially ordered sets.
\end{corollary}

In particular, Corollary~\ref{cor:iterated_tau_perp} justifies writing $\V \leq_\tau \W$ without specifying the ambient category.

\subsection{$\tau$-exceptional sequences}\label{sec:tau_exceptional}

We are now prepared to recall the definition of a $\tau$-exceptional sequence from \cite{BM_exceptional}

\begin{definition}\label{def:tau_seq}
    Let $\W \subseteq \mods\Lambda$ be a functorially finite wide subcategory. An ordered sequence $(M_k,\ldots,M_1)$ of indecomposable modules is a \emph{$\tau$-exceptional sequence} in $\W$ if $M_1$ is $\tau$-rigid in $\W$ and $(M_k,\ldots,M_2)$ is a $\tau$-exceptional sequence in $\J_\W(M_1)$.
\end{definition}

\begin{remark}\
    \begin{enumerate}
        \item In order to clarify the recursive definition, we note that $(M_k,\ldots,M_1)$ is a $\tau$-exceptional sequence if and only if there exists a sequence $\W_{k-1} \subseteq \cdots \subseteq\W_0$ of (functorially finite) wide subcategories such that (i) $\W_0 = \W$, and (ii) For all $i$ we have that $M_i$ is $\tau$-rigid in $\W_{i-1}$ and that $\W_i = \J_{\W_{i-1}}(M_i)$.
    In particular, if $(M_k,\ldots,M_1)$ is a $\tau$-exceptional sequence, then for $i < j$ we have $\Hom_\Lambda(M_i,M_j) = 0$ because $M_j\in M_i^\perp$ and $\Ext^1_\Lambda(M_i,M_j) = 0$ since $M_j \in \prescript{\prescript{}{(\W_{i-1})}{\perp}}{}{\tau_{\W_{i-1}} M_i}$ (see also Proposition~\ref{prop:AStau}). Moreover, for all $i$ we have $\Ext^1_\Lambda(M_i,M_i) = 0$ since $M_i$ is $\tau$-rigid in $\W_{i-1}$.
    \item When $\Lambda$ is hereditary, a sequence $(M_k,\ldots,M_1)$ of indecomposables is a $\tau$-exceptional sequence if and only if it is an \emph{exceptional sequence}; i.e., if and only if each $M_i$ is a brick, $\Hom(M_i,M_j) = 0$ for $i < j$, and $\mathrm{Ext}^n(M_i,M_j) = 0$ for $i \leq j$ and $n > 0$. Exceptional sequences over hereditary algebras are a classical object of study, but the condition on Ext vanishing is often quite restrictive for non-hereditary algebras. Since this condition is void for $n > 1$ in the hereditary case, one could also restrict the vanishing condition on Ext to only $n = 1$ to obtain {\c S}en's ``weak exceptional sequences'' \cite{sen}. (These are themselves a subclass of the well-studied ``stratifying systems'', see e.g. the introduction of \cite{MT} and the references therein.) In case the bricks and $\tau$-rigid modules of an algebra coincide, every $\tau$-exceptional sequence will be a weak exceptional sequence, but one consequence of {\c S}en's work is that the converse does not hold in general.
    \end{enumerate}
\end{remark}

We denote by $\tex(\W)$ the set of $\tau$-exceptional sequences in $\W$.

\begin{definition}\cite[Definition~3.1]{MT}\label{def:TF_adm}
    Let $\W \subseteq \mods\Lambda$ be a functorially finite wide subcategory. An ordered sequence $(M_k,\ldots,M_1)$ of indecomposable modules is a \emph{TF-admissible ordered $\tau$-rigid module} in $\W$ if $\bigoplus_{j = 1}^k M_j$ is $\tau$-rigid in $\W$ and for all $j \in \{2,\ldots,k\}$, one has
    $$M_j \notin \Gen_\W\left(\bigoplus_{i = 1}^{j-1} M_i\right).$$
\end{definition}

We denote by $\tfo\W$ the set of TF-admissible ordered $\tau$-rigid modules in $\W$.

\begin{remark}
    If $\bigoplus_{j = 1}^k M_j$ is gen-minimal, then $(M_k,\ldots,M_1)$ is a TF-admissible ordered $\tau$-rigid module. The converse, however, is not true. For example, let $\Lambda = K(1 \rightarrow 2 \rightarrow 3)$ and take $\W = \mods\Lambda$. Then $P_1 \oplus (P_1/P_3)$ is a $\tau$-rigid module which is not gen-minimal. We then have that $(P_1,P_1/P_3)$ is a TF-admissible ordered $\tau$-rigid module in $\mods\Lambda$, but that $(P_1/P_3,P_1)$ is not. In particular, we emphasize that the property of being gen-minimal does not depend on how the indecomposable direct summands are ordered, while the property of being TF admissible does.
\end{remark}

\begin{definition}
    Let $\mathcal{P}$ be a poset. We say that a chain $(x_0 \covered x_1 \covered \cdots \covered x_m)$ is a \emph{saturated top chain} of $\mathcal{P}$ if $x_m$ is maximal in $\mathcal{P}$ and each relation $x_i \covered x_{i + 1}$ is a cover relation. 
\end{definition}

We denote by $\stopc(\mathcal{P})$ the set of saturated top chains of $\mathcal{P}$. We conclude this section by describing bijections between the sets $\tfo(\W)$, $\tex(\W)$, and $\stopc(\tperp(\W))$ which have appeared in the literature. We note that all three of the bijections in the following theorem are built upon Jasso's seminal work on reduction of $\tau$-rigid modules \cite{jasso}.

\begin{theorem}\label{thm:tf_tau}
    Let $\W \subseteq \mods\Lambda$ be a wide subcategory.
    \begin{enumerate}
        \item \cite[Theorem~5.1]{MT} There is a bijection $\chi: \tfo(\W) \rightarrow \tex(\W)$ given as follows. Let $(M_k,\ldots,M_1) \in \tfo(\W)$. For $j \in \{1,\ldots,k\}$, let $\mathcal{F}_j := \left(\bigoplus_{i < j} M_i\right)^{\perp_\W}$ and let $N_j := f_{\mathcal{F}_j}(M_j)$. (Note that $N_1 = M_1$.) Then $\chi(M_k,\ldots,M_1) := (N_k,\ldots,N_1)$.
        \item \cite[Theorem~10.1]{BM_wide} There is a bijection $\psi: \tex(\W) \rightarrow \stopc(\tperp(\W))$ given as follows. Let $(M_k,\ldots,M_1) \in \tex(\W)$, and denote $\W_0 := \W$. For $j \in \{1,\ldots,k\}$, iteratively define $\W_j = \J_{\W_{j-1}}(M_j)$. Then $\psi(M_k,\ldots,M_1) := (\W_k \coveredtau \W_{k-1} \coveredtau \cdots \coveredtau \W_0)$.
        \item \cite[Theorem~4.3]{BM_wide} Let $(M_k,\ldots,M_1) \in \tfo(\W)$. For $j \in \{1,\ldots,k\}$, let $\W_j = \J_\W\left(\bigoplus_{i \leq j} M_i\right)$. Then $\psi \circ \chi(M_k,\ldots,M_1) = (\W_k \coveredtau \W_{k-1} \coveredtau \cdots \coveredtau \W_1 \coveredtau \W)$.
    \end{enumerate}
\end{theorem}

As a special case, note that $\tau$-exceptional sequences of length one are precisely the indecomposable $\tau$-rigid modules. In particular, suppose
$\U \coveredtau \W$ is a cover relation in $\tperp(\mods\Lambda)$, or equivalently that $(\U \coveredtau \W) \in \stopc(\tperp(\W))$. Then by Theorem~\ref{thm:tf_tau}(2), there exists a unique indecomposable module $M \in \W$ which is $\tau$-rigid and satisfies $\U = \J_\W(M)$. This leads to the following definition.

\begin{definition}\label{def:brick_label}
    Let $\U \coveredtau \W$ be a cover relation in $\tperp(\mods\Lambda)$, and let $M \in \W$ be the unique indecomposable $\tau$-rigid module which satisfies $\U = \J_\W(M)$. We refer to $\beta(M)$ as the \emph{brick label} of $U \coveredtau \W$.
\end{definition}

\begin{remark}\label{rem:brick_label}
    We have chosen to label $\U \coveredtau \W$ with $\beta(M)$ rather than with $M$ for two reasons. First, as discussed in Remark~\ref{rem:brick}, the property of being a brick is stable under passing to a wide subcategory, while the property of being $\tau$-rigid may not be. Second, we have by construction that $\Gen_\W M = \Filt(\Gen_\W M)$ is completely join-irreducible. Thus one can read $\beta(M)$ from the brick labeling of $\tors(\W)$, while in general finding $M$ requires more information. This is the approach we will use to construct ``$\kappa$-exceptional sequences'' in Section~\ref{sec:kappa_exceptional}.
\end{remark}

The brick-labeling of $\tperp(\mods\Lambda)$ will play a major role in the remainder of this paper. Indeed, we will show in Section~\ref{sec:el} that for many algebras, this labeling can be made into an ``EL-labeling''. In order to do so, it is useful to understand which bricks can appear as brick labels in $\tperp(\mods\Lambda)$. To that end, we have the following definition.

\begin{definition}\label{def:sfbrick}
    Let $B \in \W$ be a brick. We say that $B$ is an \emph{sf-brick} if there exists a semibrick $\mathcal{X}$ with $B \in \mathcal{X}$ such that $\Filt(\Gen_\W \mathcal{X})$ is functorially finite.
\end{definition}

We conclude this section with the following.

\begin{proposition}\label{prop:sf-brick}
    Let $B \in \W$ be a brick. Then there exists a cover relation $U \coveredtau \V$ in $\tperp(\W)$ with label $B$ if and only if $B$ is an sf-brick.
\end{proposition}

\begin{proof}
    Suppose first that $B$ labels a cover relation $\U \coveredtau \V$ in $\tperp(\W)$. If $\V = \W$, this means $\Filt(\Gen_\W B) = \Gen_\W(\beta^{-1}_\W(B))$ is a functorially finite torsion class of $\W$ and we are done. Otherwise, there exists a $\tau$-exceptional sequence $\sigma \in \tex(\W)$ of the form $\sigma = (\beta^{-1}_\V(B), M_{k-1},\ldots,M_1)$. Let $\chi^{-1} (\sigma) = (N_k,\ldots,N_1) \in \tfo(\W)$ be the corresponding TF-admissible ordered $\tau$-rigid module in $\W$, and denote $N := \bigoplus_{i = 1}^k N_i$. By Theorem~\ref{thm:tf_tau}, it follows that $\U = \J_\W(V)$. Now let $N' \in \tr(\W)$ be the unique gen-minimal $\tau$-rigid module which satisfies $\Gen_\W(N') = \Gen_\W(N)$. Then $N_k$ must be a direct summand of $N'$ by the assumption that $(N_k,\ldots,N_1) \in \tfo(\W)$. It then follows from Theorem~\ref{thm:DIJ}(3) that $B \in \X(N')$ and that $\Filt(\Gen_\W(\X(N'))) = \Gen_\W(N')$ is functorially finite in $\W$.
    
    Now suppose that $B$ is an sf-brick and let $\mathcal{S} \in \sbrick(\W)$ be a semibrick with $B \in \mathcal{S}$ such that $\Filt(\Gen_\W \mathcal{S})$ is functorially finite. By Theorem~\ref{thm:DIJ}(3), there exists $M \in \tr(\W)$ which is gen-minimal and satisfies $\Gen_\W(M) = \Filt(\Gen_\W \mathcal{S})$. Since $M$ is gen-minimal, any ordering of its direct summands gives a TF-admissible ordered $\tau$-rigid module. Thus we can choose some $\rho = (M_k,\ldots,M_1) \in \tfo(\W)$ such that $M = \bigoplus_{i = 1}^k M_i$ and $B = \beta\left(f_{\mathcal{F}_k} M_k\right)$ in the notation of Theorem~\ref{thm:DIJ}(3). By Theorem~\ref{thm:tf_tau}, this means that $B$ is the brick label of the cover relation $\J_\W(M) \covered \J_\W\left(\bigoplus_{i < k} M_i\right)$ in $\tperp(\W)$.
\end{proof}

\begin{remark}
    We note that there may exist sf-bricks $B$ for which the torsion class $\Filt(\Gen_\W B)$ is not functorially finite. See \cite[Example~3.13]{asai} for an example.
\end{remark}



\section{An EL-labeling for the lattice of wide subcategories}\label{sec:el}

In this section, we consider a special type of $\tau$-exceptional sequence, which we call \emph{hom-orthogonal chains}. We show that (up to permutation) at most one hom-orthogonal chain can exist between two $\tau$-perpendicular subcategories $\U \leq_\tau \W$. We then give examples of algebras where hom-orthogonal chains always exist. As a consequence, we conclude the the lattices of wide subcategories of these algebras admit EL-labelings.
We begin by recalling the definition of an edge-lexicographic (EL) labeling from \cite[Definition~3.2.1]{wachs}.

Let $P$ be a poset with smallest element $\hat{0}$ and largest element $\hat{1}$.
We let $\lab: P\to Q$ be an  edge-labeling, where $Q$ is a poset.
Frequently, $Q$ is taken to be the integers with their usual total order.
We will take $Q$ to be the set of bricks of $\Lambda$, which we will totally order (see Definition~\ref{def: rho}).

To each maximal chain $\sigma = (\hat{0}=x_{k+1} \covered x_k \covered x_{k-1}\covered \cdots \covered x_1 \covered x_0= \hat{1})$, we associate a tuple
$(\lab(x_{k}, x_{k-1}]), \lab([x_{k-1} x_{k-2}]), \ldots, \lab([x_1, x_0]))$.
We totally order the set of all such tuples in reflected lexicographic order (reading left-to-right).
Equivalently, this is the usual lexicographic order, reading tuples from right-to-left.
We say a $\sigma$ is \emph{increasing} if $\lab([x_{i+1}, x_{i}]) > \lab([x_{i}, x_{i-1}])$ for each $i\in [1,k]$.

\begin{definition}\label{def:EL-label}
We say that $\lab:P\to Q$ is an edge-lexicographic (EL) labeling for $P$ provided that for each interval $[x,y]$ in $P$ there is a unique increasing maximal chain, and furthermore that this increasing chain is smallest in the reflected lexicographic order.
\end{definition}

We now recall the definitions of $\beta, \mathcal{X}, \chi$, and $\psi$ from Theorems~\ref{thm:DIJ} and~\ref{thm:tf_tau}.

\begin{definition}
    Let $\W \in \tperp(\mods\Lambda)$, and let $\sigma \in \stopc(\tperp\W)$. Consider the corresponding $\tau$-exceptional sequence $\psi^{-1}(\sigma) =:(U_k,\ldots,U_1)$. We say that $\sigma$ is a \emph{hom-orthogonal chain} if $\Hom_\Lambda(\beta(U_j),\beta(U_i)) = 0$ for all $1 \leq i < j \leq k$.
\end{definition}

The following technical lemma is the basis of showing that, up to permutation, there can exist at most one hom-orthogonal chain with a designated target category.

\begin{lemma}\label{lem:unique_hom_ortho_chain}
    Let $\W \in \tperp(\mods\Lambda)$, let $\sigma = (\W_k \covered \cdots \covered \W_0 = \W) \in \stopc(\tperp\W)$, and suppose that $\sigma$ is a hom-orthogonal chain. Let $\psi^{-1}(\sigma)=:(U_k,\ldots,U_1)$ and $(\psi\circ\chi)^{-1}(\sigma)=:(M_k,\ldots,M_1)$ be the $\tau$-exceptional sequence and TF-admissible ordered $\tau$-rigid module corresponding to $\sigma$. Denote $M := \bigoplus_{i = 1}^k M_i$. Then $\mathcal{S} := \{\beta(U_k),\ldots,\beta(U_1)\}$ is the unique semibrick which satisfies $\Filt(\Gen_\W\mathcal{S}) = \Gen_\W M$.
\end{lemma}

\begin{proof}
    To show that $\mathcal{S}$ is a semibrick, it suffices to show that $\Hom(U_i, U_j) = 0$ for $i \neq j$. If $i < j$, this follows from the fact that $U_j \in \J_{\W_{i-1}}(U_i) \subseteq U_i^\perp$. If $i > j$, this follows from the assumption that $\sigma$ is hom-orthogonal.
    
    It remains to show that $\Filt(\Gen_\W\mathcal{S}) = \Gen_\W\left(\bigoplus_{i = 1}^k M_i\right)$. (The uniqueness of $\mathcal{S}$ then follows from Remark~\ref{rem:sbrick_unique}.) We first note that, by the definition of $\chi$, we have $U_i \in \Gen_\W(M_i)$ for each $i$. Since torsion classes are closed under quotients in $\W$, this implies that $\Filt\Gen_\W(\mathcal{S}) \subseteq \Gen_\W\left(\bigoplus_{i = 1}^k M_i\right)$. It thus suffices to show that $M_i \in \Filt\Gen_\W(\mathcal{S})$ for all $i$. We proceed by induction on $i$. For $i = 1$, we have $M_i = U_i$, so there is nothing to show. For the inductive step, denote $M_{<i}:= \bigoplus_{j = 1}^{i-1} M_i$. we consider the canonical exact sequence
    $$0 \rightarrow \ker(q) \rightarrow M_i \xrightarrow{q} U_i \rightarrow 0$$
    associated to the torsion pair $(\Gen_\W M_{<i} ,(M_{<i})^{\perp_\W})$. Then $\ker(q) \in \Gen_\W M_{<i} \subseteq \Filt(\Gen_\W\mathcal{S})$ by the induction hypothesis. Since torsion classes are closed under extensions, this implies that $M_i \in \Filt(\Gen_\W\mathcal{S})$, as desired.
\end{proof}

\begin{proposition}\label{prop:uniqueness}
    Let $\U \leq_\tau \W$ be $\tau$-perpendicular subcategories, and suppose there exist hom-orthogonal chains
    $\sigma = (\U = \W_k \covered \cdots \covered \W_0 = \W)$ and $\rho = (\U = \V_k \covered \cdots \covered \V_0 = \W)$. Denote $\psi^{-1}(\sigma) =: (U_k,\ldots,U_1)$ and $\psi^{-1}(\rho) =: (V_k,\ldots,V_1)$ the $\tau$-exceptional sequences corresponding to $\sigma$ and $\rho$, respectively. Then there exists a bijection $f: \{1,\ldots,k\} \rightarrow \{1,\ldots,k\}$ such that $\beta(U_i) = \beta(V_{f(i)})$ for all $i$.
\end{proposition}

\begin{proof}
    Let $M$ be the direct sum of the modules which are indecomposable split projective in $\lperpW{U}$. Then by Lemma~\ref{lem:unique_hom_ortho_chain} we have that $\{\beta(U_k),\ldots,\beta(U_1)\} = \X(M) = \{\beta(V_k),\ldots,\beta(V_1)\}$.
\end{proof}

We now establish sufficient criteria for the existence of hom-orthogonal chains.

\begin{definition}\label{def: rho}
    We say a total order $\preceq$ on the set of sf-bricks in $\mods\Lambda$ is \emph{reverse-hom-orthogonal} or \emph{rho} if for all sf-bricks $B \neq C \in \mods\Lambda$ we have that $B \preceq C$ implies $\Hom_\Lambda(C,B) = 0$.
\end{definition}

\begin{remark}\label{rem: rho}
    Let $B$ and $C$ be sf-bricks, and suppose that $\Filt(\Gen B) \subseteq \Filt(\Gen C)$. Then in particular $B \in \Filt(\Gen C)$, and so $\Hom_\Lambda(C,B) \neq 0$. Thus one must have $C \preceq B$ in any rho order. In particular, by identifying bricks with completely join-irreducible objects, any rho order must refine the reverse of the partial order inherited from $\tors(\mods\Lambda)$. See Remark~\ref{rem:brick_lab}(2).
\end{remark}

The main example of algebras admitting an rho order are the so-called \emph{representation-directed} algebras. These algebras are characterized by having no cycles in their Auslander-Reiten quivers. As a consequence of this definition, it turns out that every representation-directed algebra is representation finite (and thus $\tau$-tilting finite), and that every indecomposable module over a representation-directed algebra is a $\tau$-rigid sf-brick. See \cite[Corollary~IX.3.4]{ASS} and
\cite[Proposition~7.1]{TW2}. Some examples of representation-directed algebras are hereditary and tilted algebras of Dynkin type. See \cite[Proposition~IX.6.5]{ASS}. More generally, we have that an algebra $\Lambda$ admits an rho order if and only if the lattice $\tors\Lambda$ is \emph{extremal}. See \cite[Appendix~A]{keller}.

We now show that the existence of an rho order guarantees the existence of hom-orthogonal chains. To be precise, we consider the following definitions. Note that the following notations are well-defined by Proposition~\ref{prop:sf-brick}.

\begin{definition} Let $\preceq$ be an rho order on the set of sf-bricks in $\mods\Lambda$, and let $\V \leq_\tau \W$ be $\tau$-perpendicular subcategories.
    \begin{enumerate}
        \item Let $\sigma = (\V = \W_k \coveredtau \cdots \coveredtau \W_0 = \W)$ and $\rho = (\V = \V_k \coveredtau \cdots \coveredtau \V_0 = \W)$ be chains in $\stopc(\tperp\W)$ which end in $\V$. Denote $\psi^{-1}(\sigma) =: (M_k,\ldots,M_1)$ and $\psi^{-1}(\rho) =: (N_k,\ldots,N_1)$ the $\tau$-exceptional sequences corresponding to $\sigma$ and $\rho$. We write $\sigma \preceq_{lex} \rho$ if $(\beta(M_k),\ldots,\beta(M_1)) \preceq_{lex} (\beta(N_k),\ldots,\beta(N_1))$, where $\preceq_{lex}$ is the (reflected) lexicographic order on $k$-tuples of bricks.
        \item Let $\sigma = (\V = \W_k \coveredtau \cdots \coveredtau \W_0 = \W)$ be a chain in $\stopc(\tperp\W)$ which ends in $\V$. We say $\sigma$ is \emph{increasing} if $\beta(M_j) \preceq \beta(M_i)$ for all $1 \leq i < j \leq k$.
    \end{enumerate}
\end{definition}

\begin{lemma}\label{lem:existence}
    Let $\preceq$ be an rho order on the set of sf-bricks in $\mods\Lambda$, and let $\V \leq_\tau \W$ be wide subcategories. Let $\sigma = (\V = \W_k \coveredtau \cdots \coveredtau \W_0 = \W)$ be smallest with respect to $\preceq_{lex}$ amongst the chains in $\stopc(\tperp\W)$ which end in $\V$. Then
    \begin{enumerate}
        \item $\sigma$ is increasing.
        \item $\sigma$ is a hom-orthogonal chain.
        \item $\sigma$ is the unique increasing chain in $\stopc(\tperp\W)$ which ends in $\V$.
    \end{enumerate}
\end{lemma}

\begin{proof}
    (1) Denote by $(M_k,\ldots,M_1):= \psi^{-1}(\sigma)$ the $\tau$-exceptional sequence corresponding to $\sigma$. We prove the result by
	induction on $k:= \rk(\W) - \rk(\V)$.
	
	There is nothing to prove if $k = 1$, so suppose $k > 1$. By assumption, $\beta(M_1)$ is minimal amongst those bricks $B$ for which $\mathcal{V} \subseteq \J_\W(\beta_\W^{-1}B)$. 	It follows that $\psi(M_k,\ldots,M_2)$ is minimal amongst the chains in $\stopc(\J(M_1))$ which end at $\mathcal{V}$. This means $\psi(M_k,\ldots,M_2)$ is increasing by the induction hypothesis. It remains only to show that $\beta(M_1) \preceq \beta(M_2)$.
	
	By Theorem~\ref{thm:tf_tau}, there exists $N \in \W$ such that $N\oplus M_1$ is $\tau$-rigid in $\W$ and satisfies $\J_\W(N \oplus M_1) = \J_{\J_\W(M_1)}(M_2)$ and $M_2 \in \Gen_{\J_\W(M_1)}N$. In particular, $\beta(M_2) \in \Filt(\Gen_\W\beta(N))$, and so $\Hom(\beta(N),\beta(M_2)) \neq 0$. It follows that $\beta(N) \preceq \beta(M_2)$. Finally, we have that $\mathcal{V} \leq_\tau \J_\W(N)$, and so $\beta(M_1) \preceq \beta(N)$ by the minimality of $\beta(M_1)$. This concludes the proof.
	
	(2) It is immediate from the definitions that any increasing chain must be hom-orthogonal.
	
	(3) Let $\rho$ be an increasing chain in $\stopc(\tperp\W)$ which ends in $\V$. We will show that $\rho = \sigma$. Let $(N_k,\ldots,N_1) = \psi^{-1}(\rho)$ the the $\tau$-exceptional sequence corresponding to $\rho$. Note that by (1), we have that both $\sigma$ and $\rho$ are hom-orthogonal chains. Thus by Proposition~\ref{prop:uniqueness} there exists a bijection $f: \{1,\ldots,k\} \rightarrow \{1,\ldots,k\}$ such that $\beta(M_i) = \beta(N_{f(i)})$ for all $i$. Since both $\sigma$ and $\rho$ are increasing, this is only possible if $f$ is the identity. We conclude that $\psi^{-1}(\sigma) = \psi^{-1}(\rho)$, which implies the result.
\end{proof}

We now conclude our first main theorem.

\begin{theorem}[Theorem~\ref{thm:intro:mainA}]\label{thm:mainA}
    Suppose the set of sf-bricks in $\mods\Lambda$ admits an rho order $\preceq$. Then the brick-labeling of $\tperp(\mods\Lambda)$ is an EL-labeling with respect to the total order $\preceq$.
\end{theorem}

\begin{proof}
    The existence and uniqueness of a increasing chain are precisely Lemma~\ref{lem:existence} parts (1) and (3), respectively.
\end{proof}

In particular, Theorem~\ref{thm:mainA} implies that hereditary and tilted algebras of Dynkin type have lattices of wide subcategories which admit EL-labelings.



\section{The flag property for canonical join representations}\label{sec:flag}

Let $L$ be a completely semidistributive lattice. Generalizing the finite case from \cite{emily_canonical}, we recall that the \emph{canonical join complex} of $L$ is the simplicial complex with underlying set $\cji(L)$ such that a subset $A \subseteq \cji(L)$ spans a simplex if and only if $\Join A$ is a canonical join representation. The main result of \cite{emily_canonical} is that if $L$ is finite, then the canonical join complex is flag.

Now let us also recall from \cite{BCZ} that when $L = \tors\Lambda$ for some finite-dimensional algebra $\Lambda$, we can identify $\cji(L)$ with the set of bricks in $\mods\Lambda$. It is then shown in \cite[Theorem~1.8]{BCZ} that a set of bricks $A$ joins canonically if and only if $A$ is a semibrick (see Theorem~\ref{thm:brick_labeling}). This is also a flag condition, in the sense that $A$ is a semibrick if and only if every subset $B \subseteq A$ with $|B| = 2$ is a semibrick. In particular, the canonical join complex of $\tors\Lambda$ is also flag.

The purpose of this section is to generalize these results by proving the flag condition for a large family of (infinite) completely semidistributive lattices. In doing so, we also give a positive answer to \cite[Question~4.13]{enomoto} for this family (see Corollary~\ref{cor:extended_kappa_bij}). The lattices we consider are defined as follows.

\begin{definition}\label{def:kappa_lattice}
    Let $L$ be a completely semidistributive lattice. We say that $L$ is a \emph{well-separated completely semidistributive lattice}, abbreviated \emph{ws-csd-lattice} if for any $x \not\leq y \in L$ there exists $j \in \cji(L)$ such that $j \leq x$ and $y \leq \kappa(j)$.
\end{definition}

\begin{remark}\label{rem:kappa_lattice}\
    We refer readers to \cite[Sections~1 and~3.1]{RST} for a detailed discussion of the well-separated property and other conditions on infinite semidistributive lattices. In particular, we point out the following.
    \begin{enumerate}
        \item Every finite (completely) semidistributive lattice is a ws-csd-lattice.
        \item Both lattices of torsion classes are always ws-csd lattices, see \cite[Section~8.2]{RST}.
        \item Suppose $y \lneq x \in L$. Then the element $j$ coming from Definition~\ref{def:kappa_lattice} must lie in $\jlab[y,x]$ by Proposition~\ref{prop:cover}(1). In particular, this means the interval $[y,x]$ contains a cover relation. Lattices with this property are sometimes called ``weakly atomic'', see \cite[Section~3.1]{RST}.
    \end{enumerate}
\end{remark}

We now prove the main result of this section, which generalizes \cite[Theorem~5.13]{RST} and \cite[Proposition~4.9]{enomoto}. 

\begin{theorem}[Theorem~\ref{thm:intro:mainC}]\label{thm:covers_canonical}
    Let $L$ be a ws-csd-lattice. Let $A \subseteq \cji(L)$ and $x = \Join A$. Then the following are equivalent.
    \begin{enumerate}
        \item $A$ joins canonically; i.e., $A = \CJR(x)$.
        \item $\Join A$ is an irredundant join representation of $x$ which refines every other irredundant join representation of $x$.
        \item For all $i\neq j \in A$, one has $i \leq \kappa(j)$.
        \item $A$ is an antichain and every $j \in A$ labels a cover relation of the form $u \covered x$.
    \end{enumerate}
    Moreover, if (1-4) hold, then $A$ is precisely the set of elements which label cover relations of the form $u \covered x$.
\end{theorem}

\begin{proof}
    Note that the moreover part can be found as \cite[Lemma~4.10]{enomoto}.

    $(1 \implies 2)$: Suppose $A$ join canonically and suppose for a contradiction that $\Join A$ is not irredundant. Thus there exists $j \in A$ such that $x = \Join (A\setminus\{j\})$. By definition, $\Join A$ must refine $\Join (A\setminus\{j\})$, so there exists $i \in A \setminus \{j\}$ such that $j \leq i$. This means that $A$ is not an antichain, which contradicts Definition~\ref{def:canonical_join_rep}(2).
    
    $(2 \implies 3)$: This follows using an argument nearly identical to that of \cite[Lemma~4.8]{enomoto}. Indeed, suppose (2) holds and that there exist $i, j \in A$ such that $i \not\leq \kappa(j)$. As in the proof of \cite[Lemma~4.8]{enomoto}, this means that $i \join j = i \join j_*$ by \cite[Lemma~2.57]{FJN}. Denoting $B:= (A \setminus\{j\}) \cup \{j_*\})$, we thus have a join representation $x = \Join B$. 
    We will show that $\Join B$ is irredundant.
    Observe that $\Join (B\setminus\{j_*\}) = \Join A\setminus \{j\} \lneq \Join A$ because $\Join A$ is irredundant.
    Now, consider $\Join (B\setminus \{i\})$ for some $i\neq j_*$.
    Then $\Join (B\setminus \{i\}) \leq \Join (A\setminus \{i\}) \lneq \Join A$ because $\Join A$ is irredundant.
    Therefore $\Join B$ is irredundant. 
    By assumption we have that $\Join A$ refines $\Join B$. Since $j \not\leq j_*$, this means there exists $j \neq i \in A$ such that $j \leq i$, which contradicts that $\Join A$ is irredundant.
    
    $(3 \implies 4)$: Suppose (2) holds. The fact that $A$ is an antichain is immediate from the fact that $i \not \leq \kappa(j)$. It then follows from Proposition~\ref{prop:cover}(2) that every $j \in A$ labels a cover relation of the form $u \covered x$.
    
    $(4 \implies 1)$: The proof is similar to that of \cite[Proposition~3.7]{BCZ}. Let $x = \Join B$ be a join representation. We will show that $\Join A$ refines $\Join B$.
    
    Let $j \in A$, and let $u \covered x$ be labeled by $j$. By Proposition~\ref{prop:cover}(1), this means $\kappa(j) \geq u$. Now there exists $b \in B$ such that $b \not \leq u$ (otherwise we would have $\Join B = u$).
    By Proposition~\ref{prop:cover}(3), we have $u= \kappa(j)\meet x$ and $b\le x$.
    Thus also we have $b \not \leq \kappa(j)$. Definition~\ref{def:kappa_lattice} then says that there exists $i \in \cji(L)$ such that $i \leq b$ and $\kappa(j) \leq \kappa(i)$. We will show that $i = j$, which will imply that $\Join A$ refines $\Join B$. 
    
    First note that $x \not \leq \kappa(i)$ since $i \leq x$. Thus $x \meet \kappa(j) \leq x \meet \kappa(i) \lneq x$. By Proposition~\ref{prop:cover}(3), this means $u = x \meet \kappa(j) = x \meet \kappa(i)$. It follows that $i \not \leq u$, and so $x = u \join i$. Using Proposition~\ref{prop:cover}(3) once again, we conclude that $i = \jlab[u, x] = j$.
\end{proof}

As a consequence, we obtain the flag condition for the canonical join complex.

\begin{corollary}[Theorem~\ref{thm:intro:mainB}]\label{cor:flag}
    Let $L$ be a ws-csd-lattice. Then the canonical join complex of $L$ is a flag simplicial complex.
\end{corollary}

\begin{proof}
    Let $A \subseteq \cji(L)$. It is clear from Condition~2 in Theorem~\ref{thm:covers_canonical} that $A$ joins canonically if and only if every subset $B \subseteq A$ with $|B| = 2$ joins canonically. This implies the result.
\end{proof}

\begin{remark}
    Suppose $L = \tors\Lambda$ is the lattice of torsion classes of a finite-dimensional algebra. Consider two distinct bricks $B \neq C \in \mods\Lambda$. Then by \cite[Theorem~8.6]{RST}, we have that $\Filt(\Gen B) \leq \kappa(\Filt(\Gen C))$ if and only if $\Hom_\Lambda(B,C) = 0$. That is, Condition~2 in Theorem~\ref{thm:covers_canonical} is precisely the semibrick condition on bricks used in \cite{BCZ}.
\end{remark}

We conclude this section with the following, which shows that the answer to \cite[Question~4.13]{RST} is ``yes'' when $L$ is a ws-csd lattice.

\begin{corollary}\label{cor:extended_kappa_bij}
    Let $L$ be a ws-csd-lattice and let $x \in \cjrep(L)$. Then $\overline{\kappa}(x) \in \cmrep(L)$ and $\CMR(\overline{\kappa}(x)) = \{\kappa(j) \mid j \in \CJR(x)\}$. In particular, the extended kappa-maps $\overline{\kappa}$ and $\overline{\kappa}^d$ give inverse bijections between $\cjrep(L)$ and $\cmrep(L)$.
\end{corollary}

\begin{proof}
    Let $x \in \cjrep(L)$, and denote $A = \{\kappa(j) \mid j \in \CJR(x)\}$. Then for all $\kappa(i) \neq \kappa(j) \in A$, Condition~2 in Theorem~\ref{thm:covers_canonical} tells us that $\kappa(i) \geq \kappa^d(\kappa(j))$. This precisely says that the dual of Condition~2 holds for $\overline{\kappa}(x) = \Meet A$, and so $\Meet A$ is a canonical meet representation. The fact that $\overline{\kappa}$ and $\overline{\kappa}^d$ are inverse bijections then follows from duality.
\end{proof}


\section{The $\kappa$-order and core label orders}\label{sec:kappa_order}

In this section, we recall three additional partial orders which can be placed on the canonically join-representable elements of a ws-csd lattice $L$: the $\kappa$-order, the upper core label order, and the lower core label order (or shard intersection order). We show that the $\kappa$-order is the common refinement of the two core label orders and give necessary and sufficient conditions for these orders to coincide.

\subsection{Labels of intervals}\label{sec:intervals}

Fix a ws-csd lattice $L$. We note that not all of the definitions in this section will require the assumption that $L$ be well-separated, but we keep this assumption on $L$ for readability.

As a starting point, we need the following result, which essentially follows from \cite[Theorem~4.3]{RST}.

\begin{lemma}\label{prop:intervals}
    Let $x \leq y \in L$. Then the interval $[x,y]$ is also a ws-csd-lattice with the join and meet operations inherited from $L$. Moreover:
    \begin{enumerate}
        \item There is a bijection $\jlab[x,y] \rightarrow \cji[x,y]$ given by $j \mapsto x \join j$.
        \item Identifying each $j$ with its image under this bijection, the label of any cover relation in the interval $[x,y]$ is the same whether computed in $L$ or in the sublattice $[x,y]$.
    \end{enumerate}
\end{lemma}

As a consequence of Lemma~\ref{prop:intervals}, every atom of the sublattice $[x,y]$ is of the form $x \join j$ for some $j \in \cji(L)$. Thus the following is well-defined.

\begin{definition}\label{def:atom}
    Let $x \leq y \in L$. We say a completely join-irreducible element $j \in \jlab[x,y]$ is an \emph{atom} of the interval $[x,y]$ if $x \join j$ is an atom in the sublattice $[x,y]$. We denote by $\atom[x,y]$ the set of atoms in $[x,y]$. 
    Dually, we say a completely meet-irreducible element $\kappa(j) \in \mlab[x,y]$ is a \emph{coatom} of $[x,y]$ if $y\meet \kappa(j)$ is a coatom in the sublattice $[x,y]$, and we denote by $\coatom[x,y]$ the set of coatoms in $[x,y]$.
\end{definition}

\begin{example}
    Let $L$ be the lattice in Figure~\ref{fig:run_ex}, and consider the interval $[j_4,\hat{1}]$. The elements which are complely join-irreducible in $[j_4,\hat{1}]$ are then $m_1 = j_4 \join j_2$ and $m_2 = j_4 \join j_1$. By identifying $m_1$ with $j_2$ and $m_2$ with $j_1$, it is then possible to verify (1) and (2) in Lemma~\ref{prop:intervals}. Moreover, we have $\atom[j_4,\hat{1}] = \{j_1,j_2\}$.
\end{example}

We now consider special types of intervals in which the atoms and coatoms are in bijection with one another.

\begin{definition}\label{def:nuclear}
    Let $x \leq y \in L$. We say that $[x,y]$ is a \emph{nuclear interval} if
    $$x = \bigwedge \{y \meet \kappa(j) \mid \kappa(j) \in \coatom[x,y]\}$$
    and for all $x \leq z < y$ there exists $\kappa(j) \in \coatom[\hat{0},\hat{1}]$ such that $z \leq \kappa(j)$.
    Dually we say that $[x,y]$ is a \emph{conuclear interval} if  $$x = \Join \{y \join j \mid j \in \atom[x,y]\}$$
    and for all $x < z \leq y$ there exists $j \in \atom[\hat{0},\hat{1}]$ such that $j \leq z$.
\end{definition}

\begin{lemma}\label{lem:nuclear_canonical_join_rep}\
    \begin{enumerate}
        \item Suppose $[\hat{0},\hat{1}]$ is nuclear. Then $\hat{0} \in \cmrep(L)$ and $\CMR(\hat{0}) = \coatom[\hat{0},\hat{1}]$.
        \item Suppose $[\hat{0},\hat{1}]$ is conuclear. Then $\hat{1} \in \cjrep(L)$ and $\CJR(\hat{1}) = \atom[\hat{0},\hat{1}]$.
    \end{enumerate} 
\end{lemma}

\begin{proof}
    We prove only (2) since (1) is dual. The assumption that $[\hat{0},\hat{1}]$ is conuclear implies that $\hat{1} = \Join \{j \mid j \in \atom[\hat{0},\hat{1}]\}$. Now for $i \neq j \in \atom[\hat{0},\hat{1}]$, it is clear that $i \join j$ is a canonical join representation (because it is an irredundant join of atoms). The result thus follows from Corollary~\ref{cor:flag}.
\end{proof}

\begin{example}\label{ex:nuclear}
    Let $L$ be the lattice in Figure~\ref{fig:run_ex}. Since $|L| < \infty$, every interval in $[x,y] \subseteq L$  satiisfies that (a) if $x < z \leq y$ then there exists $j \in \atom[x,y]$ such that $x \join j \leq z$, and (b) if $x \leq z < y$ then there exists $\kappa(j) \in \coatom[x,y]$ such that $z \leq y \meet \kappa(j)$. On the other hand, we have that $[\hat{0},m_1]$ is both nuclear and conuclear, while the interval $[j_3, \hat{1}]$ is neither.
\end{example}

As a special case of these constructions, we recall the following from \cite{AP}. See also \cite[Theorem~4.12, Proposition~4.13]{DIRRT} and \cite[Theorem~3.12]{jasso}, which explicitly address the case of functorially finite torsion classes (which will ultimately be all we need in Section~\ref{sec:kappa_tau}). For readability, given a wide subcategory $\W \subseteq \mods\Lambda$, we denote by $\brlab_\W$, $\atom_\W$, etc. the relevant constructions in the lattice $\tors(\W)$.

\begin{theorem}\label{thm:asai_pfeifer_1}
    Let $\Lambda$ be a finite dimensional algebra, let $\W \subseteq \mods\Lambda$ a wide subcategory, and let $\U \subseteq \T$ be torsion classes in $\tors(\W)$. Then the following hold.
    \begin{enumerate}
        \item \cite[Theorem~5.2]{AP} The following are equivalent.
            \begin{enumerate}
                \item The interval $[\U,\T]$ is nuclear.
                \item The interval $[\U,\T]$ is conuclear.
                \item $\U^{\perp_\W} \cap \T \in \wide(\W)$.
            \end{enumerate}
        \item \cite[Theorem~4.2]{AP} Suppose the equivalent conditions from (1) hold. Then there is a label-preserving isomorphism of lattices $[\U,\T] \rightarrow \wide(\U^{\perp_\W} \cap \T)$ given by $\T' \mapsto \T' \cap \U^{\perp_\W}.$ That is, for any cover relation $\U \subseteq \U' \covered \T' \subseteq \T$ in $\tors(\W)$, one has that $\brlab_\W[\U',\T'] = \brlab_{\U^{\perp_\W} \cap \T}[\U' \cap \U^{\perp_\W}, \T' \cap \U^{\perp_\W}]$. In particular: \begin{enumerate}
            \item $\brlab_\W[\U,\T] = \brick(\U^{\perp_\W} \cap \T)$.
            \item $\atom_\W[\U,\T] = \{\Filt(\Gen_\W B) \mid B \text{ is simple in }\U^{\perp_\W} \cap \T\}.$
        \end{enumerate}
    \end{enumerate}
\end{theorem}

The following two results extend Theorem~\ref{thm:asai_pfeifer_1}(1) to the generality of ws-csd lattices.

\begin{proposition}\label{prop:nuclear}
    The following are equivalent.
    \begin{enumerate}
        \item $[\hat{0},\hat{1}]$ is a nuclear interval.
        \item $[\hat{0},\hat{1}]$ is a conuclear interval.
        \item $\hat{1} \in \cjrep(L)$ and $\overline{\kappa}(\hat{1}) = \hat{0}$.
        \item $\hat{0} \in \cmrep(L)$ and $\overline{\kappa}^{d}(\hat{0}) = \hat{1}$.
    \end{enumerate}
    Moreover, if (1)-(4) all hold, then $\coatom[\hat{0},\hat{1}] = \{\kappa(j) \mid j \in \atom[\hat{0},\hat{1}]\}$.
\end{proposition}

\begin{proof}
    The equivalence $(3 \iff 4)$ is an immediate consequence of Corollary~\ref{cor:extended_kappa_bij}, so we prove only $(2 \implies 3 \implies 1)$ and the moreover part.
    The direction $(1\implies 4 \implies 2)$ is similar by duality.

    $(2\implies 3)$: Suppose that $[\hat{0},\hat{1}]$ is conuclear. Then $\hat{1} \in \cjrep(L)$ and $\CJR(\hat{1}) = \atom[\hat{0},\hat{1}]$ by Lemma~\ref{lem:nuclear_canonical_join_rep}.
    
    We claim that $\overline{\kappa}(\hat{1}) = \Meet \coatom[\hat{0},\hat{1}]$. Indeed, the moreover part of Theorem~\ref{thm:covers_canonical} says that there is a bijection $\sigma$ between the atoms and coatoms of $[\hat{0},\hat{1}]$ such that for each $j \in \atom[\hat{0},\hat{1}]$ the cover relation $[\sigma(j),\hat{1}]$ is labeled by $j$. Proposition~\ref{prop:cover}(1) then implies that $\sigma(j) = \kappa(j)$, which proves the claim.

    Now suppose for a contradiction that $\hat{0} \neq \overline{\kappa}(\hat{1})$. Then since $[\hat{0},\hat{1}]$ is conuclear, there exists $j \in \atom[\hat{0},\hat{1}]$ such that $j \leq \overline{\kappa}(\hat{1})$. On the other hand, we have $\overline{\kappa}(\hat{1}) \leq \kappa(j)$ by the claim. Since we know $j \not \leq \kappa(j)$, we have reached a contradiction.

    $(3 \implies 1)$: 
    Suppose (3) holds. Then Proposition~\ref{prop:covers_canonical_existence} tells us that for every $\hat{0} \leq z < \hat{1}$ there exists $\kappa(j) \in \coatom[\hat{0},\hat{1}]$ such that $z \leq \kappa(j)$. Moreover, as in the proof of $(2 \implies 3)$, Theorem~\ref{thm:covers_canonical} implies that there is a bijection $\sigma: \CJR(\hat{1}) \rightarrow \coatom[\hat{0},\hat{1}]$ such that for all $j \in \CJR(\hat{1})$ the cover relation $\sigma(j) \covered \hat{1}$ is labeled by $j$. Proposition~\ref{prop:cover} then implies that $\sigma(j) = \kappa(j)$. It follows that $\hat{0} = \overline{\kappa}(\hat{1})$ is the meet of the coatoms of $L$.

    Finally, we prove the moreover part. Suppose (1)-(4) all hold and let $j \in \atom[\hat{0},\hat{1}]$. We showed in the proof of $(2\implies 3)$ that $j$ labels a cover relation $\kappa(j) \covered \hat{1}$, which by definition says that $\kappa(j) \in \coatom[\hat{0},\hat{1}]$. The fact that every coatom is of this form then follows from the fact that conditions (1)-(4) taken together are self-dual.
\end{proof}

\begin{corollary}\label{cor:nuclear}
    Let $x \leq y \in L$. Then $[x,y]$ is a nuclear interval if and only if it is a conuclear interval. Moreover, in this case we have $\coatom[x,y] = \{\kappa(j) \mid j \in \atom[x,y]\}$.
\end{corollary}

\begin{proof}
    By Lemma~\ref{prop:intervals}, we have that the sublattice $[x,y]$ is a ws-csd lattice. The result thus follows by applying Proposition~\ref{prop:nuclear} to this sublattice.
\end{proof}

We are now prepared for the main definitions of this section.

\begin{definition}\label{def:pop}
    Let $x \in \cjrep(L)$.
    \begin{enumerate}
        \item We denote by $\popd(x)$ the meet of the elements which are covered by $x$; that is, $\popd(x) = x \meet \left(\Meet \coatom[\hat{0},x]\right)$.
        \item We denote by $\pop(x)$ the join of the elements which cover $x$; that is, $\pop(x) = x \join \left(\Join \atom[x,\hat{1}]\right)$.
        \item The \emph{lower core} of $x$ is the interval $\Cored(x) := [\popd(x),x]$. The \emph{lower core label set} of $x$ is $\cored(x) := \jlab(\Cored(x))$. 
        \item The \emph{upper core} of $x$ is the interval $\Core(x) := [\overline{\kappa}(x),\pop(\overline{\kappa}(x)]$. The \emph{upper core label set} of $x$ is $\core(x) := \jlab(\Core(x))$.
        \item We denote $W(x) := \core(x) \cap \cored(x)$.
    \end{enumerate}
\end{definition}

The operators ``$\popd$'' and ``$\pop$'' are sometimes called the \emph{pop-stack sorting operators}, and have been of recent interest in the field of dynamical combinatorics. See e.g. \cite[Section~1.2]{DW} and there references therein. The name ``core'' comes from the relation between these operators and the \emph{core-label order}, which we discuss in Section~\ref{sec:kappa_order}

\begin{remark}\label{rem:atoms} Recall that the labels of the upper covers of $\overline{\kappa}(x)$ coincide with those of the lower covers of $x$. Thus one can think of $\Cored(x)$ as the interval ``spanned'' by the lower covers of $x$. With this perspective, $\Core(x)$ is ``spanned'' by the same join-irreducible elements. Said differently, one has $\atom(\Cored(x)) = \atom(\Core(x))$.
\end{remark}

Note in particular that $\popd(\hat{0}) = \hat{0}$, and so $\cored(\hat{0}) = \emptyset$. Moreover, we have that $\Cored(\hat{1}) = L$ (and thus that $\core(\hat{1}) = \cji(L)$) if and only if $L$ is (co)nuclear

\begin{example}\label{ex:pop}
        Let $L$ be the lattice in Figure~\ref{fig:run_ex} and $x = m_1$. Recall that $\overline{\kappa}(x) = j_1$. Then
    $\popd(x) = \hat{0}$ and $\pop(\overline{\kappa}(x)) = \hat{1}$. (Note in particular that $\pop(\popd(x)) = \hat{1} \neq x$.) We then have
    $$\cored(x) = \{j_2,j_3,j_4\} \qquad\qquad \core(x) = W(x) = \{j_2,j_3\}.$$
\end{example}

We now give an explicit example in the lattice of torsion classes. Recall the notation $\alpha(-)$ from Theorem~\ref{thm:IT}.

\begin{example}\label{ex:asai_pfeifer}
    Let $\Lambda$ be a finite dimensional algebra and $\W \subseteq \mods\Lambda$ a wide subcategory. Let $\T \in \cjrep(\tors(\W))$. Then:
    \begin{enumerate}
        \item By \cite[Theorem~6.7]{AP}, ${\Cored}_\W(\T) = [{\popd}_\W(\T), \T]$ is a nuclear interval which satisfies $\T \cap ({\popd}_\W(\T))^{\perp_\W} = \alpha(\T)$. In particular, ${\brlab}_\W({\Cored}_\W(\T)) = \brick(\alpha(\T))$ by Theorem~\ref{thm:asai_pfeifer_1}(2a).
        \item By construction and Theorem~\ref{thm:asai_pfeifer_1}(2b), we have $$\atom(\Cored(\T)) = \atom(\Core(\T)) = \{\Filt(\Gen_\W B) \mid B \text{ is simple in }\alpha(\T)\}.$$ Thus Theorem~\ref{thm:asai_pfeifer_1}(2) tells us there are (label-preserving) isomorphisms
        $${\Cored}_\W(\T) \cong \tors(\alpha(\T)) \cong {\Core}_\W(\T).$$
        In particular, this means $W_\W(\T) = \{\Filt(\Gen_\W B) \mid B \in \brick(\alpha(\T))\}$. This is our justification for using the symbol $W(x)$ in Definition~\ref{def:pop}(5).
    \end{enumerate}
\end{example}

If $L$ is finite, the following is a straightforward consequence of the definitions. More generally, the difficulty lies in showing that the intervals in question are (co)atom-based.

\begin{proposition}\label{prop:core_atomic}
    Let $x \in \cjrep(L)$. Then both $\Cored(x)$ and $\Core(x)$ are (co)nuclear.
\end{proposition}

\begin{proof}
    We prove that $\Cored(x)$ is nuclear, as the result for $\Core(x)$ then follows from Corollary~\ref{cor:extended_kappa_bij} and duality. Let $A = \CJR(x)$ be the canonical join representation of $x$ in the lattice $L$, and let $A' = \{\popd(x) \join j \mid j \in A\}$. We claim that $A'$ is the canonical join representation of $x$ in the sublattice $\Cored(x)$. Indeed, we already know that every $\popd(x) \join j \in A'$ is completely join-irreducible in $\cored(x)$ by Lemma~\ref{prop:intervals}, and it is clear that $x = \Join A'$. Thus let $x = \Join B$ be a join representation of $x$ in the sublattice $\Cored(x)$. Then $\Join B$ is also a join representation of $x$ in $L$. Thus for every $j \in A$, there exists $b \in B$ such that $j \leq b$. Since also $\popd(x) \leq b$ by assumption, this means $x \join j \leq b$ and so $A'$ refines $B$. This proves the claim.
    
    It follows that for all $z$ with $\popd(x) \leq z < x$ there exists $\kappa(j) \in \coatom[\popd(x),x]$ such that $z \leq \kappa(j)$ by Proposition~\ref{prop:covers_canonical_existence} (since $x\in \cjrep[\cored(x)]$).
    By definition $\popd(w)$ is equal to the meet of the coatoms of $\Cored(x)$.
    Thus, $\Cored(x)$ is nuclear, hence also conuclear by Proposition~\ref{prop:nuclear}.
\end{proof}

The following can be deduced immediately from Proposition~\ref{prop:cover}.

\begin{proposition}\label{prop:core_formulas}
    Let $x \in \cjrep(L)$. Then
    \begin{enumerate}
        \item $\cored(x) = \{j \in \cji(L) \mid j \leq x \text{ and } \kappa(j) \geq \popd(x)\}.$
        \item $\core(x) = \{j \in \cji(L) \mid j \leq \pop(\overline{\kappa}(x)) \text{ and } \kappa(j) \geq \overline {\kappa}(x)\}$.
        \item $W(x) = \{j \in \cji(L) \mid j \leq x \text{ and } \kappa(j) \geq \overline{\kappa}(x)\}$.
    \end{enumerate}
\end{proposition}

As a consequence, we obtain the following relationship between $\popd$ and $\overline{\kappa}$. We note that this result has previously been shown for finite semidistributive lattices in \cite[Theorem~9.1]{DW}.

\begin{proposition}\label{prop:pop_formula}
Let $x \in L$. Then the following hold.
\begin{enumerate}
    \item If $x \in \cjrep(L)$, then $\popd(x) = x \meet \overline{\kappa}(x)$ and $\CMR(\overline{\kappa}(x)) = \coatom(\Cored(x))$. Moreover, if in addition $\popd(x) \in \cmrep(L)$, then $\CMR(\popd(x)) \supseteq \CMR(\overline{\kappa}(x))$, with equality if and only if $\popd(x) = \overline{\kappa}(x)$.
    \item If $x \in \cmrep(L)$, then $\pop(x) = x \join \overline{\kappa}^{d}(x)$ and $\CJR(\overline{\kappa}^d(x)) = \atom(\Core(x))$. Moreover, if in addition $\pop(x) \in \cjrep(L)$, then $\CJR(\pop(x)) \supseteq \CJR(\overline{\kappa}^d(x))$, with equality if and only if $\pop(x) = \overline{\kappa}^d(x)$.
\end{enumerate}
\end{proposition}

\begin{proof}
    We prove only (1) as the proof of (2) is dual. Consider the interval $\Cored(x) = [\popd(x),x]$, which is (co)nuclear by Proposition~\ref{prop:core_atomic}.
    Hence, by Lemma~\ref{lem:nuclear_canonical_join_rep}, in the sublattice $\Cored(x)$, the canonical meet representation  of $\popd(x)$ is given by the coatoms of $\Cored(x)$.
    Then by the moreover part of Theorem~\ref{thm:covers_canonical}, we have \begin{equation}\label{eqn:coatom}\coatom(\Cored(x)) = \{\kappa(j) \mid j \in \CJR(x)\}.\end{equation}
    Thus $\overline{\kappa}(x)\meet x =  \Meet \{x \meet \kappa(j) \mid \kappa(j) \in \coatom(\Cored(x))\}$, which is equal to $\popd(x)$ by definition. 
    Moreover, Corollary~\ref{cor:extended_kappa_bij} and Equation~\ref{eqn:coatom} imply that $\CMR(\overline{\kappa}(x)) = \coatom(\Cored(x))$.

    It remains to prove the moreover part.
    Suppose $\popd(x) \in \cmrep(L)$ and let $\kappa(j) \in \CMR(\overline{\kappa}(x))$. Then $j \in \atom(\Cored(x))$ by Equation~\ref{eqn:coatom} and the moreover part of Corollary~\ref{cor:nuclear}. By Definition, this means $j$ labels a cover relation of the form $\popd(x) \covered u$, and so $\kappa(j) \in \CMR(\popd(x))$ by the dual of Theorem~\ref{thm:covers_canonical}.
\end{proof}

We conclude this section with two results which will be useful in relating the $\kappa$-order with the core label orders.

\begin{proposition}\label{prop:K_sets_join}
    Let $x \in L$. Then the following hold.\begin{enumerate}
        \item If $x \in \cjrep(L)$, then $\CJR(x) \subseteq W(x)$ and $\Join W(x) = x$.
        \item If $x \in \cmrep(L)$, then
        $\{j \mid \kappa(j) \in \CMR(x)\} \subseteq \W(\overline{\kappa}^d(x))$ and $\Meet \{\kappa(j) \mid j \in \W(\overline{\kappa}^d(x))\} = x$.
        \item If $x \in \cjrep(L)$, then $\Meet \{\kappa(j) \mid j \in W(x)\} = \overline{\kappa}(x)$.
\end{enumerate}
\end{proposition}

\begin{proof}
    (1) The fact that $\CJR(x) \subseteq W(x)$ is an immediate consequence of condition (3) in Theorem~\ref{thm:covers_canonical}. This in particular implies that $\bigvee W(x) \geq x$. The fact that $\bigvee W(x) \leq x$ then follows from the fact that every $j \in W(x)$ satisfies $j \leq x$ by Proposition~\ref{prop:core_formulas}.
    
    (2) Denote $A = \{\kappa(j) \mid j \in \W(\overline{\kappa}^d(x))\}$. Let $\kappa(j) \in \CMR(x)$. Then $j \in \CJR(\overline{\kappa}^d(x)) \subseteq \W(\overline{\kappa}^d(x))$ by Corollary~\ref{cor:extended_kappa_bij} and (1). This in particular implies that $\Meet A \leq x$. The fact that $\Meet A \geq x$ then follows from the fact that every $\kappa(j) \in A$ satisfies $\kappa(j) \geq \overline{\kappa}(\overline{\kappa}^d(x)) = x$ by Proposition~\ref{prop:core_formulas} and Corollary~\ref{cor:extended_kappa_bij}.
    
    (3) By Corollary~\ref{cor:extended_kappa_bij}, this follows by applying (2) to the element $\overline{\kappa}(x)$.
\end{proof}

\begin{remark}
    Proposition~\ref{prop:K_sets_join}(1) can be seen as an extension of Example~\ref{ex:asai_pfeifer} in the following sense. Let $\T \in \cjrep(\tors(\W))$. Then the canonical joinands of $\T$ are precisely the atoms of ${\Cored}_\W(\T)$, which correspond to the simple objects in $\alpha(\T)$ by Example~\ref{ex:asai_pfeifer}(2). Moreover, we have that $\Filt(\Gen_\W (\alpha(\T))) = \Join W_\W(\T) = \T$ by Theorem~\ref{thm:IT}. We also note that the explicit description of the $\kappa$-map for $\tors(\W)$ given in Theorem~\ref{thm:BTZ_A} could also be used to interpret Proposition~\ref{prop:K_sets_join}(2,3) as extending Example~\ref{ex:asai_pfeifer} in similar ways.
\end{remark}

\begin{proposition}\label{prop:join_irreps_cores}
    Let $j \in L$ be completely join-irreducible. Then
    \begin{enumerate}
        \item $\Cored(j) = [j_*,j]$.
        \item $\Core(j) = [\kappa(j),\kappa(j)^*]$.
        \item $\W(j) = \core(j) = \cored(j) = \{j\}.$
    \end{enumerate}
\end{proposition}

\begin{proof}
    All three items follow immediately from the definitions.
\end{proof}


\subsection{The $\kappa$-order and the core label orders}

We now use the constructions in Section~\ref{sec:intervals} to study three different partial orders on the set $\cjrep(L)$.

\begin{definition}
    Let $L$ be a ws-csd lattice.
    \begin{enumerate}
        \item The \emph{lower core label order} is $(\cjrep(L),\leq_{\clod})$, where $x \leq_{\clod} y$ if and only if $\cored(x) \subseteq \cored(y)$.
        \item The \emph{upper core label order} is $(\cjrep(L),\leq_{\clo})$, where $x \leq_{\clo} y$ if and only if $\core(x) \subseteq \core(y)$.
        \item The \emph{$\kappa$-order} is $(\cjrep(L),\leq_\kappa)$, where $x \leq_\kappa y$ if and only if $x \leq y$ and $\overline{\kappa}(y) \leq \overline{\kappa}(x)$.
    \end{enumerate}
\end{definition}

\begin{example}\label{ex:clo}
    Let $L$ be the lattice from Figure~\ref{fig:run_ex}. The upper and lower core label orders and the $\kappa$-order of $L$ are shown in Figure~\ref{fig:ex_clo}. We note that $(L,\leq_{\clod}) \neq (L,\leq_\kappa)$ as shown in \cite[Example~4.28]{enomoto}. On the other hand, we have that $(L,\leq_{\clo}) = (L,\leq_\kappa)$. Note also that $(L,\leq_{\clod})$ is not a lattice (as is shown in \cite[Figure~7]{muhle}), but that $(L,\leq_{\clod}) = (L,\leq_\kappa)$ is a lattice.
\end{example}

  \begin{figure}
   \begin{tikzpicture}
        \begin{scope}[decoration={
	markings,
	mark=at position 0.7 with {\arrow[scale=1.5]{>}}}
	]
        \draw[postaction=decorate] (0,6) -- (-1.5,4);
        \draw[postaction=decorate,dashed](0,6) -- (0,4);
        \draw[postaction=decorate](0,6) -- (1.5,4);
        \draw[postaction=decorate](-1.5,4) -- (0,2);
        \draw[postaction=decorate,dashed](0,4) -- (-1.5,2);
        \draw[postaction=decorate](1.5,4) -- (0,2);
        \draw[postaction=decorate,dashed](0,4) -- (1.5,2);
        \draw[postaction=decorate](-1.5,4) -- (-1.5,2);
        \draw[postaction=decorate](1.5,4) -- (1.5,2);
        \draw[postaction=decorate](-1.5,2) -- (0,0);
        \draw[postaction=decorate](0,2) -- (0,0);
        \draw[postaction=decorate](1.5,2) -- (0,0);
        \draw[postaction=decorate,bend right,smooth](0,6) to (-3,3);
        \draw[postaction=decorate,bend right,smooth](-3,3) to (0,0);
        \end{scope}
		\node[draw,circle,fill=white] at (0,0) {$\hat{0}$};
		\node [draw,circle,fill=white] at (-3,3) {$j_4$};
	    \node[draw,circle,fill=white] at (-1.5,2) {$j_3$};
	    \node[draw,circle,fill=white] at (0,2) {$j_2$};
	    \node[draw,circle,fill=white] at (1.5,2) {$j_1$};
	    \node[draw,circle,fill=white] at (1.5,4) {$m_3$};
	    \node[draw,circle,fill=white] at (0,4) {$m_2$};
	    \node[draw,circle,fill=white] at (-1.5,4) {$m_1$};
	    \node[draw,circle,fill=white] at (0,6) {$\hat{1}$};
	    
	    \begin{scope}[shift = {(8,0)}]
	            \begin{scope}[decoration={
	markings,
	mark=at position 0.7 with {\arrow[scale=1.5]{>}}}
	]
        \draw[postaction=decorate] (0,6) -- (-1.5,4);
        \draw[postaction=decorate,dashed](0,6) -- (0,4);
        \draw[postaction=decorate](0,6) -- (1.5,4);
        \draw[postaction=decorate](-1.5,4) -- (0,2);
        \draw[postaction=decorate,dashed](0,4) -- (-1.5,2);
        \draw[postaction=decorate](1.5,4) -- (0,2);
        \draw[postaction=decorate,dashed](0,4) -- (1.5,2);
        \draw[postaction=decorate](-1.5,4) -- (-1.5,2);
        \draw[postaction=decorate](1.5,4) -- (1.5,2);
        \draw[postaction=decorate](-1.5,2) -- (0,0);
        \draw[postaction=decorate](0,2) -- (0,0);
        \draw[postaction=decorate](1.5,2) -- (0,0);
        \draw[postaction=decorate](-1.5,4) to (-3,2);
        \draw[postaction=decorate,dashed](0,4) to (-3,2);
        \draw[postaction=decorate](-3,2) to (0,0);
        \end{scope}
		\node[draw,circle,fill=white] at (0,0) {$\hat{0}$};
		\node [draw,circle,fill=white] at (-3,2) {$j_4$};
	    \node[draw,circle,fill=white] at (-1.5,2) {$j_3$};
	    \node[draw,circle,fill=white] at (0,2) {$j_2$};
	    \node[draw,circle,fill=white] at (1.5,2) {$j_1$};
	    \node[draw,circle,fill=white] at (1.5,4) {$m_3$};
	    \node[draw,circle,fill=white] at (0,4) {$m_2$};
	    \node[draw,circle,fill=white] at (-1.5,4) {$m_1$};
	    \node[draw,circle,fill=white] at (0,6) {$\hat{1}$};
	    \end{scope}
	\end{tikzpicture}
    \caption{The core label orders and $\kappa$-order for the lattice in Figure~\ref{fig:run_ex}. The left diagram is $(L,\leq_{\clo}) = (L,\leq_{\kappa})$ and the right diagram is $(L,\leq_{\clod})$}\label{fig:ex_clo}
    \end{figure}
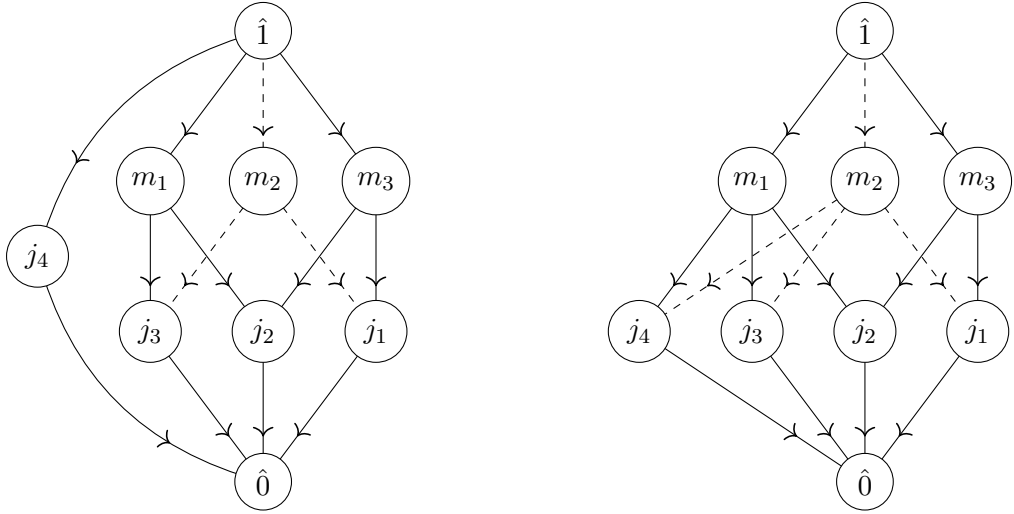

The (reverse of) the lower core label order was first introduced under the same \emph{shard intersection order} by Reading in \cite{reading} in the case where $L$ is the weak order on a finite Coxeter group. Reading then extended to the poset of regions of any simplicial hyperplane arrangement in \cite[Section~9-7.4]{reading_book}, where he also noted that one could make the same construction for any congruence uniform lattice. M\"uhle then went on to characterize precisely for which finite congruence uniform lattices the lower core label order is a lattice in \cite{muhle}. The lower core label order was further generalized to the class of (finite) ``semidistrim" lattices under the name ``row-core label order" by Defant and Williams in \cite[Section~11.4]{DW}. In the same paper, they also introduce the ``pop-core label order", which coincides with the $\kappa$-order of a finite semidistributive lattice. In the infinite case, Enomoto introduced (and named) the $\kappa$-order in \cite{enomoto}. One of his main results is that when $L$ is the lattice of torsion classes, the lower core label order and the ordering of the corresponding wide subcategories all coincide. More precisely, he proved the following.

\begin{theorem}\cite[Corollary~4.26]{enomoto}\label{thm:enomoto}
    Let $\Lambda$ be a finite-dimensional algebra and $\W \subseteq \mods\Lambda$ a wide subcategory. Then $\leq_{\clod}$ and $\leq_\kappa$ coincide on $\cjrep(\tors(\W))$. Moreover, the association $\T \mapsto \alpha(\T)$ induces a lattice isomorphism $(\cjrep(\tors(\W)),\leq_\kappa)\rightarrow \wide(\W)$.
\end{theorem}

\begin{example}\label{ex:kronecker_kappa}
   Let $L_{kr}(K)$ be as in Figure~\ref{fig:kronecker} and let $KQ$ be the path algebra over the Kronecker quiver as in Example~\ref{ex:kronecker}. Recall from Examples~\ref{ex:no_can_join}(2) and~\ref{ex:jirr} that $\cjrep(L_{kr}(K)) = L_{kr}(K) \setminus \{\emptyset\}$. Moreover, by Example~\ref{ex:kronecker} we have that $L_{kr}(K) \cong \tors(\mods KQ)$, and so the orders $\leq_\kappa$, $\leq_{\clod}$, and $\leq_{\clo}$ coincide by Theorem~\ref{thm:enomoto}. (This fact is also straightforward to verify directly.) These three orders are then given as follows:
   \begin{enumerate}
       \item Restricted to $(2^{\mathbb{P}^1(K)} \setminus \{\emptyset\}) \{\hat{0},\hat{1},s\}$, one has $x \leq_\kappa y$ if and only if $x \leq y$, where $\leq$ is the order in $L_{kr}(K)$.
       \item For $(n,t) \in \mathbb{N} \times \{-1,1\}$, one has $\hat{0} \leq_\kappa (n,t) \leq_\kappa \hat{1}$.
   \end{enumerate}
   In addition, it can be shown that $\tperp(\mods KQ)$ is isomorphic to the restriction of the $\kappa$ order of $L_{kr}(K)$ to $\mathbb{N}\times \{-1,1\} \cup \{\hat{0},\hat{1},s\}$. See \cite[Section~7]{BuH}. 
\end{example}

As a lattice-theoretic analog of Theorem~\ref{thm:enomoto}, we have the following.

\begin{proposition}\label{prop:K_sets}
    Let $x, y \in \cjrep(L)$. Then $x \leq_\kappa y$ if and only if $W(x) \subseteq W(y)$.
\end{proposition}

\begin{proof}
    Suppose first that $x \leq_\kappa y$. Since $x \leq y$, we have
    $$\{j \mid j \leq x\} \subseteq \{j \mid j \leq y\}.$$
    Likewise since $\overline{\kappa}(x) \geq \overline{\kappa}(y)$, we have
    $$\{j \mid \kappa(j) \geq \overline{\kappa}(x)\} \subseteq \{j \mid \kappa(j) \geq \overline{\kappa}(y)\}.$$
    Taking intersections, Proposition~\ref{prop:core_formulas} thus implies that $W(x) \subseteq W(y)$.
    
    Now assume that $W(x) \subseteq W(y)$. It then follows from Proposition~\ref{prop:K_sets_join} that $x = \Join W(x) \leq \Join W(y) = y$ and that $\overline{\kappa}(x) = \Meet \kappa(W(x)) \geq \Meet \kappa(W(y)) = \overline{\kappa}(y)$. We conclude that $x \leq_\kappa y$.
\end{proof}

The $\kappa$-order also satisfies the following duality.

\begin{proposition}\label{prop:self_dual}
    The map $\overline{\kappa}$ induces a poset isomorphism $(\cjrep(L),\leq_\kappa) \rightarrow (\cjrep(L^d), \leq_{\kappa^d})$ between the $\kappa$-order of $L$ and that of $L^d$.
\end{proposition}

\begin{proof}
    This is an immediate consequence of the fact that $\kappa^d$ is precisely the $\kappa$-map for the lattice $L^d$. More explicitly, suppose $x \leq_\kappa y$ in the lattice $L$. That is, $x, y \in \cjrep(L)$, $x \leq_L y$ and $\overline{\kappa}(x) \geq_L \overline{\kappa}(y)$. By Corollary~\ref{cor:extended_kappa_bij}, this means $\overline{\kappa}(x), \overline{\kappa}(y) \in \cmrep(L) = \cjrep(L^d)$, $\overline{\kappa}^d(\overline{\kappa}(x)) \geq_{L^d} \overline{\kappa}^d(\overline{\kappa}(y))$ and $\overline{\kappa}(x) \leq_{L^d} \overline{\kappa}(y)$; i.e., that $\kappa(x) \leq_{\kappa^d} \kappa(y)$ in the lattice $L^d$.
\end{proof}

\begin{remark}
    Let $\Lambda$ be a finite-dimensional algebra and $\W \subseteq \mods\Lambda$ a wide subcategory. It is well known that the torsion-free classes of $\W$ form a lattice $\tors(\W)$ which is dual to $\tors(\W)$. Proposition~\ref{prop:self_dual} then implies that the $\kappa$-orders of $\tors(\W)$ and $\torf(\W)$ are isomorphic. In particular, both will be isomorphic to the lattice $\wide(\W)$ by Theorem~\ref{thm:enomoto}. This is also implicit in \cite[Remark~4.20]{enomoto}.
\end{remark}

We now obtain the following, where the implication $(2a \implies 2b)$ can be found as \cite[Proposition~4.30]{enomoto}.

\begin{proposition}\label{prop:orders_coincide}
    Let $L$ be a ws-csd lattice. Then:
    \begin{enumerate}
        \item The following are equivalent:
            \begin{enumerate}
        \item $(\cjrep(L),\leq_\kappa) = (\cjrep(L),\leq_{\clod})$.
         \item $W(x) = \cored(x)$ for all $x \in \cjrep(L)$.
    \end{enumerate}
    \item The following are equivalent:
            \begin{enumerate}
        \item $(\cjrep(L),\leq_\kappa) = (\cjrep(L),\leq_{\clo})$.
        \item $W(x) = \core(x)$ for all $x \in \cjrep(L)$.
         \end{enumerate}
    \item If $(\cjrep(L),\leq_{\clo}) = (\cjrep(L),\leq_{\clod})$, then the properties (1a), (1b), (2a), and (2b) all hold.
    \end{enumerate}
\end{proposition}

\begin{proof}
    We prove only (1) and (3), since the proof of (2) is analogous to that of (1).

    (1) The fact that $(1b \implies 1a)$ follows from Proposition~\ref{prop:K_sets} and the definition of $W(x)$. To see that $(1a \implies 1b)$, assume (1a) and let $x \in \cjrep(L)$. By the definitions, we need only show that $\cored(x) \subseteq W(x)$. To see this, let $j \in \cored(x)$. Then $j \leq_{\clod} x$ by Proposition~\ref{prop:join_irreps_cores}. Thus by assumption, we have that $j \leq_\kappa x$, and so $j \in W(x)$ by another application of Proposition~\ref{prop:join_irreps_cores}.
    
    (3) Suppose that $(\cjrep(L),\leq_{\clo}) = (\cjrep(L),\leq_{\clod})$ and let $x \in \cjrep(L)$. We will show that $\core(x) = \cored(x)$. By the Definition of $W(x)$, this will imply that properties (1b) and (2b) must hold, thus proving the result.
    
    Let $j \in \cji(L)$. By Propositions~\ref{prop:join_irreps_cores} and the assumption that the two core label orders coincide, we thus have that
    \begin{eqnarray*}
        j \in \cored(x) & \iff & j \leq_{\clod} x\\
        & \iff & j \leq_{\clo} x\\
        &\iff& j \in \core(x).
    \end{eqnarray*}
    This concludes the proof.
\end{proof}

One may hope that Proposition~\ref{prop:orders_coincide} can be modified into a statement about when the three partial orders on $\cjrep(L)$ are isomorphic as abstract posets. The following example shows that this it is not immediately obvious how to do so.

\begin{example}\label{ex:weird_clo}
    Let $L$ be the lattice shown in Figure~\ref{fig:run_ex} with an extra element $j_5$ whose cover relations are $m_3 \covered j_5 \covered \hat{1}$. The lattice $L$, both of the core label orders, and the $\kappa$-order are all shown in Figure~\ref{fig:weird_clo}. We observe that the two core label orders are isomorphic as abstract posets, where the isomorphism exchanges $j_4$ and $j_5$. On the other hand, neither core label order is isomorphic to the $\kappa$-order.
\end{example}

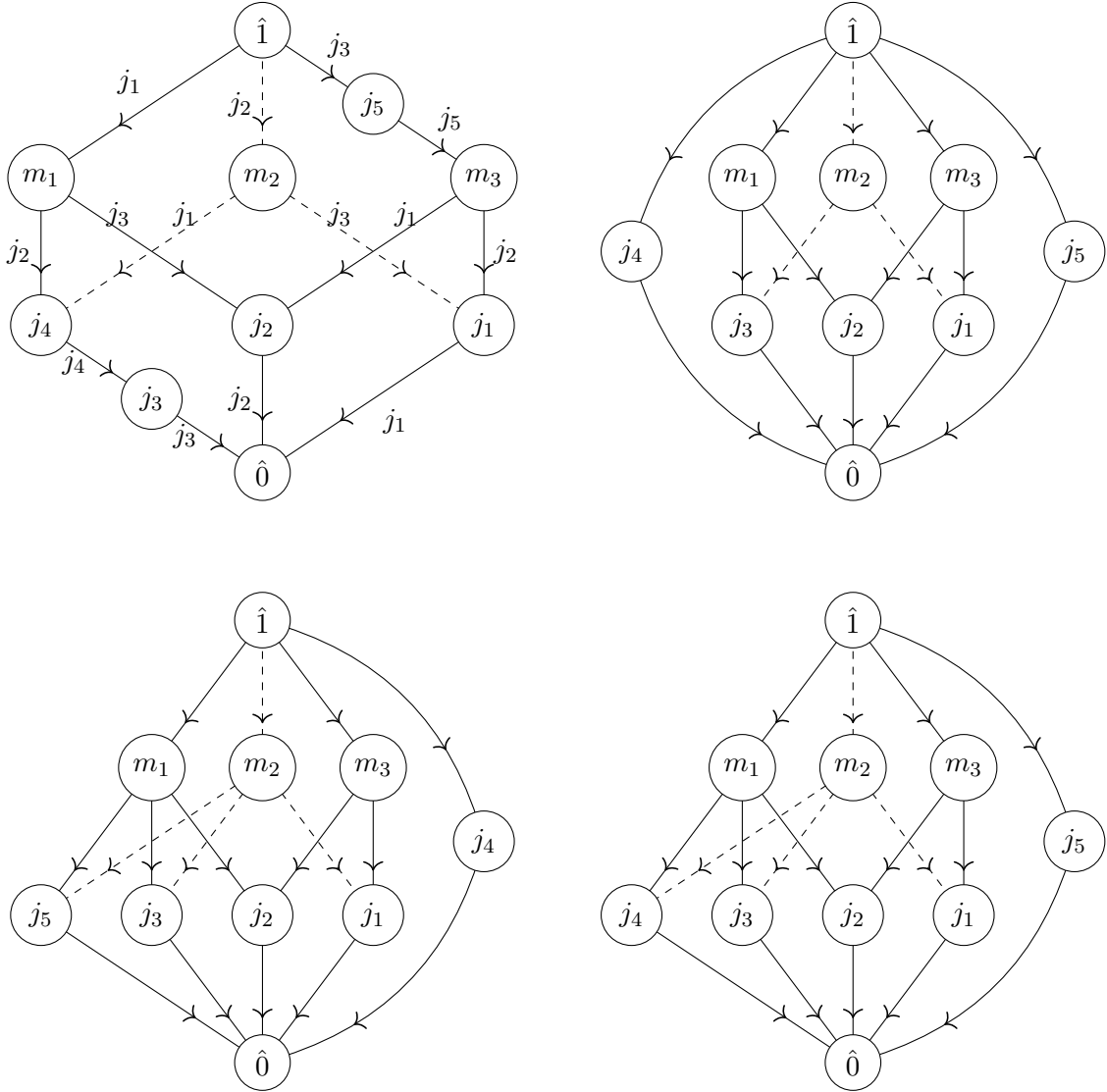
\begin{figure}
   \begin{tikzpicture}
        \begin{scope}[decoration={
	markings,
	mark=at position 0.65 with {\arrow[scale=1.5]{>}}}
	]
        \draw[postaction=decorate] (0,6) -- (-3,4) node [midway,above left] {$j_1$};
        \draw[postaction=decorate,dashed](0,6) -- (0,4) node [midway,left] {$j_2$};
        \draw[postaction=decorate](0,6) -- (1.5,5) node [midway,above right] {$j_3$};
        \draw[postaction=decorate](1.5,5) -- (3,4) node [midway,above right] {$j_5$};
        \draw[postaction=decorate](-3,4) -- (0,2) node [near start,right] {$j_3$};
        \draw[postaction=decorate,dashed](0,4) -- (-3,2) node [near start,left] {$j_1$};
        \draw[postaction=decorate](3,4) -- (0,2) node [near start, left] {$j_1$};
        \draw[postaction=decorate,dashed](0,4) -- (3,2) node [near start,right] {$j_3$};
        \draw[postaction=decorate](-3,4) -- (-3,2) node [midway, left] {$j_2$};
        \draw[postaction=decorate](3,4) -- (3,2) node [midway, right] {$j_2$};
        \draw[postaction=decorate](-3,2) -- (-1.5,1) node [midway,left] {$j_4$};
        \draw[postaction=decorate](-1.5,1) -- (0,0) node [midway,left] {$j_3$};
        \draw[postaction=decorate](0,2) -- (0,0) node [midway,left] {$j_2$};
        \draw[postaction=decorate](3,2) -- (0,0) node [midway,below right] {$j_1$};
        \end{scope}
		\node[draw,circle,fill=white] at (0,0) {$\hat{0}$};
	    \node[draw,circle,fill=white] at (-3,2) {$j_4$};
	    \node[draw,circle,fill=white] at (0,2) {$j_2$};
	    \node[draw,circle,fill=white] at (3,2) {$j_1$};
	    \node[draw,circle,fill=white] at (-1.5,1) {$j_3$};
	    \node[draw,circle,fill=white] at (1.5,5) {$j_5$};
	    \node[draw,circle,fill=white] at (3,4) {$m_3$};
	    \node[draw,circle,fill=white] at (0,4) {$m_2$};
	    \node[draw,circle,fill=white] at (-3,4) {$m_1$};
	    \node[draw,circle,fill=white] at (0,6) {$\hat{1}$};
	    
	    	      \begin{scope}[shift = {(8,-8)}]
	            \begin{scope}[decoration={
	markings,
	mark=at position 0.7 with {\arrow[scale=1.5]{>}}}
	]
        \draw[postaction=decorate] (0,6) -- (-1.5,4);
        \draw[postaction=decorate,dashed](0,6) -- (0,4);
        \draw[postaction=decorate](0,6) -- (1.5,4);
        \draw[postaction=decorate](-1.5,4) -- (0,2);
        \draw[postaction=decorate,dashed](0,4) -- (-1.5,2);
        \draw[postaction=decorate](1.5,4) -- (0,2);
        \draw[postaction=decorate,dashed](0,4) -- (1.5,2);
        \draw[postaction=decorate](-1.5,4) -- (-1.5,2);
        \draw[postaction=decorate](1.5,4) -- (1.5,2);
        \draw[postaction=decorate](-1.5,2) -- (0,0);
        \draw[postaction=decorate](0,2) -- (0,0);
        \draw[postaction=decorate](1.5,2) -- (0,0);
        \draw[postaction=decorate](-1.5,4) to (-3,2);
        \draw[postaction=decorate,dashed](0,4) to (-3,2);
        \draw[postaction=decorate](-3,2) to (0,0);
        \draw[postaction=decorate, bend left](0,6) to (3,3);
        \draw[postaction=decorate, bend left](3,3) to (0,0);
        \end{scope}
		\node[draw,circle,fill=white] at (0,0) {$\hat{0}$};
		\node [draw,circle,fill=white] at (-3,2) {$j_4$};
	    \node[draw,circle,fill=white] at (-1.5,2) {$j_3$};
	    \node[draw,circle,fill=white] at (0,2) {$j_2$};
	    \node[draw,circle,fill=white] at (1.5,2) {$j_1$};
	    \node[draw,circle,fill=white] at (1.5,4) {$m_3$};
	    \node[draw,circle,fill=white] at (0,4) {$m_2$};
	    \node[draw,circle,fill=white] at (-1.5,4) {$m_1$};
	    \node[draw,circle,fill=white] at (0,6) {$\hat{1}$};
	    \node[draw,circle,fill=white] at (3,3) {$j_5$};
	    \end{scope}
	    
	     	    	      \begin{scope}[shift = {(0,-8)}]
	            \begin{scope}[decoration={
	markings,
	mark=at position 0.7 with {\arrow[scale=1.5]{>}}}
	]
        \draw[postaction=decorate] (0,6) -- (-1.5,4);
        \draw[postaction=decorate,dashed](0,6) -- (0,4);
        \draw[postaction=decorate](0,6) -- (1.5,4);
        \draw[postaction=decorate](-1.5,4) -- (0,2);
        \draw[postaction=decorate,dashed](0,4) -- (-1.5,2);
        \draw[postaction=decorate](1.5,4) -- (0,2);
        \draw[postaction=decorate,dashed](0,4) -- (1.5,2);
        \draw[postaction=decorate](-1.5,4) -- (-1.5,2);
        \draw[postaction=decorate](1.5,4) -- (1.5,2);
        \draw[postaction=decorate](-1.5,2) -- (0,0);
        \draw[postaction=decorate](0,2) -- (0,0);
        \draw[postaction=decorate](1.5,2) -- (0,0);
        \draw[postaction=decorate](-1.5,4) to (-3,2);
        \draw[postaction=decorate,dashed](0,4) to (-3,2);
        \draw[postaction=decorate](-3,2) to (0,0);
        \draw[postaction=decorate, bend left](0,6) to (3,3);
        \draw[postaction=decorate, bend left](3,3) to (0,0);
        \end{scope}
		\node[draw,circle,fill=white] at (0,0) {$\hat{0}$};
		\node [draw,circle,fill=white] at (-3,2) {$j_5$};
	    \node[draw,circle,fill=white] at (-1.5,2) {$j_3$};
	    \node[draw,circle,fill=white] at (0,2) {$j_2$};
	    \node[draw,circle,fill=white] at (1.5,2) {$j_1$};
	    \node[draw,circle,fill=white] at (1.5,4) {$m_3$};
	    \node[draw,circle,fill=white] at (0,4) {$m_2$};
	    \node[draw,circle,fill=white] at (-1.5,4) {$m_1$};
	    \node[draw,circle,fill=white] at (0,6) {$\hat{1}$};
	    \node[draw,circle,fill=white] at (3,3) {$j_4$};
	    \end{scope}
	    
	    	      \begin{scope}[shift = {(8,0)}]
	            \begin{scope}[decoration={
	markings,
	mark=at position 0.7 with {\arrow[scale=1.5]{>}}}
	]
        \draw[postaction=decorate] (0,6) -- (-1.5,4);
        \draw[postaction=decorate,dashed](0,6) -- (0,4);
        \draw[postaction=decorate](0,6) -- (1.5,4);
        \draw[postaction=decorate](-1.5,4) -- (0,2);
        \draw[postaction=decorate,dashed](0,4) -- (-1.5,2);
        \draw[postaction=decorate](1.5,4) -- (0,2);
        \draw[postaction=decorate,dashed](0,4) -- (1.5,2);
        \draw[postaction=decorate](-1.5,4) -- (-1.5,2);
        \draw[postaction=decorate](1.5,4) -- (1.5,2);
        \draw[postaction=decorate](-1.5,2) -- (0,0);
        \draw[postaction=decorate](0,2) -- (0,0);
        \draw[postaction=decorate](1.5,2) -- (0,0);
        \draw[postaction=decorate,bend left](0,6) to (3,3);
        \draw[postaction=decorate,bend left](3,3) to (0,0);
        \draw[postaction=decorate, bend right](0,6) to (-3,3);
        \draw[postaction=decorate, bend right](-3,3) to (0,0);
        \end{scope}
		\node[draw,circle,fill=white] at (0,0) {$\hat{0}$};
		\node [draw,circle,fill=white] at (-3,3) {$j_4$};
	    \node[draw,circle,fill=white] at (-1.5,2) {$j_3$};
	    \node[draw,circle,fill=white] at (0,2) {$j_2$};
	    \node[draw,circle,fill=white] at (1.5,2) {$j_1$};
	    \node[draw,circle,fill=white] at (1.5,4) {$m_3$};
	    \node[draw,circle,fill=white] at (0,4) {$m_2$};
	    \node[draw,circle,fill=white] at (-1.5,4) {$m_1$};
	    \node[draw,circle,fill=white] at (0,6) {$\hat{1}$};
	    \node[draw,circle,fill=white] at (3,3) {$j_5$};
	    \end{scope}
	\end{tikzpicture}
	\caption{A semidistributive lattice $L$ with its join-irreducible labeling (top left), the corresponding $\kappa$-order (top right), upper core label order (bottom left), and lower core label order (bottom right). }\label{fig:weird_clo}
\end{figure}


\section{$\kappa^d$-exceptional sequences}\label{sec:kappa_tau}

The purpose of this section is to relate $\tau$-exceptional sequences to the operators $\overline{\kappa}^d$ and $\pop$ on the lattice $\tors\Lambda$. More precisely, let $\W \subseteq \mods\Lambda$ be a functorially finite wide subcategory and suppose $M \in \W$ is $\tau$-rigid (in $\W$). In Section~\ref{sec:combinatorial_tau_tilting}, we describe the torsion classes $\lperpW{(\tau_\W M)}$ and $\lperpW{(\overline{\tau_\W}M)}$ as coming from applying ${\pop}_\W$ and $\overline{\kappa_\W}^d$ to the torsion class $\Gen_\W M$. (Recall the definition of $\overline{\tau_\W}$ from Notation~\ref{not:tau_bar} and that the subscript $\W$ means the lattice being considered is $\tors(\W)$.) In Section~\ref{sec:kappa_exceptional}, we then use this description to explain how all of the $\tau$-exceptional sequences in $\W$ can be read directly from the brick labeling of $\tors(\W)$. The result is a ``combinatorialization'' of the notion of a $\tau$-exceptional sequence given in Definition~\ref{def:kappa_exceptional}.

\subsection{Rowmotion and the AR translate}\label{sec:combinatorial_tau_tilting}

In this section, we describe the combinatorial relationship between the torsion classes $\Gen_\W M$, $\lperpW{(\tau_\W M)}$, and $\lperpW{(\overline{\tau_\W} M)}$. We fix for the duration of this section a functorially finite wide subcategory $\W \subseteq \mods\Lambda$. The starting point is the following, which describes how the map $\overline{\kappa}^{d}$ is computed on the lattice of torsion classes.

\begin{theorem}\cite[Theorem~B]{BTZ}\label{thm:BTZ_A} Let $\T \in \cmrep(\tors(\W))$, so that $\T = \lperpW{\mathcal{S}}$ for some semibrick $\mathcal{S} \in \sbrick(\W)$. Then  $\overline{\kappa_\W}^{d}( \lperpW{\mathcal{S}}) = \Filt(\Gen_\W\mathcal{S})$.
\end{theorem}

We also need the following. This result is well-known, but we include an explanation of how it can be deduced from explicit results in the literature.

\begin{theorem}\label{thm:regular}
    Let $\T \in \tors(\W)$ be a functorially finite torsion class. Then $\T \in \cjrep(\tors\W)\cap \cmrep(\tors\W)$ and $|\mathrm{CJR}(\T)| + |\mathrm{CMR}(\T)| = \rk(\W)$. Moreover, if $\T = \Gen_\W M$ with $M \in \W$ a gen-minimal $\tau$-rigid module, then $|\mathrm{CJR}(\T)| = \rk(M)$.
\end{theorem}

\begin{proof}
    Since $\T$ is functorially finite, we can write $\T = \Gen_\W M$ with $M \in \W$ a gen-minimal $\tau$-rigid module. By \cite[Theorem~3.1]{DIJ}, given an inclusion $\U \subsetneq \T$ there exists a cover relation $\V \covered \T$ such that $\U \subseteq \V$, and the dual property holds as well. By Proposition~\ref{prop:covers_canonical_existence} and its dual, this means that $\T \in \cjrep(\tors\W) \cap \cmrep(\tors\W)$.
    
    Now recall from Theorem~\ref{thm:covers_canonical} that the canonical joinands of $\T$ are in bijection with cover relations of the form $\U \covered \T$ and the canonical meetands of $\T$ are in bijection with cover relations of the form $\T \covered \U$. The fact that $|\mathrm{CJR}(\T)| + |\mathrm{CMR}(\T)| = \rk(\W)$ is thus a consequence of the fact that the mutation graph of support $\tau$-tilting pairs is $\rk(\W)$-regular, see \cite[Section~2]{AIR}.
    
    Finally, we recall from Theorem~\ref{thm:DIJ}(3) that there is a semibrick $\X(M)$ with $|\X(M)| = \rk(M)$ and $\Filt(\Gen_\W\X(M)) = \Gen_\W M$. 
    
    By Theorem~\ref{thm:brick_labeling}(1), this means the canonical join representation of $\T$ is
    $$\T = \Join \{\Filt \Gen_\W(B) \mid B \in \X(M)\},$$
    and so $|\CJR(\T)| = \rk(M)$.
\end{proof}

In particular, Theorem~\ref{thm:regular} implies that both $\overline{\kappa_\W}(\T)$ and $\overline{\kappa_\W}^d(\T)$ are defined when $\T$ is functorially finite.

Together, Theorems~\ref{thm:BTZ_A} and~\ref{thm:regular} allow us to prove the following.

\begin{corollary}\label{cor:kappa_proj}
    Let $P$ be projective in $\W$. Then
    \begin{enumerate}
        \item $\overline{\kappa_\W}^{d}(\Gen_\W P) = P^{\perp_\W} =\lperpW{(\overline{\tau_\W} P)}$, and
        \item ${\pop}_\W(\Gen_\W P) = \W$
    \end{enumerate}
\end{corollary}

\begin{proof}
 First note that $P$ is $\tau$-rigid in $\W$, and so $\Gen_\W P \in \cjrep(\tors\W)$. Now let $\mathcal{S}$ be the set of modules which are simple in $\W$. We can partition $\mathcal{S} = \mathcal{S}_1 \sqcup \mathcal{S}_2$, where $\mathcal{S}_1$ contains those modules whose ($\W$-)projective covers are direct summands of $P$ and $\mathcal{S}_2 = \mathcal{S} \setminus \mathcal{S}_1$. Then $\mathcal{S}_1 \subseteq \Gen_\W P = \lperpW{\mathcal{S}_2}$ and $P^{\perp_\W} = \Filt\mathcal{S}_2 = \Filt(\Gen_\W \mathcal{S}_2)$. 
 Theorem~\ref{thm:BTZ_A} then implies that $\overline{\kappa_\W}^{d}(\Gen_\W P) = P^{\perp_\W}$. Finally, by Proposition~\ref{prop:pop_formula} we have that $\mathcal{S} \subseteq (\Gen_\W P) \vee (P^{\perp_\W}) = {\pop}_\W(\Gen_\W P)$. Since torsion classes are closed under extensions, this implies that ${\pop}_\W(\Gen_\W P) = \W$.
\end{proof}

The main result of this section is the following.

\begin{proposition}\label{prop:kappa_tau}
    Let $M \in \tr(\W)$ be gen-minimal. Then ${\pop}_\W(\Gen_\W M) = \lperpW{(\tau_\W M)}$.
\end{proposition}

We emphasize that we are using $\tau$, not $\overline{\tau}$, in Proposition~\ref{prop:kappa_tau}. In particular, since any basic projective module is also gen-minimal, Proposition~\ref{prop:kappa_tau} gives an alternative proof of the fact that ${\pop}_\W(\Gen_\W P) = \W$ whenever $P$ is projective in $\W$.

\begin{proof}
    Since $\J_\W(M) = (M^{\perp_\W}) \cap (\lperpW{(\tau_\W M)})$ is a wide subcategory, it follows from Theorem~\ref{thm:asai_pfeifer_1}(1) that $[\Gen_\W M\lperpW{(\tau_\W M)}]$ is a nuclear interval. By definition, this means $$\lperpW{(\tau_\W M)} = \Gen_\W M \join \left\{j \ \middle| \ \kappa_\W(j) \in \CMR_\W(\Gen_\W M) \text{ and } (\Gen_\W M) \join j \leq \lperpW{(\tau_\W M)}\right\}.$$
    To prove the result, it therefore suffices to show that $(\Gen_\W M) \join j \leq \lperpW{(\tau_\W M)}$ for every $\kappa_\W(j) \in \CMR_\W(\Gen_\W M)$.
    
    Now recall from Theorems~\ref{thm:min_extending} and~\ref{thm:BTZ_A} that the elements of $\CMR_\W(\Gen_\W M)$ are precisely the torsion classes $\W \cap \lperp{X}$ for $X$ a minimal extending module of $\Gen_\W M$. Let $X$ be such a minimal extending module.
    Now, every proper factor of $X$ which lies in $\W$ must lie in $\Gen_\W M \subseteq  \lperpW{(\tau_\W M)}$.
    By Proposition~\ref{prop:AStau}, this means that $\Hom_\Lambda(X,\tau_\W M) \neq 0$ if and only if $\Ext^1_\Lambda(M,X) = 0$.
    Assume for the sake of contradiction that there is a nonsplit exact sequence $X\hookrightarrow E \twoheadrightarrow M$. Then by the definition of a minimal extending module, we have that $E \in \mathcal{T}$. In particular, we have that $M$ is not split projective in $\mathcal{T}$, which contradicts Proposition~\ref{prop:functorially_finite}.
\end{proof}

Combining the results of this section, we have the following.

\begin{corollary}\label{cor:kappa_indec}
    Let $M \in \tr(\W)$ be indecomposable. Then exactly one of the following holds.
    \begin{enumerate}
        \item $M$ is projective in $\W$, $\overline{\kappa_\W}^{d}(\Gen_\W M) = M^{\perp_\W} =  \lperpW{(\overline{\tau_\W} M)}$, and ${\pop}_\W(\Gen_\W M) = \W = \lperpW{(\tau_\W M)}$.
        \item $M$ is not projective in $\W$ and ${\pop}_\W(\Gen_\W M) = \overline{\kappa_\W}^{d}(\Gen_\W M) = \lperpW{(\tau_\W M)}$.
    \end{enumerate}
\end{corollary}

\begin{proof}
    If $M$ is projective in $\W$, the result is the special case of Corollary~\ref{cor:kappa_proj} where $P$ is indecomposable. If $M$ is not projective, then by Proposition~\ref{prop:kappa_tau} we have that ${\pop}_\W(\Gen_\W M) =  \lperpW{(\tau_\W M)}$. Thus it remains only to show that $\lperpW{(\tau_\W M)} = \overline{\kappa_\W}^{d}(\Gen_\W M)$.

    First note that $\Gen_\W M \in \cmrep(\tors(\W))$ and that $|\CJR_\W(\overline{\kappa_\W}^d(\Gen_\W M))| = \rk(\W)-1$ by Theorem~\ref{thm:regular} and Corollary~\ref{cor:extended_kappa_bij}. We also know that $\lperpW{(\tau_\W M)} \in \cjrep(\tors(\W))$ because it is functorially finite, see \cite[Theorem~2.10]{AIR}. Since $M$ is not projective, we have $\lperpW{(\tau_\W M)} \neq \W$, and so $|\CJR_\W( \lperpW{(\tau_\W M}))| \leq \rk(\W) - 1$ by Theorem~\ref{thm:regular}. It then follows from Proposition~\ref{prop:pop_formula}(2) that the sets $\CJR_\W(\lperpW{(\tau_\W M)}) = \CJR_\W({\pop}_\W(\Gen_\W M))$ and $\CJR_\W(\overline{\kappa_\W}^d(\Gen_\W M))$ must coincide, and so $\lperpW{(\tau_\W M)} = \overline{\kappa_\W}^{d}(\Gen_\W M)$.
\end{proof}

\begin{remark}
 Corollary~\ref{cor:kappa_indec}(2) is only true under the assumption that $M$ is indecomposable. For example, let $\Lambda = K(1 \rightarrow 2 \rightarrow 3)$ and let $M = I_2 \oplus S_2$. Then $\Gen M = \lperp{P_1}$. Thus Theorem~\ref{thm:BTZ_A} tells us that $\overline{\kappa}^{d}(\Gen M) = \Gen P_1$, while Proposition~\ref{prop:kappa_tau} tells us that $\pop(\Gen M) = \lperp{\tau M} = \lperp{(P_2 \oplus P_3)} = \Gen(P_1 \oplus S_2)$.
\end{remark}


\subsection{$\kappa^d$-exceptional sequences}\label{sec:kappa_exceptional}

We recall that for $\Lambda$ an arbitrary finite-dimensional algebra, understanding the $\tau$-exceptional sequences for $\mods\Lambda$ requires one to understand which wide subcategories are $\tau$-perpendicular subcategories and which torsion classes are functorially finite. Currently, there is no purely combinatorial criteria which allows one to recover these special subcategories directly from the lattice of torsion classes. As such, we restrict to the case where the lattice of torsion classes $\tors(\mods\Lambda)$ is finite; i.e., we assume that $\Lambda$ is a $\tau$-tilting finite algebra (see Definition-Theorem~\ref{def:tau_tilting_finite}). For such algebras, every torsion class is functorially finite and every wide subcategory is a $\tau$-perpendicular subcategory. In particular, Proposition~\ref{prop:kappa_tau} allows one to compute ${\pop}_\W(\T)$ for any wide subcategory $\W \subseteq \mods\Lambda$ and any torsion class $\T\in \tors(\W)$. This is the motivation for the following (recursive) definition.

\begin{definition}\label{def:kappa_exceptional}
    Let $L$ be a finite semidistributive lattice and let $(j_k,\ldots,j_1)$ be a sequence of (completely) join-irreducible elements of $L$. We recursively call $(j_k,\ldots,j_1)$ a \emph{$\kappa^d$-exceptional sequence} if both of the following hold.
    \begin{enumerate}
        \item $j_i \in \core(j_1) = \jlab[j_1,\pop(j_1)]$ for all $i > 1$.
        \item $(j_k \join j_1,\ldots, j_2 \join j_1)$ is a $\kappa^d$-exceptional sequence in the lattice $\Core(\overline{\kappa}^{-1}(j_1)) = [j_1,\pop(j_1)]$.
    \end{enumerate}
\end{definition}

\begin{remark}
    Recall from Lemma~\ref{prop:intervals} that $\Core(j_1) = [j_1,\pop(j_1)]$ is itself a finite semidistributive lattice, and that supposing $j_i \in \core(j_1)$ is equivalent to supposing that $j_i \join j_1 \in \cji(\Core(j_1))$. The recursive part of the definition can then be seen as ``cutting out'' successively smaller (interval) sublattices from $L$.
\end{remark}

We defer giving examples of $\kappa^d$ to exceptional sequences to Section~\ref{sec:kappa_exceptional_ex}. To conclude the present section, we build towards our final main them. In particular, we need the following, which combines Theorem~\ref{thm:asai_pfeifer_1} with Corollary~\ref{cor:kappa_indec}.

\begin{corollary}\label{cor:reduction}
    Suppose $B \in \W$ is a brick and that $\Filt(\Gen_\W B)$ is functorially finite. Then there is a label-preserving bijection $\Core(\Filt(\Gen_\W B)) \cong \J_\W(\beta_\W^{-1}(B)).$
\end{corollary}

\begin{proof}
    Since $\Filt(\Gen_\W B)$ is functorially finite, it follows from Theorem~\ref{thm:DIJ}(2) that $B$ is in the image of the bijection $\beta_\W$, and thus that $\beta_\W^{-1}(B)$ is an indecomposable $\tau$-rigid module which satisfies $\Gen_\W(\beta_\W^{-1}(B)) = \Filt(\Gen_\W B)$. Moreover, we have $\pop(\Filt(\Gen_\W B)) = \lperpW{(\tau_\W (\beta_\W^{-1}(B)))}$ by Proposition~\ref{prop:kappa_tau}. The result then follows from Theorem~\ref{thm:asai_pfeifer_1}(2)
\end{proof}

We are now prepared to prove our final main theorem.

\begin{theorem}[Theorem~\ref{thm:intro:mainD}]\label{thm:mainD}
    Let $\Lambda$ be a $\tau$-tilting finite algebra, and let $\W \subseteq \mods\Lambda$ be a wide subcategory. Then there is a bijection $\rho$ from the set of $\tau$-exceptional sequences for $\W$ to the set of $\kappa^d$-exceptional sequences for $\tors(\W)$ given as follows. Let $(M_k,\ldots,M_1)\in \tex(\W)$ and let $\psi(M_k,\ldots,M_1) = (\W_k\coveredtau \cdots \coveredtau \W_1 \coveredtau \W_0 = \W)$ be the corresponding saturated top chain in $\tperp(\W)$. (See Theorem~\ref{thm:tf_tau}.) For $1 \leq i \leq k$, denote $\T_i = \Gen_{\W_{i-1}} M_i \in \tors(\W_{i-1})$. Then
    $$\rho(M_k,\ldots,M_1) = (\T_k,\ldots,\T_1).$$
    Moreover, for each $1 \leq i \leq k$, one has $\brlab_{\W_{i-1}}[(\T_i)_*,\T_i] = \beta(M_i)$.
\end{theorem}

\begin{proof}
    Note that since $\Lambda$ is $\tau$-tilting finite, we have that every torsion class is functorially finite (Definition-Theorem~\ref{def:tau_tilting_finite}), every wide subcategory is a $\tau$-perpendicular subcategory (Example~\ref{ex:tau_perp}), and every brick is an sf-brick. In particular, for every $\V \in \wide(\W)$ there are bijections from the set of indecomposable modules in $\tr(\V)$ to $\cji(\tors(\V))$ and $\brick(\V)$ given by $M \mapsto \Gen_\V M$ and $M \mapsto \beta(M)$. (See Theorem~\ref{thm:regular}.)

    We now prove the result by induction on $\rk(\W)$. The result is trivial for $\rk(\W) = 0$, so suppose that $\rk(\W) > 0$ and that the result holds for all $\V$ with $\rk(\V) < \rk(\W)$.  Now by Definition~\ref{def:brick_label}, $\beta(M_1)$ is the (unique) brick label of the cover relation $\W_1 \coveredtau \W$ in $\tperp(\W)$. In particular, $M_1$ bijectively determines both $\T_1$ and $\W_1$. We also have $\Filt(\Gen_\W \beta(M_1)) = \T_1$ and ${\pop}_\W(\T_1) = \lperpW(\tau_\W M_1)$ by Corollary~\ref{cor:kappa_indec}. This in particular means that ${\brlab}_\W[(\T_1)_*,\T_1] = \beta(M_1)$.
    
    Now by Corollary~\ref{cor:reduction}, we have a label-preserving isomorphism $\Core(\T_1) \cong \tors(\W_1)$. Also, by assumption we have that $(M_k,\ldots,M_2) \in \tex(\W_1)$ and that $\rho(M_k,\ldots,M_2) = (\T_k,\ldots,\T_2)$. As $\rk(\W_1) = \rk(\W) - 1$ by Equation~\ref{eqn:jasso}, the induction hypothesis implies that $\rho$ is a bijection for the wide subcategory $\W_1$ and that $\brlab_{\W_{i-1}}[(\T_i)_*,\T_i] = \beta(M_i)$ for $1 < i \leq k$. Since $\W_1$ was bijectively determined by $M_1$, this implies the result.
\end{proof}

\begin{remark}
In particular, Theorem~\ref{thm:mainD} implies that if one identifies every indecomposable $\tau$-rigid module with the corresponding completely join-irreducible torsion class, then the $\tau$-exceptional sequences of $\mods\Lambda$ can be determined from only from the underlying lattice structure of the lattice of torsion classes.
\end{remark}

\subsection{Examples of $\kappa^d$-exceptional sequences}\label{sec:kappa_exceptional_ex}

In this section, we describe the $\kappa^d$-exceptional of the lattices in Figures~\ref{fig:run_ex} and~\ref{fig:weird_clo}(top left). We then observe to what extent these sequences can be used to label the corresponding core label orders and discuss our work in progress towards formalizing these results.

 \begin{example}\label{ex:detailed_ex_1}
     Let $L$ be as in Figure~\ref{fig:run_ex}. Then $L$ has seven maximal $\kappa^d$-exceptional sequences, in the sense that they cannot be extended by adding more terms on the left. They are:
    $$(j_1,j_2,j_4)\qquad\qquad (j_1,j_3,j_2)\qquad\qquad (j_2,j_1,j_4) \qquad\qquad (j_2,j_1,j_4)$$
    $$(j_3,j_1,j_2) \qquad\qquad (j_3,j_2,j_1) \qquad\qquad (j_4,j_3).$$
    We observe the following:
    \begin{enumerate}
        \item The number of maximal $\kappa^d$ exceptional sequences coincides with the number of maximal chains in the upper core label order $(L,\leq_{\clo})$, which is redrawn in Figure~\ref{fig:labeled_clo}
        \item There is a bijection between $\cji(L)$ and the set of cover relations in $(L,\leq_{\clo})$ of the form $u \covered \hat{1}$ which sends $j_i$ to $\overline{\kappa}^d(j_i)$. This allows us to label each such cover relation in the upper core label order. See Figure~\ref{fig:labeled_clo}.
        \item The pattern in (2) is inherited by each sublattice of the form $\Core(j_i)$. Moreover, passing to such a sublattice respects the upper core label order in the following way. Denote $L' = [j_i,\pop(j_i)]$. Then by identifying $\cji(L')$ with $\core(j_i)$ as in Lemma~\ref{prop:intervals}, we can then identify the upper core label order $(L',\leq_{\clo})$ with the interval $[\hat{0},\overline{\kappa}^d(j_i)] \subseteq (L,\leq_{\clo})$.
        \item Observation (2) allows us to label every cover relation in $(L,\leq_{\clo})$, see Figure~\ref{fig:labeled_clo}. Moreover, the lattice $L$ is extremal (there is a maximal chain which has all of the completely join-irreducible elements as labels). From the maximal chain on the left side of $L$, we define a total order $\preceq$ on $\cji(L)$ by $j_1 \preceq j_2 \preceq j_4 \preceq j_3$. It is straightforward to verify that this total order makes the labeling of $(L,\leq_{\clo})$ into an EL-labeling.
        \item There does not seem to be a natural way to turn the $\kappa^d$-exceptional sequences into an EL-labeling of the lower core label order $(L,\leq_{\clod})$, which coincides with the $\kappa$-order as shown in Figure~\ref{fig:ex_clo}. Indeed, all seven maximal chains of $(L,\leq_{\clod})$ contain three cover relations, while only six of the maximal $\kappa^d$-exceptional sequences contain three elements.
    \end{enumerate}
 \end{example}

  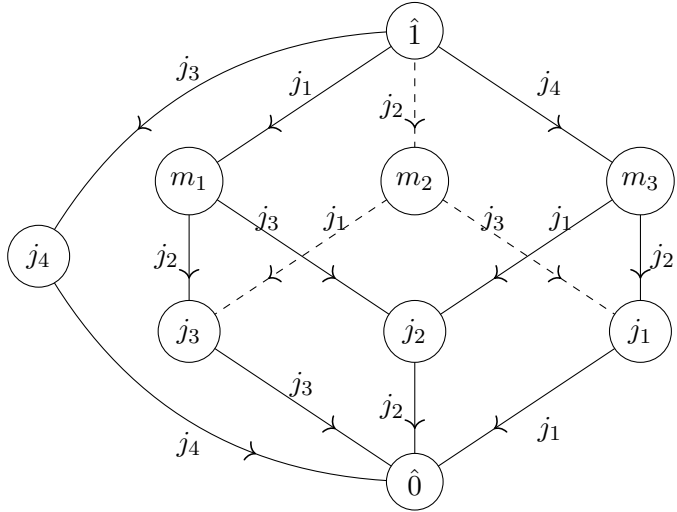
\begin{figure}
   \begin{tikzpicture}
        \begin{scope}[decoration={
	markings,
	mark=at position 0.65 with {\arrow[scale=1.5]{>}}}
	]
        \draw[postaction=decorate] (0,6) -- (-3,4) node[midway, above]{$j_1$};
        \draw[postaction=decorate,dashed](0,6) -- (0,4) node[midway,left]{$j_2$};
        \draw[postaction=decorate](0,6) -- (3,4) node[midway,above right]{$j_4$};
        \draw[postaction=decorate](-3,4) -- (0,2) node[near start, right]{$j_3$};
        \draw[postaction=decorate,dashed](0,4) -- (-3,2) node [near start, left]{$j_1$};
        \draw[postaction=decorate](3,4) -- (0,2) node [near start,left]{$j_1$};
        \draw[postaction=decorate,dashed](0,4) -- (3,2) node [near start,right]{$j_3$};
        \draw[postaction=decorate](-3,4) -- (-3,2) node [midway,left]{$j_2$};
        \draw[postaction=decorate](3,4) -- (3,2) node [midway, right]{$j_2$};
        \draw[postaction=decorate](-3,2) -- (0,0) node[midway,above]{$j_3$};
        \draw[postaction=decorate](0,2) -- (0,0) node [midway,left]{$j_2$};
        \draw[postaction=decorate](3,2) -- (0,0) node[midway,below right]{$j_1$};
        \draw[postaction=decorate,bend right,smooth](0,6) to (-5,3);
        \draw[postaction=decorate,bend right,smooth](-5,3) to (0,0);
        \node at (-3,0.5) {$j_4$};
        \node at (-3,5.5) {$j_3$};
        \end{scope}
		\node[draw,circle,fill=white] at (0,0) {$\hat{0}$};
		\node [draw,circle,fill=white] at (-5,3) {$j_4$};
	    \node[draw,circle,fill=white] at (-3,2) {$j_3$};
	    \node[draw,circle,fill=white] at (0,2) {$j_2$};
	    \node[draw,circle,fill=white] at (3,2) {$j_1$};
	    \node[draw,circle,fill=white] at (3,4) {$m_3$};
	    \node[draw,circle,fill=white] at (0,4) {$m_2$};
	    \node[draw,circle,fill=white] at (-3,4) {$m_1$};
	    \node[draw,circle,fill=white] at (0,6) {$\hat{1}$};
	\end{tikzpicture}
    \caption{A labeling of the upper core label order of the lattice in Figure~\ref{fig:run_ex} by $\kappa^d$-exceptional sequences}\label{fig:labeled_clo}
    \end{figure}
 
 \begin{example}\label{ex:detailed_ex_2}
    Let $L$ be as in the top left of Figure~\ref{fig:weird_clo}. Then $L$ has ten maximal $\kappa^d$-exceptional sequences, in the sense that they cannot be extended by adding more terms on the left. They are:
    $$(j_5,j_2,j_1) \qquad\qquad (j_3,j_5,j_1) \qquad\qquad (j_2,j_3,j_1) \qquad\qquad (j_5,j_1,j_2) \qquad\qquad (j_3,j_5,j_2)$$
    $$(j_1,j_3,j_2) \qquad\qquad (j_4,j_3) \qquad\qquad (j_2,j_1,j_4) \qquad\qquad (j_1,j_2,j_4) \qquad\qquad (j_3,j_5).$$
    We observe the following:
    \begin{enumerate}
        \item There are nine maximal chains in $(L,\leq_{\clo})$ (redrawn in Figure~\ref{fig:labeled_clo_2}), but there are ten maximal $\kappa^d$-exceptional sequences. On the other hand, one could argue that $(j_3,j_5)$ should not be counted as maximal since it can be extended to the right to obtain $(j_3,j_5,j_2)$.
        \item The element $\hat{1}$ covers four elements in $(L,\leq_{\clo})$. These cover relations are of the form $\overline{\kappa}(j_i) \covered \hat{1}$ for $i \neq 5$, see Figure~\ref{fig:labeled_clo_2}. In particular, if we ignore $(j_3,j_5)$ as suggested by (1), then the cover relations are in bijection with those join-irreducible elements which can appear on the left of a maximal $\kappa^d$-exceptional sequence.
        \item As in Example~\ref{ex:detailed_ex_1}(3), for any $i \in \{1,\ldots,5\}$ we can identify the upper core label order of $[j_i,\pop(j_i)]$ with the interval $[\hat{0},\overline{\kappa^d}(j_i)] \subseteq (L,\leq_{\clo})$. We can thus label $(L,\leq_{\clo})$ with the maximal $\kappa^d$-exceptional sequences other than $(j_3,j_5)$, as shown in Figure~\ref{fig:labeled_clo_2}. \item Both the cover relations $[j_3,m_1]$ and $[j_3,m_2]$ in Figure~\ref{fig:labeled_clo_2} are labeled with $j_5$. This seems to indicate that the maximal chains of $(L,\leq_{\clo})$ cannot be build from bottom-up using $\kappa$-exceptional sequences.
        \item The lattice $L$ is not extremal. However, the total order $\preceq$ on $\cji(L)$ given by $j_1 \preceq j_5 \preceq j_2 \preceq j_4 \preceq j_3$ still makes the labeling in Figure~\ref{fig:labeled_clo_2} into an EL-labeling.
    \end{enumerate}
 \end{example}
 
   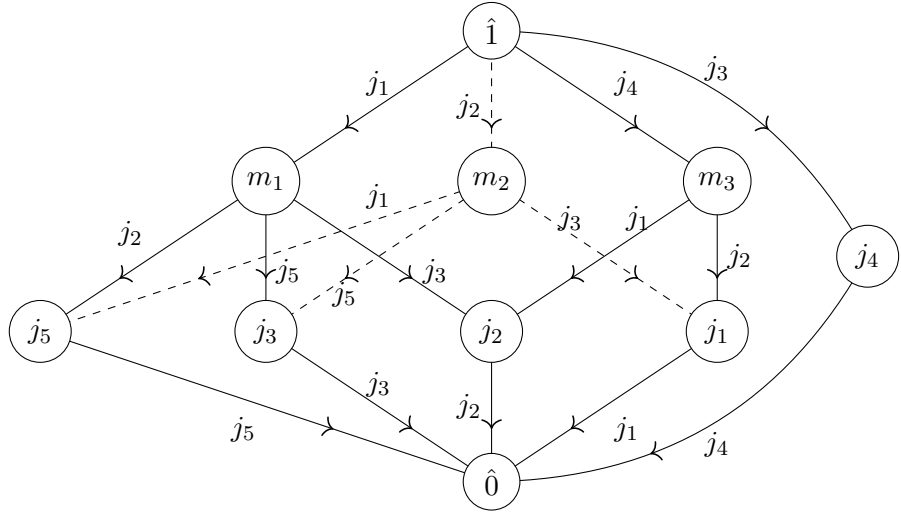
\begin{figure}
   \begin{tikzpicture}
        \begin{scope}[decoration={
	markings,
	mark=at position 0.65 with {\arrow[scale=1.5]{>}}}
	]
        \draw[postaction=decorate] (0,6) -- (-3,4) node[midway, above]{$j_1$};
        \draw[postaction=decorate,dashed](0,6) -- (0,4) node[midway,left]{$j_2$};
        \draw[postaction=decorate](0,6) -- (3,4) node[midway,above right]{$j_4$};
        \draw[postaction=decorate](-3,4) -- (0,2) node[near end, above]{$j_3$};
        \draw[postaction=decorate,dashed](0,4) -- (-3,2) node [near end, right]{$j_5$};
        \draw[postaction=decorate](3,4) -- (0,2) node [near start,left]{$j_1$};
        \draw[postaction=decorate,dashed](0,4) -- (3,2) node [near start,right]{$j_3$};
        \draw[postaction=decorate](-3,4) -- (-3,2) node [near end,above right]{$j_5$};
        \draw[postaction=decorate](3,4) -- (3,2) node [midway, right]{$j_2$};
        \draw[postaction=decorate](-3,2) -- (0,0) node[midway,above]{$j_3$};
        \draw[postaction=decorate](0,2) -- (0,0) node [midway,left]{$j_2$};
        \draw[postaction=decorate](3,2) -- (0,0) node[midway,below right]{$j_1$};
        \draw[postaction=decorate,bend left,smooth](0,6) to (5,3);
        \draw[postaction=decorate,bend left,smooth](5,3) to (0,0);
        \draw[postaction=decorate](-3,4)--(-6,2) node[midway,above left]{$j_2$};
        \draw[postaction=decorate,dashed](0,4)--(-6,2) node[near start,above]{$j_1$};
        \draw[postaction=decorate](-6,2)--(0,0) node[midway,below left]{$j_5$};
        \node at (3,0.5) {$j_4$};
        \node at (3,5.5) {$j_3$};
        \end{scope}
		\node[draw,circle,fill=white] at (0,0) {$\hat{0}$};
		\node [draw,circle,fill=white] at (-6,2) {$j_5$};
	    \node[draw,circle,fill=white] at (-3,2) {$j_3$};
	    \node[draw,circle,fill=white] at (0,2) {$j_2$};
	    \node[draw,circle,fill=white] at (3,2) {$j_1$};
	    \node[draw,circle,fill=white] at (3,4) {$m_3$};
	    \node[draw,circle,fill=white] at (0,4) {$m_2$};
	    \node[draw,circle,fill=white] at (-3,4) {$m_1$};
	    \node[draw,circle,fill=white] at (5,3) {$j_4$};
	    \node[draw,circle,fill=white] at (0,6) {$\hat{1}$};
	\end{tikzpicture}
    \caption{A labeling of the upper core label order of the lattice in Figure~\ref{fig:run_ex} by $\kappa^d$-exceptional sequences}\label{fig:labeled_clo_2}
    \end{figure}
    
\section{Discussion and future work}

Theorem~\ref{thm:mainD} shows that, at least for $\tau$-tilting finite algebras, $\tau$-exceptional sequences can be seen as a special case of a more general combinatorial construction. A natural question is whether other constructions from $\tau$-tilting yield similar combinatorial generalizations. Towards this end, in a future paper we will discuss how to extending Definition~\ref{def:kappa_exceptional} to obtain \emph{signed} $\kappa^d$-exceptional sequences. Alternatively, this amounts to describing when a completely join-irreducible element in a ws-csd lattice is (relatively) ``projective''.

We will also look into extending Theorem~\ref{thm:mainA} in two different ways. On the representation theory side, there exist algebras whose sf-bricks do not admit an rho order, but whose $\tau$-excpetional sequences can still be used to describe an EL-labeling of the poset $\tperp(\mods\Lambda)$. A simple example is the preprojective algebra of type $A_2$, see Figure~\ref{fig:preproj}. On the combinatorics side, let $L$ be a finite semidistributive lattice. In both Examples~\ref{ex:detailed_ex_1} and~\ref{ex:detailed_ex_2} we showed that the maximal chains of $(L,\leq_{\clo})$ can be labeled by $\kappa^d$-exceptional sequences. Moreover, in both cases, there is a total order $\preceq$ on $\cji(L)$ which makes this labeling into an EL-labeling. In our future work, we will examine when these phenomenon hold more generally. In particular, we will discuss how $\kappa^d$-exceptional sequences are related to pulling triangulations of \emph{Newton polytopes} (also known as \emph{Harder-Narasimhan polytopes} or \emph{submodule polytopes}, see e.g. \cite{BKT,fei1}), as well as other convex polytopes coming from hyperplane arrangements.

   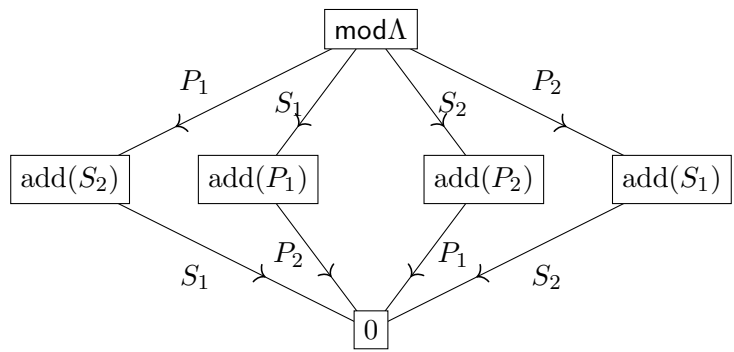
\begin{figure}
   \begin{tikzpicture}
        \begin{scope}[decoration={
	markings,
	mark=at position 0.65 with {\arrow[scale=1.5]{>}}}
	]
        \draw[postaction=decorate] (0,6) -- (-4,4) node [midway,above left] {$P_1$};
        \draw[postaction=decorate](0,6) -- (-1.5,4) node [midway,left] {$S_1$};
        \draw[postaction=decorate](0,6) -- (1.5,4) node [midway,right] {$S_2$};
        \draw[postaction=decorate](0,6) -- (4,4) node [midway,above right] {$P_2$};
        \draw[postaction=decorate](-4,4) -- (0,2) node [midway,below left] {$S_1$};
        
        \draw[postaction=decorate](-1.5,4) -- (0,2) node [midway,left] {$P_2$};
        
        \draw[postaction=decorate](1.5,4) -- (0,2) node [midway, right] {$P_1$};
        
        \draw[postaction=decorate](4,4) -- (0,2) node [midway,below right] {$S_2$};
        \end{scope}
		\node[draw,fill=white] at (0,2) {$0$};
	    \node[draw,fill=white] at (-4,4) {$\mathrm{add}(S_2)$};
	    \node[draw,fill=white] at (-1.5,4) {$\mathrm{add}(P_1)$};
	    \node[draw,fill=white] at (1.5,4) {$\mathrm{add}(P_2)$};
	    \node[draw,fill=white] at (4,4) {$\mathrm{add}(S_1)$};
	    \node[draw,fill=white] at (0,6) {$\mods\Lambda$};
	\end{tikzpicture}
    \caption{The lattice of wide subcategories of the preprojective algebra of type $A_2$ and its labeling by $\tau$-exceptional sequences. As a quotient of a path algebra, we have $\Lambda = (1 \leftrightarrows 2)/R^2$, where $R$ is the ideal generated by the arrows. The total order $P_1 \preceq S_2 \preceq P_2 \preceq S_1$ makes this into an EL-labeling. }\label{fig:preproj}
    \end{figure}

\bibliographystyle{amsplain}
\bibliography{EL_bib.bib}

\end{document}